\documentclass[12pt,reqno]{amsart}
\allowdisplaybreaks
\usepackage{graphicx}
\usepackage{amssymb}
\usepackage{bm}
\usepackage{tikz}
\usepackage{mathtools}
\usepackage{xcolor}

\newtheorem{theorem}             {Theorem}    [section]
\newtheorem{lemma}      [theorem]{Lemma}
\newtheorem{proposition}[theorem]{Proposition}
\newtheorem{corollary}  [theorem]{Corollary}

\theoremstyle{definition}

\newtheorem{example}   [theorem]{Example}
\newtheorem{definition}[theorem]{Definition}

\newtheorem{remark}    [theorem]{Remark}
\newtheorem{notation}  [theorem]{Notation}

\newcommand{\ds}{\displaystyle}
\newcommand{\ab}{\allowbreak}

\newcommand{\thebottomline}{\renewcommand{\thefootnote}{}
  \renewcommand{\footnoterule}{}
  \phantom{M}\footnotetext{\tiny{}\hfill
    \textit{\noindent\romannumeral\day.%
\romannumeral\month.\romannumeral\year}}}

\newcommand{\case}[1]{\par\smallskip\noindent%
\textit{Case $($#1$\kern0.1em)$\ }}
\let\oldtocsection=\tocsection
\let\oldtocsubsection=\tocsubsection
\let\oldtocsubsubsection=\tocsubsubsection
\renewcommand{\tocsection}[2]{\hspace{0em}\oldtocsection{#1}{#2}}
\renewcommand{\tocsubsection}[2]{\hspace{1em}\oldtocsubsection{#1}{#2}}
\renewcommand{\tocsubsubsection}[2]{\hspace{2em}\oldtocsubsubsection{#1}{#2}}

\newcommand{\cA}{\mathcal{A}}
\newcommand{\cP}{\mathcal{P}}
\newcommand{\cPS}{\mathcal{PS}}
\newcommand{\cU}{\mathcal{U}}
\newcommand{\cV}{\mathcal{V}}
\newcommand{\cW}{\mathcal{W}}
\newcommand{\cX}{\mathcal{X}}

\newcommand{\bC}{\mathbb{C}}

\newcommand{\bR}{\mathbb{R}}

\newcommand{\E}{\text{E}}
\newcommand{\tr}{\text{tr}}
\newcommand{\Tr}{\mathop{\text{Tr}}}
\newcommand{\cov}{\mathrm{cov}}

\let\phi=\varphi
\newcommand{\id}{\mathit{id}}

\newcommand{\Kr}{\mathit{Kr}}

\newcommand{\sep}{\textrm{\ sep.\ }}
\newcommand{\all}{\textit{\,all}}
\newcommand{\nO}{\textrm{\O}}
\newcommand{\NC}{\textrm{NC}}
\newcommand{\snckalt} [2]{S_{\NC}^{\textit{k-alt}}(#1, #2)}
\newcommand{\snckea}  [2]{S_{\NC}^{\textit{k-e-a}}(#1, #2)}
\newcommand{\sncjea}  [2]{S_{\NC}^{\textit{j-e-a}}(#1, #2)}

\newcommand{\sncaplus}[2]{S_{\NC}^{\textit{\all}+}(#1, #2)}

\makeatletter\long\def\@makefntext#1{\@setpar{\@@par\@tempdima
\hsize \advance\@tempdima-10pt\parshape \@ne
10pt\@tempdima}\par \parindent 1em\noindent \hbox to
\z@{\hss$\m@th^{\@thefnmark}$}#1} 
\newcounter{int}
\newcommand{\listcomments}{%
\smallskip\hrule\smallskip\noindent 
\@whilenum\value{int}<\thejmpnumber\do
{\tiny Comm. \theint{} is on page \pageref{jmp\theint}\stepcounter{int}, }
\smallskip\hrule\smallskip}
\def\@captionfont{\small}
\makeatother
\newcounter{jmpnumber}

\renewcommand{\thefootnote}{(\arabic{footnote})}

\begin{document}
\title[second order even and $R$-diagonal elements]
       {Second Order Cumulants: second order even elements 
        and ${\bm R}$-diagonal elements}

\author[Arizmendi]{Octavio Arizmendi}
\address{Centro de Investigaci{\'o}n en Matem{\'a}ticas, Guanajuato, Mexico}
\email{octavius@cimat.mx}
\thanks{The first author was supported by CONACYT Grant 222668.}

\author[mingo]{James A. Mingo} \address{Department
  of Mathematics and Statistics, Queen's University, Jeffery
  Hall, Kingston, Ontario, K7L 3N6, Canada}

\email{mingo@mast.queensu.ca} 

\thanks{The second author was  supported by a Discovery Grant from
  the Natural Sciences and Engineering Research Council of
  Canada}

\begin{abstract}
We introduce $R$-diagonal and even operators of second
order. We give a formula for the second order free cumulants
of the square $x^2$ of a second order even element in terms
of the second order free cumulants of $x$. Similar formulas
are proved for the second order free cumulants of $aa^*$,
when $a$ is a second order $R$-diagonal operator.  We also
show that if $r$ is second order $R$-diagonal and $b$ is
second order free from $r$, then $rb$ is also second order
$R$-diagonal. We present a large number of examples, 
in particular, the limit distribution of products of 
Ginibre matrices. We prove the conjectured formula of 
Dartois and Forrester for the fluctuations moments of 
the product of two independent complex Wishart 
matrices and generalize it to any number of factors. 
\end{abstract}

\maketitle


\section{Introduction}\label{sec:introduction}

Roughly forty years ago, Voiculescu devised a new kind of
independence, called \textit{free independence} \cite{vdn}, for
non-commuting random variables, which also gives a precise
meaning to being in `general position', but without the
assumption that the random variables commute. This has been
particularly fruitful in dealing with matrix-valued random
variables. 

The context for free probability theory is a non-commutative
probability space $(\cA, \phi)$. Here $\cA$ is a unital
algebra over $\bC$, and $\phi: \cA \rightarrow \bC$ is a
linear functional with $\phi(1) = 1$.  Elements of $\cA$ are
our random variables and $\phi$ is our expectation. If $a_1,
\dots, a_s \in \cA$ then $\{\phi(a_{i_1} \cdots a_{i_n}) \mid i_1,
\dots, i_n \in \{1, \dots, s\}\}$ are the mixed moments of
$\{a_1, \dots, a_s\}$. If $a_1, \dots,\ab a_s$ are freely
independent, then the mixed moments of $\{ a_1, \dots,\ab a_s\}$
are determined, according to a \textit{universal rule}, by
the individual moments of each $a_k$, $1 \leq k \leq n$. 

In this article, we shall further assume that $\cA$ is an
involutive algebra, i.e. there is a conjugate linear map $a
\mapsto a^*$ such that $(ab)^* = b^* a^*$. Furthermore, we
shall assume that $\phi(a^*a) \geq 0$ for $a \in \cA$ and
that $\phi$ is a trace, although many results remain valid
in greater generality. Such a pair $(\cA, \phi)$ is
frequently called a \textit{non-commutative $*$-probability
  space}. If $\cA$ is a C$^*$-algebra and $a \in \cA$ is
self-adjoint then $a$ has a spectral measure, $\mu_a$,
(relative to $\phi$) given by $\phi(p(a)) = \int_\bR p(t) \,
d \mu_a(t)$ for polynomials $p$. If $\cA$ is a C$^*$-algebra
and $a_1$ and $a_2$ are freely independent and self-adjoint
then the universal rule determines the spectral measure of
$a_1 + a_2$ from that of $a_1$ and $a_2$ (see \cite[Lecture
  12]{ns2}).

The importance of free probability theory in random matrix
theory is that it allows to understand the eigenvalue
distribution of many random matrix ensembles that are
constructed from others.

By a random matrix ensemble we mean a sequence of random
matrices $\{ X_N \}_N$ where $X_N$ is a $N \times N$ random
matrix. The eigenvalues of $X_N$ are random. If $X_N =
X_N^*$, the eigenvalues are real and random. If $X_N =
X_N^*$ and for all $z$ in the complex upper half plane
$\bC^+$, the limit $\lim_N \E( \frac{1}{N} \Tr((z -
X_N)^{-1}))$ exists, we say that the ensemble has a
\textit{limit eigenvalue distribution}.

To be concrete, if $X_N$ and $Y_N$ both have limit
eigenvalue distributions, we may ask about the limit
eigenvalue distribution $X_N + Y_N$. It turns out that
Voiculescu's free independence rule allows us to calculate
the limit eigenvalue distribution of $X_N+Y_N$. There are
many results, starting with \cite{v}, giving sufficient
conditions for two ensembles to be asymptotically free (see
\cite[Chapter 4]{ms2}).

One way to understand such universal rule is via the free
cumulants.  In classical probability, the characteristic
function of a measure and its logarithm play a prominent
role. In free probability, the logarithm of the
characteristic function gets replaced by Voiculescu's
$R$-transform given by $$R(z) = G^{\langle -1 \rangle}(z) -
z^{-1}$$ where $G(z) = \E((z - X)^{-1})$ is the
\textit{Cauchy transform of} $X$, $z$ is in the complex
upper half plane $\bC^+$ and the inverse
$G^{\langle-1\rangle}$ is defined on a suitable domain in
$\bC$. If $R$ has a power series expansion $\sum_{n \geq 1}
\kappa_n z^{n-1}$, which is always the case with a bounded
random variable, then the coefficients, $\{ \kappa_n \}_{n
  \geq 1}$, of $R$ are called the \textit{free cumulants} of
$X$.

After one understands the limit eigenvalue distribution of
an ensemble $\{X_N\}_N$, the next object of study is the
covariance of traces of resolvents
\[
\lim_N \cov(\Tr((z - X_N)^{-1}), \Tr((w - X_N)^{-1})),
\]
which we call the \textit{second order Cauchy transform} of the
limit distribution. In this paper we shall be concerned with
the combinatorial properties of this function; see \cite{dm}
for the properties of this analytic function.

The limiting covariance of the random variables $\Tr((z -
X_N)^{-1})$ and $\Tr((w - X_N)^{-1})$ is referred to in the
statistical literature as the linear spectral statistics of
the ensemble, see \cite{bs,f}.

In \cite{ms}, an extension of free independence was
initiated, called second order freeness. The purpose was to
find the analogue of universal rule above for linear
spectral statistics. In \cite{cmss}, the second order
analogue of the $R$-transform and second order cumulants
were presented. It is these second order cumulants that are
the subject of this article. 

In particular, we wish to apply them to the study of
$R$-diagonal operators and even operators. $R$-diagonal
operators are certain non-normal operators in a
$*$-probability space. It was the study of these operators
that led to the solution of the invariant subspace problem
for operators in II$_1$ factors \cite{hs}.  In random matrix
theory, they appear in the famous single ring theorem
\cite{gkz} which is a generalization of the circular
law for Ginibre matrices. This distribution is the limiting
distribution of random matrix ensembles of the form
$U_nT_nV_n$, where $T_n$ is positive, $U_n$ and $V_n$ are
Haar unitaries with independent entries.

Even operators are self-adjoint operators for which all odd
moments are zero.  There is a close relation between even
and $R$-diagonal operators.  If $a$ is $R$-diagonal, then its
distribution can be recovered from its Hermitization
$\begin{pmatrix}0&a\\ a^*&0\end{pmatrix}$, which is an even
  operator.

 The contribution of this article is to extend
  the idea of $R$-diagonality and evenness to the case of
  second order freeness. In particular, for a second order
  $R$-diagonal operator we have, letting,
  $$\beta^{(a)}_n:=\kappa_{2n}(a, a^*, \dots,\ab a,\ab a^*)$$ and
\begin{equation*}
\beta_{p,q}^{(a)}:=\kappa_{2p,2q}(a,a^*,\dots,a, a^*)
 \end{equation*}
be the determining series of $a$, we have that
\begin{equation} \label{cumulants of aa}
\kappa_{p,q}(aa^*, \dots, aa^*) =
\sum_{(\cV, \pi) \in \cPS_{NC}(p,q)} \beta_{(\cV, \pi)}
\end{equation}
where $\cPS_{NC}(p,q)$ is the set of non-crossing
partitioned permutations on a $(p, q)$-annulus introduced in
Section \ref{section:partitioned_permutations}.

There is a similar construction of the determining series of
a second order even operator so that the same result holds
(see equation (\ref{determining 1})). 

In recent work of Borot, Charbonnier, Garcia-Failde, Leid,
and Shad\-rin, \cite{bcgls}, the theory of higher order
freeness has been connected to topological recursion. In
particular, a third order version of equation
(\ref{eq:second_order_r_transform}) is now known. As it
only involves products of the terms in equation
(\ref{eq:second_order_r_transform}), we expect that much
of our formalism extends to higher orders. 

Formula \eqref{cumulants of aa} and its counterpart for even
elements imply a functional equation between the series of
$*$-cumulants of $a$ and the cumulant of $aa^*$, and thus is
useful to relate their fluctuation moments. As an
application of our main results, at the end of the paper, in
Section \ref{sec:examples}, we work out the $*$-cumulants of
a number of examples. These include second order Haar
unitaries and products of semicircular and circular
operators. In particular, we highlight three important
applications to random matrix ensembles.

\begin{enumerate}

\item
We prove the conjectured formula of Dartois and Forrester
\cite{df} for the fluctuations moments of the product of two
independent complex Wishart matrices and generalize it to
any number of factors. See Remark
\ref{remark:connection_to_dartois_forrester}.

\item
We obtain the second order $*$-cumulants and second order
Cauchy transform of Wishart ensembles with a given
covariance: i.e. $WAW^*$ where $W$ is a Ginibre Matrix and
$A$ is a deterministic matrix.

\item
We extend the calculation of Dubach and Peled \cite{dp} of
some fluctuation moments of products of Ginibre matrices, to
the case of general $*$-moments.
\end{enumerate} 

\subsection*{Outline of the paper} In Section
\ref{sec:second_order_cumulants}, we will review the basic
notations of second order freeness and the definition and
properties of even and $R$-diagonal operators of first
order.  In Section \ref{sec:statement_of_results}, we will
present the definitions and statements of our main results.
In Section \ref{sec:annular_non-crossing_partitions}, we will
recall the results about non-crossing annular permutations
we need and prove the main technical results needed in
Sections \ref{sec:second_order_r-diagonal} and
\ref{sec:even_elements}. In Section
\ref{sec:second_order_r-diagonal}, we will prove Theorem
\ref{MT2}, our main result on the determining series of
second order $R$-diagonal operators. In Section
\ref{sec:even_elements}, we will prove Theorem \ref{MT3}, our
main result on the determining series of second order even
operators. In Section
\ref{sec:products_free_random_variables}, we extend a result
of Arizmendi and Vargas on cumulants of products to the
second order case.  In Section \ref{sec:examples}, we present
many examples of computations of second order $*$-cumulants
using Theorems \ref{MT2} and \ref{MT3}.

As we shall see in Section \ref{sec:examples}, some
properties of first order freeness have a simple extension
to the second order case while others do not. For example,
contrary to the first order case, the powers of a second
order $R$-diagonal operator are not necessarily second order
$R$-diagonal operators (see Example \ref{exa:haar2}).  Also,
creating freeness of second order  by conjugating by
second order free unitary is more delicate. Recall from
\cite{ns1} that for $a$ and $b$ (first order) free elements
such that $b$ is $R$-diagonal (or even) we also have that
$bab^*$ is free from $a$. Example \ref{exa:uau2} shows that
this is not true at the second order level.  Another
interesting example that shows that second order is
more delicate is the relation between the square $s^2$ of a
semicircular $s$ and the square $cc^*$ of a circular
element. As shown in \cite{ns2}, they both correspond to a
free Poisson or Marchenko-Pastur distribution. Interestingly, $s^2$ and
$cc^*$ have different fluctuation moments. This is shown in
Examples \ref{semicircle} and \ref{examp:square_circular}.


\section{Second order cumulants and second order freeness}
\label{sec:second_order_cumulants}

\subsection{Non-crossing partitions and free cumulants}
\label{subsec:preliminaries_statements}

Before reviewing second order cumulants let us recall some
of the basics of free probability which can be found in
\cite{ns2}. We are going through the details here because
the construction for the second order case (not in
\cite{ns2}) follows a similar pattern. $R$-diagonal
operators are defined in terms of free cumulants, so let us
begin with these. We let $[n] = \{1, 2, 3, \dots, n\}$ and
$\cP(n)$ be the partitions of $[n]$. Recall that $\pi =
\{V_1, \dots, V_k\}$ is a partition of $[n]$ if $V_1 \cup \cdots
\cup V_k = [n]$ and $V_i \cap V_j = \emptyset$. The subsets $V_1,
\dots, V_k$ are the \textit{blocks} of $\pi$. For $\pi,
\sigma \in \cP(n)$, we write $\pi \leq \sigma$ if each block
of $\pi$ is contained in some block of $\sigma$. This
defines a partial order on $\cP(n)$ and with this partial
order $\cP(n)$ becomes a lattice.

A partition $\pi$ of $[n]$ has a \textit{crossing} if we can
find two blocks $V_r$ and $V_s$ (assuming $r \not = s$) and
$i < j < k < l$ with $i, k \in V_r$ and $j, k \in V_s$. A
partition is \textit{non-crossing} if it has no
crossings. Let denote the non-crossing partitions of $[n]$
by $\NC(n)$. We shall regard each partition $\pi$ on $[n]$
as a permutation of $[n]$ as follows. The cycles of the
permutation $\pi$ are the blocks of the partition $\pi$ with
the elements of each block written in increasing
order. Using the metric condition given below, 
Biane \cite{b} identified those permutations  which come
from a non-crossing partition.

For $\pi$ a partition or permutation, let $\#(\pi)$ denote
the number of blocks or cycles accordingly. If $\pi$ is a
partition or permutation of $[n]$, we let $|\pi| = n -
\#(\pi)$. For any permutations $\pi$ and $\sigma$, we always
have the triangle inequality $|\pi \sigma| \leq |\pi| +
|\sigma|$. We let $\gamma_n$ be the permutation of $[n]$
with the single cycle $(1, 2, \dots, n)$. We can now state
Biane's characterization as follows. A permutation $\pi$ 
produces a  non-crossing partition if and
only if $|\gamma_n| = |\pi| + |\pi^{-1}\gamma_n|$; i.e. if
the triangle inequality becomes an equality. Note that we
simultaneously think of $\pi$ as a permutation and as a
partition; we shall do this frequently in this paper.

Now suppose that $(\cA, \phi)$ is a non-commutative
probability space and for each $n$ we have $\kappa_n:
\cA^{\otimes n} \rightarrow \bC$ an $n$-linear
functional. Given a partition $\pi$ of $[n]$ we can define an
$n$-linear map $\kappa_\pi: \cA^{\otimes n} \rightarrow \bC$
by the following rule. For $a_1, \dots, a_n \in \cA$
\[
\kappa_\pi(a_1, \dots, a_n) 
=
\mathop{\prod_{V \in \pi}}_{V = (i_1, \dots, i_l)}
\kappa_l(a_{i_1}, \dots, a_{i_l}).
\]
We are assuming here that $i_1 < i_2 < \cdots < i_l$. 

This construction can be then used to define the sequence
$\{\kappa_n\}_n$ through the moment-cumulant formula:
\begin{equation}\label{eq:moment-cumulant}
\phi(a_1\cdots a_n) = \sum_{\pi \in \NC(n)}
\kappa_\pi(a_1, \dots, a_n).
\end{equation}
This defines the free cumulants inductively and recursively
(see \cite[Lecture 11]{ns1}). Let us recall the definition
of $R$-diagonal elements from \cite[Lecture 15]{ns1}.
Recall that we are assuming that $\cA$ is a $*$-algebra.
For $\epsilon \in \{-1, 1\}$, we let $a^{(\epsilon)} = a$ if
$\epsilon = 1$ and $a^{(\epsilon)} = a^*$ if $\epsilon =
-1$.

\begin{definition}\label{def:r-diagonal}
Let $a \in \cA$. We say that $a$ is
$R$-\textit{diagonal} if for every $n$ and every
$\epsilon_1, \dots, \epsilon_n \in \{-1, 1\}$ we have that
\[
\kappa_n(a^{(\epsilon_1)}, a^{(\epsilon_2)}, \dots,
a^{(\epsilon_n)}) = 0
\]
whenever either $n$ is odd or there is $1 \leq i < n$ with
$\epsilon_i = \epsilon_{i+1}$. We can express this by saying
that all $*$-cumulants are zero except possibly
$\kappa_{2l}(a, a^*, \dots, a, a^*)$ and $\kappa_{2l}(a^*,
a, \dots, a^*, a)$, i.e. we get 0 unless the `$a$'s
alternate with the `$a^*$'s.
\end{definition}

\begin{definition}\label{def:even}
Let $a \in \cA$. We say that a self-adjoint element $a$ is
\textit{even} if all its odd free cumulants vanish,
i.e. $\kappa_{2l-1}(a, \dots , a) = 0$ for $l \geq 1$. By
the moment-cumulant formula, this is equivalent to having all
odd moments vanish.
\end{definition}

\subsection{Non-crossing annular permutations}
\label{subsection:non-crossing_annular_permutations}

The context for second order freeness is a second order
probability space $(\cA, \phi, \phi_2)$. Here $\cA$ is a
unital algebra over $\bC$ and $\phi: \cA \rightarrow \bC$ is
a linear functional with $\phi(1) = 1$ as before. We assume
in addition that $\phi$ is \textit{tracial}, i.e.  $\phi(ab)
= \phi(ba)$. We assume that $\phi_2: \cA \times \cA
\rightarrow \bC$ is  bilinear, symmetric, tracial in each
variable, and is such that $\phi_2(1, a) = \phi_2(a, 1) = 0$
for all $a \in \cA$. In keeping with our previous assumption,
we continue to assume that $\cA$ is a $*$-algebra. For an
element $a \in \cA$, we call the double indexed sequence $\{
\phi_2(a^m, a^n)\}_{m,n}$ the \textit{fluctuation moments}
of $a$. As mentioned above, the point of second order freeness
is that it gives a universal rule for computing fluctuation
moments of sums and products of random variables given their
individual moments and fluctuation moments. When variables
satisfy these universal rules, we say that the elements are
second order free.

\begin{figure}
  \begin{minipage}[c]{0.25\textwidth}
    \includegraphics{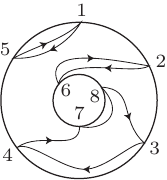}
  \end{minipage}\hfill
  \begin{minipage}[c]{0.75\textwidth}
    \caption{\small A non-crossing annular permutation of the $(5,
  3)$-annulus. From a diagrammatic point of view, we go
  around the cycles in clockwise order with the interior of
  the cycle lying between the two circles in such a way that
  the cycles do not
  intersect.\label{fig:example_non-crossing_annular}}
  \end{minipage}
\end{figure}

\setbox1=\hbox{$\ds\phi_2(a_1 \cdots a_m, b_1 \cdots b_n)
=
\delta_{m,n} \sum_{p = 1}^m
\prod_{i=1}^m \phi(a_i b_{n + p - i})$,}

Recall from \cite{ms} that unital subalgebras $\cA_1, \dots, \cA_s$ of $\cA$, where $(\cA, \phi, \phi_2)$ is a second order probability space, are \textit{second order free} if 
they are free with respect to $\phi$ and  for all $m, n \geq 1$ and all $a_1, \dots, a_m,\ab b_1, \dots, b_n \in \cA$ such that 
\begin{itemize}

\item[$(i)$]
$\phi(a_1) = \cdots = \phi(a_m) = \phi(b_1) = \cdots \phi(n_n) = 0$,

\item[$(ii)$]
$a_i \in \cA_{k_i}$ with $k_1 \not = k_2 \not = \cdots = k_m \not = k_1$,

\item[$(iii)$]
$b_j \in \cA_{l_i}$ with $l_1 \not = l_2 \not = \cdots = l_n \not = l_1$,
\end{itemize}
then for $m, n \geq 2$, we have 
\[
\phi_2(a_1 \cdots a_m, b_1 \cdots b_n)
=
\delta_{m,n} \sum_{p = 1}^m
\prod_{i=1}^m \phi(a_i b_{n + p - i})
\]
where all indices are interpreted modulo $m$.

Here we shall present the cumulant approach from \cite[Chapter 7]{cmss} to second order
freeness and for this we must start with an explanation of
\textit{non-crossing annular permutations} \cite{mn} and 
\cite[Chapter 5]{ms2}.

Let $m, n \geq 1$ be integers. As before for $\pi \in
S_{m+n}$ (the symmetric group on $[m+n]$), we let $\#(\pi)$
be the number of cycles in the cycle decomposition of $\pi$
and $|\pi| = m + n - \#(\pi)$. The following permutation
will play a special role. Let $\gamma_{m,n}$ be the
permutation with the two cycles $(1, \dots, m)(m +1, \dots,
m+n)$. A subset of $[p+q]$ of the form $\{r,
\gamma_{p,q}(r), \gamma^2_{p,q}(r), \dots,
\gamma_{p,q}^{k-1}(r)\}$ is a \textit{cyclic interval}.

As before, we
always have by the triangle inequality that $|\gamma_{m,n}|
\leq |\pi| + |\pi^{-1}\gamma_{m,n}|$. If we have equality,
then $\pi$ can be written as $\pi_1 \times \pi_2$ with
$\pi_1$ a non-crossing partition of $[m]$ and $\pi_2$ a
non-crossing partition of $[m+1, m+n] = \{m+1, \dots,
m+n\}$. This is Biane's characterization. In particular, the
equality means that no cycle of $\pi$ connects the two
cycles of $\gamma_{m,n}$.

If $\pi$ does have a cycle that connects the two cycles of
$\gamma_{m,n}$ then we call $\pi$ a \textit{non-crossing
  annular permutation} if $|\pi| + |\pi^{-1}\gamma_{m,n}| =
|\gamma_{m,n}| + 2$. It was shown in \cite{mn} that this is
equivalent to being able to draw the cycles of $\pi$ in an
annulus without the blocks crossing, see Figure
\ref{fig:example_non-crossing_annular}. We denote the
non-crossing annular permutations by $S_{NC}(m,n)$.

\subsection{Partitioned permutations}
\label{section:partitioned_permutations}
To introduce second order cumulants, we will use a
moment-cumulant formula like the one in Equation
(\ref{eq:moment-cumulant}), except we shall use the partitioned
permutations introduced in \cite{mss}.

\begin{definition}\label{def:partitioned-permutation}
Let $\pi \in S_n$ be a permutation and $\cU \in \cP(n)$ a
partition of $[n]$. We write $\pi \leq \cU$ to mean that
each cycle of $\pi$ is contained in some block of $\cU$. We
say the pair $(\cU, \pi)$ is a \textit{partitioned
  permutation} if $\pi \leq \cU$. We set $|\cU| = n -
\#(\cU)$ where as usual $\#(\cU)$ is the number of blocks of
$\cU$. We let $\cPS(n)$ be the set of partitioned
permutations of $[n]$.
\end{definition} 

When considering partitioned permutations it is sometimes
convenient to have a notation to distinguish between a
permutation $\pi$ and the partition coming from its cycle
decomposition. In these situations we shall write $0_\pi$ to
denote this partition.

Recall that $\cP(n)$ is a lattice under the order $\cU \leq
\cV$ if every block of $\cU$ is contained in some block of
$\cV$. Note that if $\cU$ and $\cV$ are partitions of $[n]$
then we have the triangle inequality $|\cU \vee \cV| \leq
|\cU| + |\cV|$. We can define a binary operation on
$\cPS(n)$ by setting $(\cU, \pi)(\cV, \sigma) = (\cU \vee
\cV, \pi\sigma)$. For this to make sense we need to make the
easy observation that $\pi\sigma \leq \cU \vee \cV$. More
importantly we can put a length function on $\cPS(n)$ by
setting $|(\cU, \pi)| = 2|\cU| - |\pi|$. This length
function satisfies the triangle inequality
\begin{equation}\label{eq:triangle_inequality}
|(\cU \vee \cV, \pi\sigma)| \leq |(\cU, \pi)| + |(\cV, \sigma)|,
\end{equation}
see \cite[\S2.3]{mss}.

Given a partitioned permutation $(\cU, \pi)$ we shall call
the partitioned permutations $(\cV, \sigma)$ and $(\cW,
\tau)$ an \textit{exact factorization} of $(\cU, \pi)$ if we
have $(\cV, \sigma)(\cW, \tau) = (\cU, \pi)$ (i.e. $\tau =
\sigma^{-1}\pi$ and $\cU = \cV \vee \cW$) and $|(\cU, \pi)|
= |(\cV, \sigma)| + |(\cW, \tau)|$, i.e. the inequality in 
(\ref{eq:triangle_inequality}) becomes an equality.

In \cite[Prop.~5.11]{cmss} all the exact
factorizations of $(1_n, \gamma_n)$ and $(1_{m+n},\ab
\gamma_{m,n})$ were found. Here $1_n$ denotes the partition
on $[n]$ with 1 block. All factorizations of $(1_n,
\gamma_n)$ are of the form $(0_\pi, \pi)
(0_{\pi^{-1}\gamma_n},\ab \pi^{-1}\gamma_n)$ with $\pi \in
NC(n)$ and $\pi^{-1}\gamma_n$ the Kreweras complement of
$\pi$.

Three possible factorizations of $(1_{m+n}, \gamma_{m,n})$
were found.  First if $\pi \in S_{NC}(m,n)$ and then
\[
(0_\pi, \pi)(0_{\pi^{-1}\gamma_{m,n}}, \pi^{-1}\gamma_{m,n})
= (1_{m+n}, \gamma_{m,n}) \textrm{\ and }
\]
\[
|(0_\pi, \pi)| + |(0_{\pi^{-1}\gamma_{m,n}},
\pi^{-1}\gamma_{m,n})| = |(1_{m+n}, \gamma_{m,n})|.
\]
We let $\cPS_{\NC}(m,n)'$ to be the set of $(\cU,
\pi)$'s where $\pi = \pi_1 \times \pi_2 \in NC(m) \times
NC([m+1, m+n])$ and $\cU \in \cP(m+n)$ is obtained from
$\pi$ as follows: one block of $\cU$ is the union of a 
cycle of $\pi_1$ and a cycle of $\pi_2$; all other blocks 
of $\cU$ are cycles of $\pi$. For such a $(\cU, \pi)$
we have
\[
(\cU, \pi)(0_{\pi^{-1}\gamma_{m,n}}, \pi^{-1}\gamma_{m,n}) =
(1_{m+n}, \gamma_{m,n}) \textrm{\ and }
\]
\[
|(\cU, \pi)| + |(0_{\pi^{-1}\gamma_{m,n}},
\pi^{-1}\gamma_{m,n})| = |(1_{m+n}, \gamma_{m,n})|.
\]
Finally if $\cV$ is obtained by joining a
cycle of $\pi_1^{-1}\gamma_{m}$ with a cycle of
$\pi_2^{-1}\gamma_n$ and then having all other blocks
of $\cV$ cycles of $\pi^{-1}\gamma_{m,n}$, then
\[
(0_\pi, \pi)(\cV, \pi^{-1}\gamma_{m,n}) = (1_{m+n},
\gamma_{m,n}) \textrm{\ and }
\]
\[
|(0_\pi, \pi)| + |(\cV, \pi^{-1}\gamma_{m,n})| = |(1_{m+n},
\gamma_{m,n})|.
\] 
We write $\cPS_{\NC}(m,n)$ to denote the union $S_{\NC}(m,n)
\cup \cPS_{\NC}(m,n)'$.

\subsection{Second order cumulants}\label{subsec:second_order_cumulants}
Let $(\cA, \phi, \phi_2)$ be a second order probability
space.  We can now describe the second order cumulants in
terms of the moment-cumulant formula. Recall that the first
order cumulants are given by
the moment-cumulant formula, see \cite[Prop.~11.4]{ns2},
\[
\phi(a_1 \cdots a_n) =
\sum_{\pi \in NC(n)} \kappa_\pi(a_1, \dots, a_n).
\]
We use the same method to construct the second order
cumulants:

\begin{eqnarray*}\lefteqn{%
\phi_2(a_1 \cdots a_m, a_{m+1} \cdots a_{m+n})} \\
& = &
\sum_{\pi \in S_{\NC}(m,n)} \kappa_\pi(a_1, \dots, a_{m+n})
+ \kern-1em
\sum_{(\cU, \pi) \in \cPS_{\NC}(m, n)'}
\kappa_{(\cU, \pi)}(a_1, \cdots , a_{m+n}).
\end{eqnarray*}
The $\kappa_\pi$'s in the first term are first order
cumulants, summed over $S_{NC}(m,n)$. The $\kappa_{(\cU,
  \pi)}$'s in the second term are second order
cumulants. Recall that first order cumulants are defined as
multiplicative functions (see \cite[p.~164]{ns2}); we shall
do the same for the second order cumulants. Indeed, if $\pi
= (1, 2, 4)(3)(5, 6)(7) \in NC(4) \times NC(3)$ and $\cU =
\{(1,2,4),(3, 5,6),(7)\}\ab \in \cP(7)$, then
\[
\kappa_{(\cU, \pi)}(a_1, a_2, a_3, a_4, a_5, a_6, a_7)
= 
\kappa_3(a_1, a_2, a_4) \kappa_{1,2}(a_3, a_5, a_6) 
\kappa_1(a_7). 
\]
In \cite[Def.~5.31]{ms2} there is a full explanation, note
that  a simplification of notation is used there in that
$(\cU, \pi)$ is denoted by $\sigma$ and is called a
\textit{marked permutation} in that the  cycles of
$\pi_1$ and $\pi_2$ joined by $\cU$ are called
\textit{marked}. In our example the marked cycles are $(3)$
and $(5,6)$.

If $a$ is an element of a second order non-commutative
probability space $(\cA, \phi, \phi_2)$, we call the set $\{
\phi(a^n)\}_{n \geq 1}$ the \textit{moment sequence} of $a$ and the
set $\{\phi_2(a^m,a^n)\}_{m,n \geq 1}$ the
\textit{fluctuation moment sequence} of $a$.

The moment-cumulant formulas may be stated in terms of the
formal power series, (see \cite[Theorem 39]{ms2}). Let us
define the cumulant series
\[
R(z)=\frac{1}{z}\sum_{n\geq1}\kappa^a_n z^{n}, \qquad
R(z,w)=\frac{1}{zw}\sum_{p,q\geq1}\kappa^a_{p,q} z^{p}w^{q}
\]
where $\kappa^a_n=\kappa_n(a,\dots,a)$ and
$\kappa^a_{p,q}=\kappa_{p,q}(a,\dots,a)$.  Let us also
denote the moment series,
\[
G(z)=\frac{1}{z}\sum_{n\geq0}\phi(a^n) z^{-n}, \qquad
G(z,w)=\frac{1}{zw}\sum_{n,m\geq1}\phi(a^p,a^q)
z^{-p}w^{-q}.
\]
Then we have the relations
\[
\frac{1}{G(z)}+ R(G(z))=z
\]
and 
\begin{multline}\label{eq:second_order_r_transform}
G(z,w)=G'(z)G'(w)R(G(z),G(w))\\+ \frac{\partial^2}{\partial
  z\partial w} \log \left(\frac{G(z)-G(w)}{z-w}\right).
\end{multline}
Note that $\frac{\partial^2}{\partial z\partial w} \log
\big(\frac{G(z)-G(w)}{z-w}\big) = \frac{G'(z) G'(w)} {(G(z)
  - G(w))^2} - \frac{1}{(z - w)^2}$, and hence we can write
(\ref{eq:second_order_r_transform}) without the
logarithm. In addition, by letting $\zeta = G(z)$ and
$\omega = G(w)$ we can write
(\ref{eq:second_order_r_transform}) as a transformation of
differential forms\footnote{We are grateful for Ga\"etan
Borot for bringing this to our attention.}
\[
[(z - w)^{-2} + G(z, w)]\, dzdw = [(\zeta - \omega)^{-2}
+ R(\zeta, \omega)]\, d\zeta d\omega.
\]
In addition, if $\widetilde{G} = \phi \circ G$ with $\phi$ a
M\"obius function (fractional linear transformation) then
$\frac{G'(z) G'(w)} {(G(z) - G(w))^2} =
\frac{\widetilde{G}'(z) \widetilde{G}'(w)}
     {(\widetilde{G}(z) - \widetilde{G}(w))^2} $. So if we
     let $F = 1/G$ then we can expand $F(z)$ as a formal
     power series in $\frac{1}{z}$ starting with $\phi(a)$
     (see \cite[Lemma 3.20]{ms2}) and then we can expand
     $\frac{F(z) - F(w)}{z - w}$ as formal power series
     starting with $1$ and hence $\log \frac{F(z) - F(w)}{z
       - w}$ also makes sense as a formal power series. As a
     function of two variables, its significance is that the
     analyticity of $\log \frac{F(z) - F(w)}{z - w}$ is
     equivalent to the univalence of $F$. See \S1 of
     \cite{hws} for further discussion.

\subsection{Second order freeness}
Recall from the introduction that second order freeness
gives a universal rule for computing the moments and
fluctuation moments of $a_1 + a_2$ provided one knows the
moments and fluctuation moments of $a_1$ and $a_2$
separately and $a_1$ and $a_2$ are second order free. The
rule is given in \cite[Def.~5.24]{ms2}. For deciding when
two or more elements are second order free it is very
convenient to use the rule of vanishing of mixed cumulants,
\cite[Thm.~5.34]{ms2}. Given subalgebras $\cA_1, \dots,
\cA_s$ of a second order non-commutative probability space
$(\cA, \phi, \phi_2)$ and elements $a_1, \dots, a_n$ with
$a_i \in \cA_{j_i}$, we say that the cumulant $\kappa_n(a_1,
\dots, a_n)$ is \textit{mixed} if there are $i_1$ and $i_2$
in $[n]$ such that $j_{i_1} \not= j_{i_2}$; the same applies for
second order cumulants $\kappa_{m,n}(a_1, \dots,
a_{m+n})$. We say that \textit{mixed cumulants vanish} if
$\kappa_n(a_1, \dots, a_n) = 0$ and $\kappa_{m,n}(a_1,\ab
\dots, \ab a_{m+n})\ab = 0$ whenever the cumulant is
mixed. The formulation of second order freeness we shall use
is that subalgebras $\cA_1, \dots, \cA_s$ are second order
free if all mixed cumulants vanish.

\subsection{Cumulants with products as entries}

A crucial tool in this paper will be the formula for writing
a cumulant with products as entries as a sum of cumulants with
all entries single operators \cite{ks}. For example for first order
cumulants we have
\begin{eqnarray*}\lefteqn{%
\kappa_2(a_1a_2, a_3a_4) =
\kappa_4(a_1, a_2, a_3, a_4) +
\kappa_1(a_1) \kappa_3(a_2, a_3, a_4)} \\
&+& 
\kappa_3(a_1, a_3, a_4) \kappa_1(a_2) +
\kappa_3(a_1, a_2, a_4) \kappa_1(a_3) +
\kappa_3(a_1, a_2, a_3) \kappa_1(a_4) \\
&+&
\kappa_2(a_1, a_4) \kappa_2(a_2, a_3) +
\kappa_2(a_1, a_4) \kappa_1(a_2) \kappa_1(a_3) \\
&+&
\kappa_1(a_1) \kappa_2(a_2, a_3) \kappa_1(a_4) +
\kappa_2(a_1, a_3) \kappa_1(a_2) \kappa_1(a_4) \\
&+&
\kappa_1(a_1) \kappa_2(a_2, a_3) \kappa_1(a_4).
\end{eqnarray*}
Note that on the right hand side of the equation above each
entry of a cumulant is a single `$a$'. The general result is
given in Theorem \ref{theorem:ks} below.

\begin{theorem}[\cite{ns2} Thm. 11.12]\label{theorem:ks}
Let $n_1, \dots, n_r$ be positive integers and $n = n_1 +
\cdots + n_r$. Let $a_1 \dots , a_n \in (\cA, \phi)$. Then
\[
\kappa_r(a_1\cdots a_{n_1}, \dots, a_{n_1 + \cdots +
  n_{r-1}+1} \cdots a_{n_1 + \cdots + n_r}) = \sum_{\pi \in
  \NC(n)} \kappa_\pi(a_1, \dots, a_n)
\]
where the sum is over all $\pi$'s such that $\pi \vee
\tau_{\vec n} = 1_n$ and $\tau_{\vec n}$ is the partition
with blocks $(1, \dots, n_1), \dots, (n_1 + \cdots +
n_{r-1}+1, \dots, n_1 + \cdots + n_r)$.
\end{theorem}

In our example above $n_1 = n_2 = 2$, $\tau_{\vec{n}} =
\{(1,2),(3,4)\}$ and of the 14 non-crossing partitions of
$[4]$ there are ten that satisfy the condition $\pi \vee
\tau_{\vec n} = 1_4$; the four that don't being $\{(1), (2),
(3), (4)\}, \{(1), (2), (3, 4)\}$, $\{(1,2), (3), (4) \}$,
and $\{(1, 2), (3,4)\}$.

We shall need the second order version of this expansion.
For example if $a_1, a_2, a_3$ are in a second order
probability space $(\cA, \phi, \phi_2)$ we have
\begin{eqnarray*}\lefteqn{%
\kappa_{1,1}(a_1a_2, a_3) = \kappa_3(a_1, a_3, a_2) +
\kappa_{2,1}(a_1, a_2, a_3)}\\ && \mbox{}+ \kappa_{1,1}(a_1,
  a_3) \kappa_1(a_2) + \kappa_1(a_1) \kappa_{1,1}(a_2, a_3).
\end{eqnarray*}

To state the theorem in the second order case we need a new
concept, namely that of a permutation that separates points.

\begin{definition}\label{def:separates_points}
Suppose we have a permutation $\pi \in S_n$ and a
subset $A \subseteq [n]$; we say that $\pi$
\textit{separates the points} of $A$, if no two points of
$A$ are in the same cycle of $\pi$.  

If $\sigma \in S_n$ and $A \subseteq [n]$ is such that
$\sigma(A) = A$, i.e $\sigma$ leaves $A$ invariant, we
denote by $\sigma|_A$ the restriction of $\sigma$ to $A$. We
can extend this to the case when $\sigma$ does not leave $A$
invariant as follows. Let $\sigma|_A$ be the permutation of
$A$ given by $\sigma|_A(k) = \sigma^r(k)$ where $r \geq 1 $
is the smallest integer such that $\sigma^r(k) \in A$ but
$\sigma^s(k) \not\in A$ for all $1 \leq s < r$. We call the
permutation $\sigma|_A$ the \textit{induced permutation} of
$\sigma$ on $A$. When $\sigma$ leaves $A$ invariant this
reduces to the restriction of $\sigma$ to $A$.

Using the idea of an induced permutation the condition of
separating points can be written as follows: $\sigma$
separates the points of $A$ if and only if $\sigma|_A =
id_A$.
\end{definition}

\begin{remark}\label{rem:separates_points} 
It was shown in \cite[Lemma 14]{mst} that the condition $\pi
\vee \tau_{\vec{n}} = 1_n$ of Theorem \ref{theorem:ks} is
equivalent to the condition that $\pi^{-1} \gamma_n$
separates the points of $N= \{n_1, n_1 + n_2, \dots, n_1 +
\cdots + n_r\}$. For example one of the permutations not
appearing in the expansion of $\kappa_2(a_1a_2, a_3a_4)$ is
$\pi = (1,2)(3,4)$. Now $\pi^{-1}\gamma_4 = (1)(2,4)(3)$ and
$N = \{2, 4\}$. So $\pi^{-1}\gamma_4$ does not separate the
points of $N$. In the second order case the condition $\pi
\vee \tau_{\vec{n}} = 1_n$ becomes that $\pi^{-1}
\gamma_{m,n}$ separates the points of $N$.
\end{remark}

\begin{proposition}[\cite{mst} Thm.~3]\label{mst}
Suppose $n_1,\dots,n_r,n_{r+1},\dots,n_{r+s}$ are positive
integers, $p=n_1+\cdots + n_r, q=n_{r+1}+\cdots+n_{r+s},$ and
\[
N=\{n_1,n_1+n_2, \dots ,n_1 + \cdots + n_{r+s}\}.
\]
Given a second probability space $(\mathcal{A},\phi,\phi_2)$
and
\[ 
a_1, \dots ,a_{n_1}, a_{n_1+1}, \dots, a_{n_1+n_2}, \dots,
a_{n_1 + \cdots + n_{r+s}} \in \mathcal{A},
\]
let $A_1 = a_1 \cdots a_{n_1}, A_2= a_{n_1+1} \cdots
a_{n_1+n_2}, \dots,\ab A_{r+s} = a_{n_1 + \cdots + n_{r+s-1}+1} \cdots
\ab a_{n_1 + \cdots + n_{r+s}}$. Then
\begin{equation} \label{MST}
\kappa_{r,s}(A_1, \dots, A_r, A_{r+1}, \dots, A_{r+s}) =
\sum_{(V, \pi)} \kappa_{(V,\pi)} (a_1, \dots ,a_{p+q}),
\end{equation}
where the summation is over those $(V,\pi)\in \cPS_{NC}(p,q)$
such that $\pi^{-1}\gamma_{p,q}$ separates the points of
$N$.
\end{proposition}

\begin{remark}\label{rem:separates_points-2}

If we let $O = \{1, n_1 +1, n_1 + n_2 +1, \dots, n_1 +
\cdots + n_{r-1}+1\} = \gamma_n(N)$ and $n = n_1 + \cdots +
n_r$, then the condition that $\pi^{-1}\gamma_n$ separates
the points of $N$ is equivalent to $\gamma_n \pi^{-1}$
separates the points of $O$. Indeed if $n_k$ and $n_l$ are
in the same orbit of $\pi^{-1}\gamma_{n}$, i.e. there is $s$
such that $(\pi^{-1}\gamma_{n})^s(n_k) = n_l$ then
$(\gamma_{n}\pi^{-1})^s(n_k+1) = n_l+1$, and conversely.
\end{remark}

\subsection{Preliminaries on even and $R$-diagonal
  operators}
\label{sec:even_and_r-diagonal}

In this subsection we review the definitions and basic
properties of the main subject of this paper: even and
$R$-diagonal operators. We shall restate the first order
results of \cite{ns1} so that the connection with our
results becomes apparent.

Recall that a probability measure determined by moments is
symmetric if and only is all of its odd moments vanish. If
$(\cA, \phi)$ is a non-commutative $*$-probability space and
$x = x^* \in \cA$, we say that $x$ is \textit{even} if
$\phi(x^{2k-1}) = 0$ for $k = 1, 2, 3, \dots $. If $x$ is
even then it is determined by its even moments, which are
just the moments of $x^2$. The corresponding relation
between the free cumulants of $x$ and the free cumulants of
$x^2$ was found by Nica and Speicher \cite{ns1}.

\begin{definition}\label{def:partition_double}
Let $\pi \in NC(n)$ be a non-crossing partition. We will 
regard $\pi$ as the permutation whose cycles are the blocks
of $\pi$ with the elements arranged in increasing order. 
Then let $\hat \pi$ be the permutation of
$[2n]$ given by $\hat\pi^2(2k) = 2 \pi(k)$ and $\hat\pi(2k)
= \gamma_{2n}(2k)$, where $\gamma_{2n}$ is the permutation
of $[2n]$ with the one cycle $(1, 2, 3,\ab \dots, 2n)$. We
call $\hat\pi$ the \textit{double} of $\pi$.
\end{definition}

\begin{definition}
Let recall from \cite{av} the notion of an even partition. A
partition is \textit{even} if all of its blocks have an even
number of elements. Likewise we say a permutation is
\textit{even }if all of its cycles have an even number of
elements. Note that this is not the usual convention used in
group theory.
\end{definition}

The theorem of Nica and Speicher is as follows.

\begin{theorem}[\cite{ns1}]\label{theorem:ns1}
Let $x = x^* \in (\cA, \phi)$ be even. Then the free
cumulants of $x^2$ can be calculated from the free cumulants
of $x$ as follows.
\begin{equation}\label{equation:ns1}
\kappa_n(x^2, \dots, x^2) =
\sum_{\pi \in \NC(n)} \kappa_{\hat\pi}(x, \dots, x).
\end{equation}
\end{theorem}

For $a \in \cA$ let $a^{(1)} = a$ and $a^{(-1)} =
a^*$. Consider the $*$-cumulants $\kappa_m(a^{(\epsilon_1)},
a^{(\epsilon_2)}, \dots, a^{(\epsilon_m)})$ of $a$. Recall
(\cite{ns1}) that $a$ is $R$-diagonal if all $*$-cumulants
of $a$ are 0 except those of the form
\[
\kappa_{2n}(a, a^*, a, a^*, \dots, a, a^*) =
\kappa_{2n}(a^*, a, a^*, a, \dots , a^*, a).
\]

The relation between the cumulants of $a^*a$ and the
$*$-cumulants of $a$ is the same as between the square $x^2$
of an even operator and $x$; see Equation
(\ref{equation:ns1}).

\begin{theorem}[\cite{ns1}]\label{theorem:ns1:2}
Let $(\cA, \phi)$ be a non-commutative $*$-probability
space and $a \in \cA$ an $R$-diagonal operator. Then
\[
\kappa_{n}(a^*a, \dots, a^*a) =
\sum_{\pi \in \NC(n)} \kappa_{\hat\pi}(a^*, a, a^*, \dots, a^*, a).
\]
\end{theorem}


\section{Statement of results}\label{sec:statement_of_results}

\noindent
Let us present the main results of the paper.

\begin{definition}\label{defi1}
Let $(A,\phi,\phi_2) $ be a  second order
$*$-non-commutative probability space.
 
(1) An element $x\in(\mathcal{A},\phi,\phi_2)$ is called
\textbf{second order even} (or even for short) if $x=x^*$ and
$x$ is such that odd moments vanish,
i.e. $\phi_2(x^p,x^q)=0$ unless $p$ and $q$ are both even and
$\phi(x^{2n+1})=0$ for all $n\geq0$.

(2) An element $a\in(\mathcal{A},\phi,\phi_2)$ is called
\textbf{second order $R$-diagonal} if it is $R$-diagonal
(i.e. as in Definition \ref{def:r-diagonal}) and the only
non-vanishing second order cumulants are of the form
$$
\kappa_{2p,2q}(a,a^*, \dots, a,a^*) =
\kappa_{2p,2q}(a^*,a, \dots, a^*,a).
$$
\end{definition}

\begin{remark} 
From the moment-cumulant formulas it is clear that an
element $x\in \mathcal{A}$ is even if and only if it
self-adjoint and $\kappa_{p,q}(x,\dots,x)=0$ unless $p$ and
$q$ are even and $\kappa_{2n+1}(x,\dots,x)=0$ for $n \geq
0$.
\end{remark}

The first of our results shows that, as in the first
order case, $R$-diagonality is preserved when multiplying by
a free element.

\begin{theorem}\label{MT1}
Let $\{a, a^*\}$ and $\{b, b^\ast\}$ be second order free
and suppose that $a$ is second order $R$-diagonal. Then $ab$
is second order $R$-diagonal.
\end{theorem}

As we will see, the combinatorics of even and $R$-diagonal
operators are controlled by even annular non-crossing
partitions. As already noted, if $\pi \in \NC(n)$ is even
then for every $k$, $k$ and $\pi(k)$ have the opposite
parity, because in the gaps, i.e. between $k$ and $\pi(k)$,
there must always be an even number of elements.

In the case of non-crossing annular permutations $\pi \in
S_{NC}(2p,2q)$ something a little weaker than this
happens. If $k$ and $\pi(k)$ are in the same circle
(i.e. the same cycle of $\gamma_{2p,2q}$), they have the
opposite parity. However if $k$ and $\pi(k)$ are in
different circles they may have the same parity. See Lemmas
\ref{lemma:parity-preserving} and
\ref{lemma:parity-reversing}.

\begin{figure}[t]
\qquad
\includegraphics{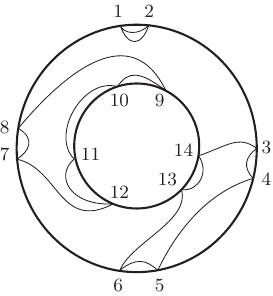}\hfill
\includegraphics{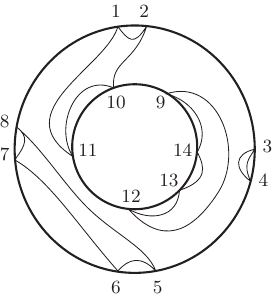}\qquad\hbox{}
\caption{\small\label{fig:parity_reversing}Two examples of even
  non-crossing annular permutations. We call the one on the left
  \textit{parity reversing} because for every $k$, $k$ and $\pi(k)$
  have the opposite parity. We call the one on the right
  \textit{parity preserving} because when $k$ and $\pi(k)$ are 
  on opposite circles they have the same parity.}
\end{figure}

Since for $x\in \mathcal{A}$ even or $R$-diagonal the
distributions of $xx^*$ may be calculated from the even
(resp. alternating) cumulants of $x$, let us introduce some
notations to encode this information.

\begin{definition}\label{def3}
Let $(A,\phi,\phi_2) $ be a second order non-commutative
probability space.

(1) Let $a$ be  second order $R$-diagonal. Define
$\beta^{(a)}_n:=\kappa_{2n}(a, a^*, \dots, a,\ab a^*)$ and
\begin{equation}
\beta_{p,q}^{(a)}:=\kappa_{2p,2q}(a,a^*,\dots,a, a^*).
\end{equation}
The sequences $(\beta_{n}^{(a)})_{n\geq1}$ and
$(\beta_{p,q}^{(a)})_{p,q\geq1}$ are called the (first and
second order) \textbf{determining sequences} of $a$

(2) Let $x$ be a second order even element. Define
$\beta^{(x)}_n:=\kappa_{2n}^{(x)} := \kappa_{2n}(x, \dots,
x)$ and letting $\kappa_{p,q}^{(x)} = \kappa_{p,q}(x, \dots,
x)$, we set
\begin{equation} \label{determining 1}
\beta^{(x)}_{p,q} := \kappa_{2p,2q}^{(x)} + \sum_{\pi \in
  \sncaplus{2p}{2q} } \kappa^{(x)}_{\pi}
\end{equation}
where, $\sncaplus{2p}{2q}$ denotes the set of even
non-crossing annular permutations $\pi$ with only through
cycles such that whenever $k$ and $\pi(k)$ are in different
cycles of $\gamma_{2p,2q}$ we have that $k$ and $\pi(k)$
have the same parity. See Figure \ref{fig:parity_reversing},
Definition \ref{def:reversing}, and Notation
\ref{not:all_thorough} for more details. We will call the
families $(\beta_{n}^{(x)})_{n\geq1}$ and
$(\beta_{p,q}^{(x)})_{p,q\geq1}$ the (first and second
order) \textbf{determining sequences} of $x$.
\end{definition}

Recall from \cite{ns2} that in the first order case the
cumulants and determining sequence are related by the
following equations.  For $x$ an even operator we 
have (see \cite[Eq.~11.15]{ns2})
\begin{equation}\label{eq:first_order_even_determining}
\kappa^{(x^2)} = \beta^{(x)} * \zeta \textrm{\ and\ }
\beta^{(x)} = \kappa^{(x^2)} * \mu.
\end{equation}
where $\zeta$ and $\mu$ are respectively the zeta and
M\"obius functions of the lattice of non-crossing partitions
\cite[Lecture 10]{ns2}. The convolution here is over the
lattice of non-crossing partitions; Equation
(\ref{eq:first_order_even_determining}) means that
\[
\kappa_n^{(x^2)} = \sum_{\pi \in NC(n)} \beta_\pi^{(x)}. 
\]
For an $R$-diagonal operator,  $a$,  we have
\begin{equation}
\kappa^{(aa^*)} = \beta^{(a)} * \zeta \textrm{\ and\ }
\beta^{(a)} = \kappa^{(aa^*)} * \mu.
\end{equation}  
See \cite[Prop.~15.6]{ns2}. As above this means that
\[
\kappa_n^{(aa*)} = \sum_{\pi \in NC(n)} \beta^{(a)}_\pi.
\]

The second of our results gives a formula for the second
order cumulants of $aa^*$ in terms of the determining
sequence of $a$.

\begin{theorem} \label{MT2}
Let $a$ be a second order $R$-diagonal operator with determining
sequences $(\beta^{(a)}_{n})_{n\geq1}$ and
$(\beta^{(a)}_{p,q})_{p,q\geq1}$ then we have
\begin{equation} \label{formula:main1}
\kappa_{p,q}(aa^*,\dots,aa^*)=\sum_{(\mathcal{V},\pi) \in
  \cPS_{NC}(p,q)}\beta^{(a)}_{(\mathcal{V},\pi)}.
\end{equation}
\end{theorem}

The proof of the last theorem will rely on the formula for
cumulants with products as arguments (Proposition \ref{mst}), and the
observation that the set 
$\{(\mathcal{V},\pi) \in
\cPS_{NC}(2p,2q)^{-}\mid \gamma_{2p,2q}\pi^{-1}$ separates
the points of $O\}$ is in bijection with $\cPS_{NC}(p,q)$
where $O = \{1, 3, 5, 7, \dots, 2p+2q-1\}$. By 
$\cPS_{NC}^-(2p, 2q)$ we mean all $(\cV, \pi) \in
\cPS_{NC}(2p, 2q)$ such that the cycles of $\pi$ alternate 
between even and odd numbers. 

Similarly we give a formula for the second order cumulants
of $x^2$ in terms of the determining sequence of $x$.
Analogous formulas for the first order case have been used
in \cite{ahs} to prove that properties of $x$ are
transferred to $x^2$, (e.g. free infinite divisibility or
representations as the multiplication with a free Poisson).

\begin{theorem}\label{MT3}
Let $x$ be an even element with determining sequences
$(\beta^{(x)}_{n})_{n\geq1}$ and
$(\beta^{(x)}_{p,q})_{p,q\geq1}$ then the second order
cumulants of $x^2$ are given by
\begin{equation} \label{formula:main2}
\kappa_{p,q}(x^2,\dots,x^2)=\sum_{(\mathcal{V},\pi) \in
  \cPS_{NC}(p,q)}\beta^{(x)}_{(\mathcal{V},\pi)}.
\end{equation}
\end{theorem} 

Note that even though formulas (\ref{formula:main1}) and
(\ref{formula:main2}) are the same, the definition of second
order determining sequences for even operators and for
$R$-diagonal ones are quite different. Furthermore,  note that using the 
second order zeta function, $\zeta$, (see \cite[\S5.4]{cmss})
we may respectively rewrite Equations (\ref{formula:main1}) 
and (\ref{formula:main2}) as
\begin{equation}\label{equation:beta_zeta}
\kappa^{(aa^*)} = \beta^{(a)} * \zeta \textrm{ and }
\kappa^{(x^2)} = \beta^{(x)} * \zeta.
\end{equation}
Then by \cite[Thm.~6.3]{cmss} we may also rewrite
Theorems \ref{MT2} and \ref{MT3} (together with their first
order counterparts) in terms of generating functions as
follows.

\begin{theorem} \label{series xx}
Let $x$ be either a second order even or second order
$R$-diagonal element with determining sequences
$(\beta_{n})_{n\geq1}$ and
$(\beta_{p,q})_{p,q\geq1}$. Denote by
$\kappa_n=\kappa_n^{(xx^*)}$ and
$\kappa_{p,q}:=\kappa_{p,q}^{(xx^*)}$, $($or $\kappa_n =
\kappa_n^{x^2}$ and $\kappa_{p,q}^{(x^2)}$ in case $x$ is
even$)$ and define the formal power series
$$
B(z)=\frac{1}{z}\sum_{n\geq1}\beta_n z^{n}, \qquad
B(z,w)=\frac{1}{zw}\sum_{m,n\geq1}\beta_{m,n} z^{m}w^{n}
$$
and 
$$
C(z)=\frac{1}{z}+\frac{1}{z}\sum_{n\geq1}\kappa_n z^{-n},
\qquad C(z,w)=\frac{1}{zw}\sum_{n,m\geq1}\kappa_{m,n}
z^{-m}w^{-n}.
$$
Then we have as a formal power series the relations
$$\frac{1}{C(z)}+B(C(z))=z,$$
and 
\[
C(z,w)=C'(z)C'(w)B(C(z),C(w))+ \frac{\partial^2}{\partial
  z\partial w} \log \left(\frac{C(z)-C(w)}{z-w}\right).
\]
\end{theorem}

Finally, we prove the following theorem which gives moments
and cumulants of products of free random variables. For $\pi
\in S_{NC}(p, q)$ we let $\textit{Kr}(\pi) =
\pi^{-1}\gamma_{p,q}$ denote the Kreweras complement of
$\pi$. Also we shall say that a permutation $\pi$ is
$k$-alternating if $\pi(i)=i+1$ mod $k$. The set of
$k$-alternating elements of $S_{NC}(kp,kq)$ is denoted
$S_{NC}^{k \textit{-alt}}(kp,kq)$. A permutation is $k$-equal
if every cycle is of size $k$. The set of $k$-alternating
and $k$ equal permutations in $S_{NC}(kp,kq)$ is denoted
$\snckea{kp}{kq}$, see Definition \ref{def:divisibility}.

\begin{theorem} 
[Second order Moments and Cumulants of Products of Free
  Variables]
\label{products} 
Let $a_1, \dots, a_k$ be operators which are second order
free and such that $\kappa_{p,q}^{(a_i)}=0$ for all $p$ and
$q$.  Denote by $a:=a_1a_2\cdots a_k$. Then
\begin{equation}\label{11a}
\phi_2(a^p,a^q)=\sum_{\pi \in \snckalt{kp}{kq}}
\kappa_{\Kr(\pi)}(a_1,a_2,\dots,a_k,\dots,a_1,a_2,\dots,a_k).
\end{equation}
Furthermore,
\begin{equation} \label{12a}
\kappa_{p,q}(a,\dots,a)=\sum_{\pi \in \snckea{kp}{kq}}
\kappa_{\Kr(\pi)}(a_1,\dots,a_k,\dots,a_1,\dots,a_k).
\end{equation}
\end{theorem}
This last theorem is of importance in wireless communication
because it includes the important case of products of
complex Wishart random matrices \cite{zwsmh}.


\section{Combinatorial lemmas on even annular permutations}
\label{sec:annular_non-crossing_partitions}

In this section we prove Propositions \ref{prop-even1} and
\ref{prop-even3}, which are the combinatorial results needed
to prove Theorems \ref{MT2} and \ref{MT3}. The case of an
even operator, Theorem \ref{MT3}, requires an analysis of
whether parity is preserved when crossing to the other
circle, c.f. Figure \ref{fig:parity_reversing}.

\begin{definition}\label{def:through_block}
Recall that $\gamma_{p,q}$ is the permutation in $S_{p+q}$
with the two cycles $(1,2,3, \dots, p)(p+1, \dots, p+q)$ and
$S_\NC(p,q)$ (the non-crossing annular permutations) is the
set of permutations $\pi$ in $S_{p+q}$ such that
at least one cycle of $\pi$ meets both cycles of
$\gamma_{p,q}$ and $|\pi| + |\pi^{-1}\gamma_{p,q}| =
|\gamma_{p,q}| + 2$, or equivalently $\#(\pi) +
\#(\gamma_{p,q} \pi^{-1}) = p + q$.  A cycle of $\pi$ that
meets both cycles of $\gamma_{p,q}$ is a \textit{through
  cycle}.
\end{definition}

For $\pi$, an even non-crossing permutation on $[n]$, $k$
and $\pi(k)$ always have the opposite parity; in the annular case
something a little weaker holds. See Lemmas
\ref{lemma:parity-preserving} and
\ref{lemma:parity-reversing}.

\begin{remark}\label{remark:unfolding}
Let $\pi \in S_\NC(p, q)$ then we can \textit{unfold} $\pi$
into a non-crossing partition on $[p + q]$. This unfolding
is not unique, but it is useful in reducing the annular case
to the disc case. We first illustrate this with an
example. Suppose $\pi = (1,5)(2,6)(3,4,7,8)$. Then $\pi \in
S_\NC(5,3)$. Let $\tilde\gamma = (1,2,3,4,7,8,6,5)$. Then
$\pi$ is non-crossing with respect to $\tilde\gamma$ in that
$|\pi| + |\pi^{-1}\tilde\gamma| = |\tilde\gamma|$, as in \S
\ref{subsec:preliminaries_statements}. See Figure
\ref{figure:unfolding}. The next lemma shows that we can do
this for every element of $S_{NC}(p,q)$.
\end{remark}

\setbox1=\hbox{\includegraphics[scale=0.85]{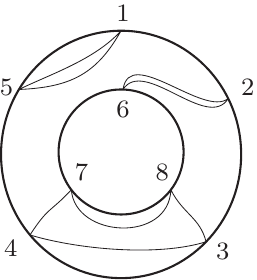}}
\setbox2=\hbox{\includegraphics[scale=0.85]{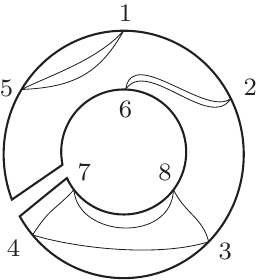}}
\setbox3=\hbox{\includegraphics[scale=0.85]{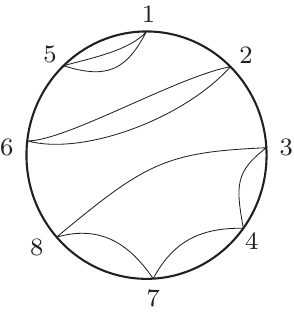}}

\begin{figure}

\noindent
$\vcenter{\hsize\wd1\box1}$\hfill
$\vcenter{\hsize\wd2\box2}$\hfill
$\vcenter{\hsize\wd3\box3}$

\caption{\small An unfolding of $\pi \in S_\NC(5,3)$. We cut
  a channel between the two circles following a cycle of
  $\pi$. In this example we go along the path from 4 to
  7. This turns the annulus into a disc. See Remark
  \ref{remark:unfolding} and Figure \ref{fig10-11} for
  another example. \label{figure:unfolding}}
\end{figure}

\begin{lemma}\label{lemma:unfolding}
Let $\pi \in S_\NC(p,q)$ and $k$ be such that $k$ and
$\pi(k)$ are in different cycles of $\gamma_{p,q}$. Let
$\tilde\gamma = \gamma_{p,q} \cdot(k,
\gamma_{p,q}^{-1}\pi(k))$, i.e. the product of $\gamma_{p,q}$ 
and the transposition $(k, \gamma_{p,q}^{-1}\pi(k))$. 
Then $\tilde\gamma$ has one cycle
and $|\pi| + |\pi^{-1}\tilde\gamma| = |\tilde\gamma|$,
i.e. $\pi$ is non-crossing with respect to $\tilde\gamma$.
\end{lemma}

\begin{proof}
Since $k$ and $\gamma_{p,q}^{-1}\pi(k)$ are in different
cycles of $\gamma_{p,q}$, $\tilde\gamma$ has only one
cycle. Also $\#(\pi) + \#(\pi^{-1}\tilde\gamma ) = \#(\pi) +
\#(\pi^{-1}\gamma_{p,q}) + 1 = p + q + 1$, since $k$ and
$\gamma_{p,q}^{-1}\pi(k)$ are in the same cycle of
$\gamma_{p,q}^{-1}\pi$. Thus $|\pi| + |\pi^{-1}\tilde\gamma|
= |\tilde\gamma|$
\end{proof}

\begin{remark}\label{remark:through-cycles}
Let $\pi$ and $k$ be as in Lemma \ref{lemma:unfolding}. Then
we can write $\tilde\gamma$ as the cycle
\[
(1,2,3, \dots, k, \pi(k), \gamma_{p,q}\pi(k), \dots, p + q,
p+1, \dots, \gamma_{p,q}^{-1} \pi(k), \gamma_{p,q}(k),
\dots, p).
\]
Suppose $c$ is a non-through cycle of $\pi$, say $c \subset
[p]$. Write $c = (i_1, i_2, \dots, \ab i_k)$. Since $c$ does
not meet $[p+1, p + q]$, at most one of the gaps $[i_1 +1,
  i_2-1]$, \dots, $[i_k+1, i_1-1]$ can contain a point on a
through block.
\end{remark}

\begin{lemma}\label{lemma:restriction}
Let $\pi \in S_\NC(p,q)$ and $\bar \pi$ be the partition of
$[p]$ obtained as follows. Each cycle of $\pi$ contained in
$[p]$ becomes a block of $\bar \pi$. All the remaining
points form one more block. Then $\bar \pi$ is non-crossing.
\end{lemma}

\begin{proof}
Let $k$ be such that $k \in [p]$ and $\pi(k) \in [p+1,
  p+q]$. Let $\tilde\gamma = \gamma_{p,q}(k,
\gamma_{p,q}^{-1}\pi(k))$. Then by Lemma
\ref{lemma:unfolding}, $\pi$ is non-crossing with respect to
$\tilde\gamma$. So also is $\sigma$, the partition which is
all singletons except for the block $(\pi(k),
\gamma_{p,q}\pi(k), \dots, \gamma_{p,q}^{-1}\pi(k))$, i.e.
\[
\sigma = \{\{1\},\{2\},\dots, \{k\}, \{\pi(k),\dots,
\gamma_{p,q}^{-1}\pi(k)\}, \{\gamma_{p,q}(k)\}, \dots, \{p\}\}.
\] 
Hence, $\pi \vee \sigma$ is also non-crossing, and thus so
is $\bar \pi = (\pi \vee \sigma)|_{[p]}$.
\end{proof}

\begin{lemma}\label{lemma:parity-preserving}
Suppose $\pi \in S_\NC(p,q)$ with $p$ and $q$ even, has all
cycles of even length. Then any cycle of $\pi$ which is
contained in one of the two cycles of $\gamma_{p,q}$
alternates between even and odd elements. Moreover, if there
is $k$ such that $k$ and $\pi(k)$ lie in different cycles of
$\gamma_{p,q}$ and have opposite parities, then all cycles
of $\pi$ alternate between even and odd elements.

\end{lemma}

\begin{proof}
Let $c = (i_1, \dots, i_{2k})$ be a cycle of $\pi$ which
lies in one cycle of $\gamma_{p,q}$, and $i_l$ and $i_{l+1}$
adjacent points in this cycle. Consider the cyclic interval
$[i_l + 1, i_{l +1}-1]$, see
\S\ref{subsection:non-crossing_annular_permutations} for the
definition. If no point of this cyclic interval lies in a
through cycle of $\pi$, then $\pi$ restricts to a partition
of this interval with all cycles even, thus there must be an
even number of elements in this interval, and hence $i_l$
and $i_{l+1}$ must have opposite parities. Now suppose that
there is an element of $[i_l+1, i_{l+1}-1]$ that is in a
through cycle of $\gamma_{p,q}\pi^{-1}$. Consider each of
the other gaps in $c$: $[i_{l+1}+1, i_{l+2}-1], \dots,
[i_{l-1}+1, i_l-1]$. By Remark \ref{remark:through-cycles}
none contains an element which is on a through cycle of
$\pi$. Thus $\pi$ restricts to a partition of each of these
intervals with each restriction having blocks of even
size. Thus each gap has an even number of elements. This
means $\pi$ alternated between even and odd numbers as we
move around a circle. By hypothesis $\pi$ alternates as we
cross over. Thus $\pi$ alternates between even and odd
numbers.
\end{proof}

\begin{lemma}\label{lemma:parity-reversing}
Suppose $\pi \in S_\NC(p,q)$ with $p$ and $q$ even, has all
cycles of even length. If there is $k$ such that $k$ and
$\pi(k)$ lie in different cycles of $\gamma_{p,q}$ and have
the same parity, then all through cycles of $\pi$ alternate
between even and odd elements except when they cross between
the cycles of $\gamma_{p,q}$, i.e. if $j$ and $\pi(j)$ are
in the same cycle of $\gamma_{p,q}$ then they have opposite
parities and if they are in different cycles of
$\gamma_{p,q}$ they have the same parity.
\end{lemma}

\begin{proof}
Suppose $k \in [p]$ and $\pi(k)$ lie in different cycles
of $\gamma_{p,q}$ and have the same parity. Let $\gamma_p$ 
be the permutation in $S_{p+q}$ which in
cycle notation is $(1,2,3, \dots, p)$, let $\tilde\pi =
\gamma_p \pi \gamma_p^{-1}$, and let $\tilde k =
\gamma_p(k)$. Then all the cycles of $\tilde \pi$ have even
length, and $\tilde k$ and $\tilde\pi(\tilde k)$ have
opposite parities and lie in different cycles of
$\gamma_{p,q}$. So by Lemma \ref{lemma:parity-preserving}
all cycles of $\tilde\pi$ alternate between odd and even
elements.

If $j$ and $\pi(j)$ are in the same cycle of
$\gamma_{p,q}$ they have opposite parities, as either they
are in $[p+1, \dots, p+q]$ and $\tilde\pi(j) = \pi(j)$, or
they are both in $[p]$ and if we let $l = \gamma_p(j)$ then
$l$ and $\tilde\pi(l) = \gamma_p (\pi(j))$ have opposite
parities and thus $j = \gamma^{-1}_p(l)$ and $\pi(j) =
\gamma_p^{-1}(\tilde\pi(l))$ have opposite parities.

Alternatively, suppose that $j$ and $\pi(j)$ are in different
cycles of $\gamma_{p,q}$. First suppose that $k \in [j]$ and
$\pi(j) \in [p+1, p+q]$. Let $l = \gamma_p(j)$. Then $l$ and
$\tilde\pi(l) = \pi(j)$ have opposite parities, by Lemma
\ref{lemma:parity-preserving}. Thus $j$ and $\pi(j)$ have
the same parity. On the other hand, suppose that $\pi(j) \in
[p]$ and $j \in [p+1, p+q]$. Then $j$ and $\tilde\pi(j) =
\gamma_p(\pi(j))$ have opposite parities, and thus $j$ and
$\pi(j)$ have the same parity.
\end{proof}

\begin{definition}\label{def:reversing}
Suppose $\pi \in S_\NC(2p, 2q)$ is even.  We say that $\pi$
is \textit{parity reversing} if for all $k$, $\pi(k)$ and
$k$ have the opposite parity. We denote the elements of
$S_\NC(2p, 2q)$ that have cycles of even length and that are
parity reversing by $S_\NC^-(2p, 2q)$. The remaining even
elements of $S_\NC(2p, 2q)$, we call \textit{parity
  preserving} because for any $k$ such that $k$ and $\pi(k)$
are in different cycles of $\gamma_{2p, 2q}$ we have that
$k$ and $\pi(k)$ have the same parity, c.f. Lemma
\ref{lemma:parity-reversing}. We denote the parity
preserving elements by $S_\NC^+(2p, 2q)$. See Figure 
\ref{fig:parity_reversing}. 
\end{definition}

\begin{remark}\label{remark:no-pairings}
If $\pi \in S_\NC(2p,2q)$ and $\gamma_{2p, 2q}\pi^{-1}$
separates the points of $O$ (c.f. Remark
\ref{rem:separates_points-2}) then we cannot have $\pi(2k) =
2l-1$ with $2k$ and $2l-1$ on different circles. For if
$\pi(2k) = 2l-1$ then we would have both $\gamma_{2p,
  2q}(2k) \not= 2l-1$, because they are on different
circles, and $\gamma_{2p, 2q} \pi^{-1} ( 2l -1) =
\gamma_{2p, 2q}(2k)$; since $\gamma_{2p,2q}(2k) \in O$, this
is contrary to our assumption about $\gamma_{2p,
  2q}\pi^{-1}$ separating the points of $O$. Thus either
$\pi(2k)$ is on the same circle as $2k$ or $\pi(2k) =
2l$. We have shown that if $l$ and $\pi(l)$ are on different
circles then either both are even or $l$ is odd. This means
that there are no \textit{pairings} in $S_\NC^-(2p, 2q)$
such that $\gamma_{2p,2q}\pi^{-1}$ separates the points of
$O$. See Figure \ref{fig:pairing_in_parity_reversing}.
\end{remark}

\begin{figure}
  \begin{minipage}[c]{0.4\textwidth}
    \includegraphics{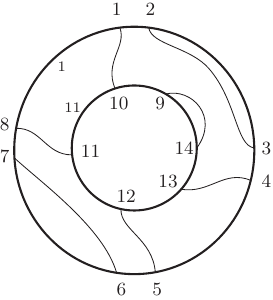}
  \end{minipage}\kern-3em
  \begin{minipage}[c]{0.7\textwidth}
    \caption{\small A pairing in $S_\NC^-(8,6)$. Note that whenever
  there is a through string, (1, 10) in this example, we
  must have two elements of $O$, $\{1,11\}$ in this example,
  in the same cycle of $\gamma_{2p,2q}\pi^{-1}$. See Remark
  \ref{remark:no-pairings}.\label{fig:pairing_in_parity_reversing}}
  \end{minipage}
\end{figure}

\begin{lemma}\label{lemma:singletons}
Suppose $\pi \in S_\NC^-(2p, 2q)$. If $\gamma_{2p,2q}
\pi^{-1}$ separates the points of $O$ then, for all $k$,
$\gamma_{2p,2q} \pi^{-1}(2k-1) = 2k-1$, or equivalently for
all $k$, $\pi(2k) = \gamma_{2p,2q}(2k)$.

\end{lemma}

\begin{proof}
Recall that for $\pi\in S_{NC}^-(2p, 2q)$, $\pi(2k-1)$ is
even, for all $k$. Thus $\gamma_{2p, 2q}\pi^{-1}(2k-1)$ is
odd for all $k$. Since $\gamma_{2p, 2q}\pi^{-1}$ separates
the points of $O$, we must have that $\gamma_{2p,
  2q}\pi^{-1}(2k-1) = 2k -1$, for all $k$. Let $2l =
\gamma_{2p,2q}^{-1}(2k-1)$ then $\pi(2l) = 2k-1 =
\gamma_{2p,2q}(2l)$.
\end{proof}

\begin{definition}
Suppose $\pi \in S_\NC^-(2p, 2q)$ is such that
$\gamma_{2p,2q}\pi^{-1}$ separates the points of $O$. Let
$\check\pi$ be the permutation defined by $2 \check\pi(k) =
\pi^2(2k)$, for $1 \leq k \leq p+q$.
\end{definition}

\begin{lemma}\label{lemma:orbits}
Suppose $\pi\in S_\NC^-(2p,2q)$ is such that
$\gamma_{2p,2q}\pi^{-1}$ separates the points of $O$. Then
$\gamma_{p,q}\check\pi^{-1}(k) = l$ if and only if
$\gamma_{2p,2q} \pi^{-1}(2k) = 2l$.
\end{lemma}

\begin{proof}
By Lemma \ref{lemma:singletons}, for all $1 \leq k \leq p +
q$, $\pi(2k) = \gamma_{2p,2q}(2k)$. Thus for any $1 \leq k
\leq p+q$, $\pi(2 \gamma_{p,q}^{-1}(k)) = \gamma_{2p,2q}(2
\gamma_{p,q}^{-1}(k)) = 2k-1 = \gamma_{2p,2q}^{-1}(2k)$. Hence
{\allowdisplaybreaks\begin{align*}
\gamma_{p,q}\check\pi^{-1}(k) &= l 
\Leftrightarrow
\check\pi(\gamma_{p,q}^{-1}(l)) = k
\Leftrightarrow
2\check\pi(\gamma_{p,q}^{-1}(l)) = 2k \\
\Leftrightarrow
\pi^2(2\gamma_{p,q}^{-1}(l)) &= 2k 
\Leftrightarrow
\pi(\gamma_{2p,2q}^{-1}(2l)) = 2k 
\Leftrightarrow
\gamma_{2p,2q}\pi^{-1}(2k) = 2l.
\end{align*}}
\end{proof}

\begin{lemma}
Suppose $\pi\in S_\NC^-(2p,2q)$ is such that
$\gamma_{2p,2q}\pi^{-1}$ separates the points of $O$. Then
$\check\pi \in S_\NC(p,q)$
\end{lemma}

\begin{proof}
We have $\#(\check\pi) = \#(\pi)$. The cycles of
$\gamma_{2p,2q} \pi^{-1}$ are the singletons $(2k-1)$, of
which there are $p+q$, and the cycles consisting of even
numbers. But by Lemma \ref{lemma:orbits} the orbits of even
numbers correspond to the orbits of $\gamma_{p,q}
\check\pi^{-1}$. Thus $\#(\gamma_{2p,2q} \pi^{-1}) = p + q +
\#(\gamma_{p,q} \check\pi^{-1})$. Hence
\[
\#(\check\pi) + \#(\gamma_{p,q} \check\pi^{-1}) = \#(\pi) +
\#(\gamma_{p,q} \pi^{-1}) - (p + q) = p + q.
\]
By Remark \ref{remark:no-pairings}, $\pi$ has a through
block with more than two elements, so $\check\pi$ will have a
through block, and thus $\check\pi \in S_\NC(p,q)$.
\end{proof}

\begin{figure}
\setbox1=\hbox{\includegraphics{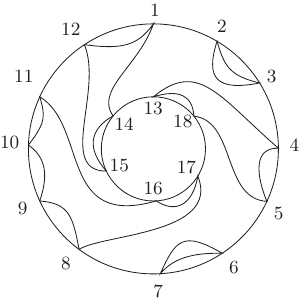}}
\setbox2=\hbox{\includegraphics{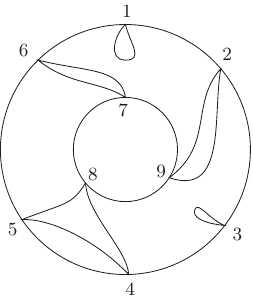}}
\hfill$\vcenter{\hsize=\wd1\box1}$
\hfill
$\vcenter{\hsize=\wd2\box2}$
\hfill{}
\caption{\small\label{fig8-9}\small On the left we see the
  permutation $\pi = (1,\ab 14, 15,12)(2,3)\ab(4,5,\ab18,\ab
  13)(6,7)(8,9,10,11,16,17)$. $\pi$ is in $S_{NC}^-(12,6)$
  and $\gamma_{12,6}\pi^{-1}$ separates the points of $O
  =\ab \{1,\ab 3,\ab 5,\ab 7,\ab 9,\ab 11,13,15,17\}$. On the right is
  $\check\pi = (1)\ab (2,9)\ab (3)\ab (4,5,8)\ab (6,7)$. Note that
  $\pi^2|_E = (2)\ab (4,18)\ab (6)\ab (8, \ab 10,\ab 16)\ab
  (12, 14)$ where $\pi^2|_E$ is the restriction of $\pi^2$
  to the even numbers.  }
\end{figure}

The next lemma will be crucial in proving Theorem \ref{MT2} 
on the determining sequences of a second order $R$-diagonal
element.

\begin{lemma} \label{lemma check} The map
$\pi \mapsto \check\pi$ is a bijection from $S_\NC^-(2p,2q)
  \cap \{\pi \mid \gamma_{2p, 2q}\pi^{-1}$ separates the
  point of $O\}$ to $S_\NC(p,q)$.

\end{lemma}

\begin{proof}
Suppose that $\pi_1$ and $\pi_2$ are in $S_\NC^-(2p,2q)\cap \{\pi
\mid \gamma_{2p, 2q}\pi^{-1} \sep\ab O \}$ and $\check\pi_1
= \check\pi_2$. (Here and below ``sep. $O$"" means ``separates the points of $O$''). Therefore $\pi_1^2(2k) = \pi_2^2(2k)$ for all
$k$. Since $\pi_i(2k) = \gamma_{2p,2q}(2k)$ for $i= 1,2$, we
only have to show that $\pi_1$ and $\pi_2$ agree on the odd
numbers. But $\pi_1(\gamma_{2p,2q}(2k)) = \pi_1^2(2k) =
\pi_2^2(2k) = \pi_2(\gamma_{2p,2q}(2k))$. Hence $\pi_1$ and
$\pi_2$ agree on the odd numbers as well. This proves
injectivity.

Let $\sigma \in S_\NC(p,q)$. Define $\pi$ by $\pi(2k) =
\gamma_{2p,2q}(2k)$ and $\pi(\gamma_{2p,2q}(2k)) = 2
\sigma(k)$. Then $\pi^2(2k) = 2\sigma(k)$. So if $(i_1,
\dots, i_k)$ is a cycle of $\sigma$, then the corresponding
cycle of $\pi$ is $(2i_1, \gamma_{2p, 2q}(2i_1), 2i_2,
\dots, 2i_k, \gamma_{2p, 2q}(2i_k))$. Thus $\pi$ is even and
by construction is a parity reversing permutation. Also
$\#(\pi) = \#(\sigma)$. Furthermore, we have that
$$\gamma_{2p,2q}(\pi^{-1}(2k-1)) = 2k-1.$$ Thus every odd
number is a singleton of $\gamma_{2p, 2q}\pi^{-1}$. As in the
proof of Lemma \ref{lemma:orbits} we have
$\gamma_{p,q}\sigma^{-1}(k) = l \Leftrightarrow
\gamma_{2p,2q}\pi^{-1}(2k) = 2l$, thus
\[
\#(\pi) + \#(\gamma_{2p,2q}\pi^{-1}) = \#(\sigma) +
\#(\gamma_{p, q}\sigma^{-1}) + p + q = 2p + 2 q.
\]
Thus $\pi \in S_\NC^-(2p, 2q)$. Since, by construction,
$\pi(2k) =\gamma_{2p,2q}(2k)$, we have $\gamma_{2p, 2q}
\pi^{-1}$ separates the points of $O$. Therefore $\pi \in
S_\NC^-(2p,2q)\cap \{\pi \mid \gamma_{2p, 2q}\pi^{-1}
\sep\ab O \}$ and $\check\pi = \sigma$.
\end{proof}

See Figure \ref{fig8-9} for an example of the relation
between $\pi$ and $\check\pi$.

\begin{lemma}\label{lemma:plus-singletons}
Let $\pi \in S_\NC^+(2p, 2q)$ be such that $\gamma_{2p,
  2q}\pi^{-1}$ separates the points of $O$. Suppose $k$ is
such that $2k$ and $\pi(2k)$ are in the same cycle of
$\gamma_{2p,2q}$. Then $\pi(2k) = \gamma_{2p,2q}(2k)$.
\end{lemma}

\begin{proof}
First let us suppose that $2k$ is in a through cycle of
$\pi$. Let $l$ be such that $2k$, $\pi(2k)$, \dots,
$\pi^{l-1}(2k)$ are in the same cycle of $\gamma_{2p, 2q}$
but $\pi^l(2k)$ is in a different cycle. Let $\tilde\gamma =
\gamma_{2p, 2q} (\pi^{l-1}(2k), \ab\gamma_{2p,
  2q}^{-1}\pi^l(2k))$, then $\tilde\gamma$ has one cycle
and, by Lemma \ref{lemma:unfolding}, $\pi$ is non-crossing
with respect to $\tilde\gamma$. See Figure
\ref{fig10-11}. Let us suppose $2k \in [2p]$; the case when
$2k \in [2p+1,\ab 2p + 2q]$ is identical. The we may write
\begin{eqnarray*}
\tilde\gamma
&=&
(1, 2, 3, \dots, 2k, \dots, \pi^{l-1}(2k),
\pi^l(2k), \gamma_{2p,2q}\pi^l(2k), \dots, \\
&& \quad
2p + 2q, 2p +1, \dots, \gamma_{2p,2q}^{-1}\pi^l(2k),
\gamma_{2p,2q}\pi^{l-1}(2k), \dots, p)
\end{eqnarray*}
The cyclic interval $I = (2k, \gamma_{2p,2q}(2k), \dots,
\pi^{l-1}(2k))$ lies in the cycle  of
$\gamma_{2p,2q}$ containing $(1,2,3, \ab\dots,\ab 2p)$, and the endpoints of $I$ lie in the same
cycle of $\pi$, see \cite[Remark 3.4 (2)]{mn}. Let $c$ be a
cycle of $\pi$ containing a point of $I$ but not $2k$. Since
$\pi$ is non-crossing with respect to $\tilde\gamma$, $c$
must be contained in $I$ and thus not be a through
cycle. Thus the gap, if it exists, between $2k$ and
$\pi(2k)$ is a union of non-through cycles of $\pi$. Hence
$\pi(2k) = 2j -1$ for some $j$. Then $\gamma_{2p,2q}(2k)$
and $\pi(2k)$ are both in $O$ and in the same cycle of
$\gamma_{2p, 2q}\pi^{-1}$. Hence they must be equal.

\begin{figure}
\setbox1=\hbox{\includegraphics{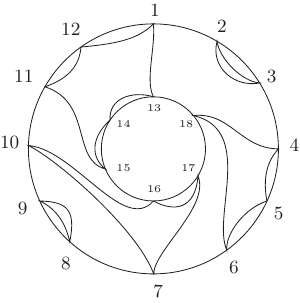}}
\setbox2=\hbox{\includegraphics{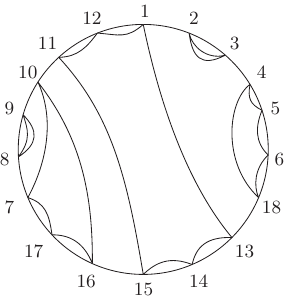}}
\hfill$\vcenter{\hsize=\wd1\box1}$
\hfill
$\vcenter{\hsize=\wd2\box2}$
\hfill{}
\caption{\small\label{fig10-11}\small On the left we see the
  permutation $\pi = (1, \ab 13, 14, 15,11, 12)(2,3)\ab(4,5,6, \ab 18),\ab
  (7, 10, 16, 17)(8,9)$. $\pi$ is in $S_{NC}^+(12,6)$
  and $\gamma_{12,6}\pi^{-1}$ separates the points of $O
  = \{1,\ab 3, 5, 7, 9,11,13,15,17\}$. On the right is
  $\pi$ drawn relative to  $\gamma_{12,6}(6,17)$.}
\end{figure}

Now suppose that $2k$ is not in a through cycle of
$\pi$. Then by Lemma~\ref{lemma:parity-preserving}, $\pi(2k)
= \gamma_{2p,2q}(2l)$ for some $l$. So $\gamma_{2p,
  2q}\pi^{-1}( \gamma_{2p, 2q}(2l)) =
\gamma_{2p, 2q}(2k)$. Then we have that $\gamma_{2p,2q}(2k)$ and
$\gamma_{2p, 2q}(2l)$ are in the same orbit of $\gamma_{2p,
  2q}\pi^{-1}$ and thus $\gamma_{2p, 2q}(2k) = \gamma_{2p,
  2q}(2l)$. Hence $\pi(2k) = \gamma_{2p, 2q}(2l) =
\gamma_{2p, 2q}(2k)$ as required.
\end{proof}

\begin{notation}\label{notation:double}
Let $V \subset [p + q]$ and $\hat V \subset [2p + 2q]$ be
the set $\hat V = \{ 2k \mid k \in V \} \cup \{ 2k + 1 \mid
k \in V \}$. We call $\hat V$ the \textit{double} of $V$.
\end{notation}

\begin{lemma}
Let $\pi \in S_\NC^+(2p,2q)$ be such that $\gamma_{2p,
  2q}\pi^{-1}$ separates the points of $O$. Let $V$ be the
union of all through cycles of $\pi$. Then there is $U
\subset [p + q]$ such that $V = \hat U$.

\end{lemma}

\begin{proof}
We must show that $V$ is the disjoint union of cyclic
intervals of the form
\[
\{2i, \gamma_{2p,2q}(2i), \gamma_{2p,2q}^2(2i), \dots,
\gamma_{2p,2q}^{2r+1}(2i) \} \eqno (*)
\] 
for if this is so then, we can take $U$ to be the
corresponding union of cyclic intervals $\{i,
\gamma_{p,q}(i), \dots, \gamma_{p,q}^r(i) \}$. To prove that
$V$ has this structure it suffices to show that for each
through cycle $c$ of $\pi$ we have that both $c \cap [2p]$
and $c \cap[2p + 1, 2p + 2q]$ have this structure, i.e. are
a disjoint union of cyclic intervals of the form $(*)$; this
reduces to showing that if $2k$ is in a through cycle of
$\pi$ and both $2k$ and $\pi(2k)$ are in the same cycle of
$\gamma_{2p,2q}$, then $\pi(2k) = \gamma_{2p,2q}(2k)$. This
is what was proved in Lemma~\ref{lemma:plus-singletons}.
\end{proof}

\begin{notation}\label{notation-grouped}
Recall that $\cPS_\NC(p,q)'$ is the set of all partitioned
permutations $(\cU, \pi)$ where $\cU$ has a block which is
the union of two cycles of $\pi$, see Section
\ref{section:partitioned_permutations}.  Suppose we are
given $(\cU, \sigma) \in \cPS_\NC(p,q)'$. Write $\sigma =
\sigma_1 \times \sigma_2 \in NC(p) \times NC(q)$. Let $U$ be
the block of $\cU$ which is the union of a cycle of
$\sigma_1$ and a cycle of $\sigma_2$. Let $S_\NC^{(\cU,
  \sigma)}(2p, 2q)$ be the set of $\pi \in S_\NC^+(2p, 2q)$
such that

\begin{itemize}
\item
each non-through cycle of $\pi$ is the double of some cycle
$c$ of $\sigma$ (c.f. Definition \ref{def:partition_double})

\item
the union of all the through cycles of $\pi$ is $\hat U$.

\end{itemize}
\end{notation}

\begin{proposition} \label{prop-even1}
\[
\{\pi \in S_\NC^+(2p, 2q) \mid \gamma_{2p, 2q}\pi^{-1} \sep
O\} = \ds\bigcup_{(\cU, \sigma) \in \cPS_\NC(p,q)'}
S_\NC^{(\cU, \sigma)}(2p, 2q)
\]
and the union is disjoint.
\end{proposition}

\begin{proof}
First let us show that the union is disjoint. Suppose $\pi
\in S_\NC^{(\cU, \sigma)}(2p, 2q)$. Since each non-through
cycle of $\pi$ is the double of a cycle of $\sigma$, all the
cycles of $\sigma$, except the two cycles joined by $\cU$,
are determined by $\pi$. Moreover the union of the two
cycles of $\sigma$ joined by $\cU$ is the union of the
through cycles of $\pi$; so these two cycles of $\sigma$,
one in $[p]$ and the other in $[p+1, p+q]$, are also
determined by $\pi$. Thus the union is disjoint.

Let $\pi \in S_\NC^+(2p, 2q)$ be such that $\gamma_{2p,
  2q}\pi^{-1}$ separates the points of $O$. Let $\pi_1\in
NC(p)$ and $\pi_2 \in NC([p+1, p+q])$ be the non-crossing
partitions constructed in Lemma \ref{lemma:restriction}. By
Lemma \ref{lemma:parity-preserving}, each non-through cycle
of $\pi$ is the double of a subset of either $[p]$ or $[p+1,
  p+q]$. By Lemma \ref{lemma:plus-singletons} the blocks of
$\pi_1$ and $\pi_2$ which come from the through cycles of
$\pi$ are also the doubles of subsets of $[p]$ and $[p+1,
p+q]$ respectively. Hence there are $\sigma_1 \in NC(p)$ and
  $\sigma_2 \in NC([p+1, p+q])$ such that $\pi_1 = \hat
  \sigma_1$ and $\pi_2 = \hat \sigma_2$. Now let $\cU \in
  \cP(p+q)$ be the partition whose blocks are just the
  blocks of $\sigma_1$ and $\sigma_2$, except we join the
  blocks of $\sigma_1$ and $\sigma_2$ coming from the
  through cycles of $\pi$. Letting $\sigma = \sigma_1 \times
  \sigma_2$, we have $\pi \in S_\NC^{(\cU, \sigma)}(2p,
  2q)$.
\end{proof}

\begin{notation}\label{not:all_thorough}
Suppose $\pi \in S_\NC(p,q)$ is such that all cycles contain
points of both cycles of $\gamma_{p,q}$. Then we say that
$\pi$ has \textit{all through cycles}. The set of
non-crossing annular permutations with all through cycles is
denoted $S_\NC^{\all}(p,q)$.
\end{notation}

\begin{lemma} \label{lem-even2}
Let $\pi \in S_\NC^\all(p,q)$, then all cycles of
$\gamma_{p,q} \pi^{-1}$ are either singletons or pairs and
the pairs are all through cycles. Moreover, if $\pi$ is
parity preserving (c.f. Def. \ref{def:reversing}), then
$\gamma_{p,q}\pi^{-1}$ separates the points of $O$.
\end{lemma}

\begin{proof}
Suppose $\pi$ has $k$ cycles. Then there are cyclic
intervals 
\[
I_1, I_2, \dots,\ab I_k \subset [p] \mbox{\ and\ } 
J_1, J_2, \dots, J_k \subset [p+1, p+q]
\] 
with 
\[
I_l = (i_1^{(l)},
i_2^{(l)}, \dots, i_{r_l}^{(l)}) \mbox{\ and\ } 
J_l = (j_1^{(l)}, j_2^{(l)}, \dots, j_{s_l}^{(l)})
\] 
such that the $l^{th}$
cycle of $\pi$ is $(i_1^{(l)}, i_2^{(l)}, \dots,
i_{r_l}^{(l)},\ab j_1^{(l)}, j_2^{(l)}, \dots,\ab
j_{s_l}^{(l)})$. Hence the only through cycles of
$\gamma_{2p, 2q}\pi^{-1}$ are of the form
$(\gamma_{p,q}(i_{r_l}^{(l)}),\ab j_1^{(l)})$ and all other
cycles are singletons. If in addition $\pi$ is parity
preserving, then $i_{r_l}^{(l)}$ and $j_1^{(l)}$ are of the
same parity, so $\gamma_{p,q}(i_{r_l}^{(l)})$ and $\ab
j_1^{(l)}$ are of opposite parities. Thus
$\gamma_{p,q}\pi^{-1}$ separates the points of $O$.
\end{proof}

\begin{lemma}\label{lemma:hat-bijection}
Let $O =\{1,3,5, \dots, 2n-1\}$.
Suppose $\pi \in \NC(2n)$ is even and $\gamma_{2n}\pi^{-1}|_O =
\id_O$. Then there is $\sigma \in \NC(n)$ such that $\pi =
\hat \sigma$, in fact given by $2 \sigma(k) =
\pi^2(2k)$. Conversely, given $\sigma \in NC(n)$,
$\hat\sigma$ is even and $\gamma_{2n}\hat\sigma^{-1}|_{O} =
id_O$.
\end{lemma}

\begin{proof}
As $\pi$ is an even permutation then the cycles of $\pi$
must alternate between even and odd numbers (because any
gaps in an orbit must skip over an even number of points),
hence $\gamma_{2n}\pi^{-1}$ leaves both $N$ (the even
numbers) and $O$ invariant. The condition
$\gamma_{2n}\pi^{-1}|_O = \id_O$ then implies that for every
$k$, $\gamma_{2n}\pi^{-1}(2k-1) = 2k-1$. Hence for every
$k$, $\pi(2k) = \gamma_{2n}(2k)$, and thus if we let $\sigma
\in \NC(n)$ be given by $2\sigma(k) = \pi^2(2k)$ we have
that $\pi = \hat\sigma$ is the double (c.f. Definition
\ref{def:partition_double}) of $\sigma$, as required by
Theorem \ref{theorem:ns1}.
\end{proof}

The next proposition will be crucial in proving Theorem
\ref{MT3}; it will serve as the annular analogue of 
Lemma \ref{lemma:hat-bijection}. 

\begin{proposition}\label{prop-even3}
Let $(\cV, \pi) \in \cPS_\NC(2p, 2q)'$ be such that $\pi$ is
even and $\gamma_{2p, 2q} \pi^{-1}$ separates the points of
$O$. Then there is $$(\cU, \sigma) \in \cPS_\NC(p, q)'$$ such
that $\cV = \widehat{\cU}$ and $\pi = \hat \sigma$. Moreover
this correspondence is a bijection.
\end{proposition}

\begin{proof}
Since $(\cV, \pi) \in \cPS_\NC(2p, 2q)'$, $\pi = \pi_1
\times \pi_2$ with $\pi_1 \in NC(2p)$ and $\pi_2 \in
NC(2q)$. Moreover $\pi_1$ and $\pi_2$ are both even and
$\gamma_{2p}\pi_1^{-1}$ and $\gamma_{2q}\pi_2^{-1}$ separate
the points of $\{1, 3, 5, \dots, 2p-1\}$ and $\{1,3,5,
\dots, 2q-1\}$ respectively. Thus, by Lemma
\ref{lemma:hat-bijection}, there are $\sigma_1 \in NC(p)$
and $\sigma_2 \in NC(q)$ such that $\pi_1 = \hat \sigma_1$
and $\pi_2 = \hat\sigma_2$. Thus $\pi = \hat \sigma$. Now
$\cU$ is formed by joining a cycle of $\sigma_1$ with a
cycle of $\sigma_2$. So if we form $\cV$ by joining the
corresponding cycles of $\pi$ and $\pi_2$ then $\cV =
\widehat \cU$. Since we can recover $\sigma$ from
$\hat\sigma$ this correspondence is a bijection.
\end{proof}


\section{$R$-diagonal elements of second order}
\label{sec:second_order_r-diagonal}\noindent
In this section we prove Theorems \ref{MT1} and \ref{MT2}
regarding $R$-diagonal elements.

\subsection{The proof of Theorem \ref{MT1}}
Theorem \ref{MT1} asserts that if $\{a, a^*\}$ and $\{b,
b^*\}$ are second order free and $a$ is $R$-diagonal of
second order then $ab$ is $R$-diagonal of second order.

Let $c = ab$ and $c^{(1)} = c$ and $c^{(-1)} = c^*$. So we
must show that
\begin{enumerate}
  
\item[$\alpha$)]
$\kappa_n(c^{(\epsilon_1)}, c^{(\epsilon_2)}, \dots,
  c^{(\epsilon_n)}) = 0$ unless $n$ is even and $\epsilon_i
  = -\epsilon_{i+1}$ for $1 \leq i < n$, and

\item[$\beta$)]
$\kappa_{p,q}(c^{(\epsilon_1)}, c^{(\epsilon_2)}, \dots,
  c^{(\epsilon_{p+q})}) = 0$ unless $p$ and $q$ are even and
  for $1 \leq i < 2p$ or $2p +1 \leq i < 2p + 2q$ we have
  $\epsilon_i = - \epsilon_{i+1}$.

\end{enumerate}
Note that ($\alpha$) is proved in \cite{ns1} and
\cite[Prop.~15.8]{ns2}.

Fix $p$ and $q$, we shall show that
$\kappa_{p,q}(c^{(\epsilon_1)}, c^{(\epsilon_2)}, \dots,
c^{(\epsilon_{p+q})})$ is necessarily 0 unless $p$ and $q$
are even and $\epsilon_i = -\epsilon_{i+1}$ for $1 \leq i <
p$ and  $p+1 \leq i < p+q$. To expand this cumulant we
use the formula for cumulants with products as entries as
written in Proposition \ref{mst},
\begin{eqnarray*}\lefteqn{%
\kappa_{p,q}( (ab)^{(\epsilon_1)}, (ab)^{(\epsilon_2)},
\dots, (ab)^{(\epsilon_{p+q})}) } \\
& = &
\sum_{\pi\in S_\NC(2p,2q)}
\kappa_\pi(x_1, x_2, \dots, x_{2(p+q)-1}, x_{2p+2q})
\end{eqnarray*}
\[
+ \mathop{\sum_{\pi \in \NC(2p) \times \NC(2q)}}_{\pi =
  \pi_1 \times \pi_2} \mathop{\sum_{\cV \geq \pi}}_{|\cV| =
  |\pi| + 1} \kappa_{(\cV, \pi)} (x_1, x_2, \dots,
x_{2(p+q)-1}, x_{2p+2q}),
\]
where 
\[
x_{2i-1} = \begin{cases}a&\epsilon_i=1\\
                        b^*&\epsilon_i=-1
            \end{cases}
\ \mbox{and}\            
x_{2i} =   \begin{cases}a^*&\epsilon_i=-1\\
                        b&\epsilon_i=1
            \end{cases} ,
\]
and the sum is over all $\pi \in S_{NC}(2p+2q)$ in the first
sum, and $\NC(2p) \times \NC(2q)$ in the second, such that
${\gamma_{2p,2q}\pi^{-1}}|_{_O} = \id_O$ and $O = \{1, 3, 5,
\dots,\ab 2(p+q)-1\}$. Moreover, in the second sum we require
the partition $\cV$ to have one block which is the union of
a cycle of $\pi_1$ and a cycle of $\pi_2$.

Since we have assumed that $\{a, a^*\}$ and $\{b, b^*\}$ are
second order free we know that both
$\kappa_{(\cV,\pi)}(x_1,\dots,x_{2p+2q})$ and
$\kappa_\pi(x_1,\dots, , x_{2p+2q}) = 0$ unless
\begin{enumerate}

\item
all cycles of $\pi$ visit either only elements of $\{a,
a^*\}$, let's call these $a$-cycles, or only elements of
$\{b, b^*\}$, let's call these $b$-cycles; moreover, since
$a$ is $R$-diagonal, the $a$-cycles must alternate between
$a$ and $a^*$ and, in particular, must have an even number
of elements (see Figure \ref{figure:cross}) and

\item
${\gamma_{2p,2q}\pi^{-1}}|_{O} = \id_O$ (c.f. Definition
  \ref{def:separates_points}).

\end{enumerate}

\begin{figure}
  \begin{minipage}[c]{0.3\textwidth}
    \includegraphics{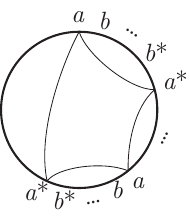}
  \end{minipage}\hfill
  \begin{minipage}[c]{0.7\textwidth}
    \caption{\small If $\kappa_\pi(x_1, \dots, x_n) \not=0$ then
  $\pi$ must consist of $a$ cycles and $b$ cycles and the
  $a$ cycles must alternate between $a$ and
  $a^*$. \label{figure:cross}}
  \end{minipage}
\end{figure}

\begin{lemma}\label{lemma:first}
Suppose $\pi \in S_\NC(2p, 2q) \cup \NC(2p) \times \NC(2q)$
satisfies $(i)$ and $(ii)$. If $\epsilon_i = -1$, then 
\[
\pi(2i)
= \gamma_{2p,2q}(2i) \mbox{\ and\ } 
\epsilon_{\gamma_{p,q}(i)} = 1.
\]
\end{lemma}

\begin{proof}
If $\epsilon_i = -1$ then $x_{2i-1} =b^*$ and $x_{2i} =
a^*$. Thus $\pi(2i)$ must be odd so that the next element in
the cycle containing $2i$ will be an $a$. Let $2j-1 =
\pi(2i)$. Then $\pi^{-1}(2j-1) = 2i$, so
$\gamma_{2p,2q}\pi^{-1}(2j-1)$ is odd. Since ${\gamma_{2p,
  2q}\pi^{-1}}|_{O}= \id_O$ we must have that
$\gamma_{2p,2q}\pi^{-1}(2j-1) = 2j-1$, so $\pi(2i) = 2j-1 =
\gamma_{2p,2q}\pi^{-1}(\pi(2i)) = \gamma_{2p,2q}(2i)$. Hence
$x_{\gamma_{2p,2q}(2i)} = a$ and so
$\epsilon_{\gamma_{p,q}(i)} = 1$.
\end{proof}

\begin{proof}[Proof of Theorem \ref{MT1}]
To prove ($\beta$) we have to show that if there is $i$
such that $\epsilon_i = \epsilon_{i+1}$ then no $\pi$ can
simultaneously satisfy  conditions (\textit{i}) and
(\textit{ii}).

Lemma \ref{lemma:first} already says that $\epsilon_i =
\epsilon_{i+1} = -1$ is impossible.  Suppose that $\epsilon_i =
\epsilon_{i+1}= 1$. Then $x_{2i+1} = a$ and so
$\pi^{-1}(2i+1) = 2j$ for some $j$ and $\epsilon_j = -1$. By
Lemma \ref{lemma:first}, $\gamma_{2p,2q}(2i) = \pi(2j) =
\gamma_{2p,2q}(2j)$; so $i = j$, but $\epsilon_j = -1$ and
$\epsilon_i = 1$. As this is impossible, we cannot have $\epsilon_i =
\epsilon_{\gamma_{p,q}(i)}$.

The arguments above show that $\epsilon_i =- \epsilon_{i+1}
$, for $1 \leq i <p$ or $p+1 \leq i <p +q$, and that
$\epsilon_p=-\epsilon_{1}$ and
$\epsilon_{p+q}=-\epsilon_{p+1}$. This implies that $p$ and
$q$ must be even and thus we have proved ($\beta$) as
desired.
\end{proof}

A \textit{second order Haar unitary} is the second order
limiting distribution of a Haar distributed random unitary
matrix.  In Proposition
\ref{prop:r-diagonality-haar-unitary} we show that a second
order Haar unitary is second order $R$-diagonal; see Section
\ref{exa:haar}.  Thus we get the following
corollary. Furthermore, in Proposition \ref{R diagonal to
  even}, part (\textit{iii}) we will see that any
$R$-diagonal element can be realized in this way.

\begin{corollary}
Let $u$ and $b$ be elements in some second order
$*$-prob\-a\-bil\-i\-ty space such that $u$ is a second
order Haar unitary and such that u and b are second order
$*$-free.  Then $ub$ is a R-diagonal.
\end{corollary}

\subsection{Cumulants of $aa^*$ for an $R$-diagonal Operator}

Now we prove Theorem \ref{MT2}, which allows us, to
calculate the cumulants of a second order $R$-diagonal
element $a$ with determining sequences
$(\beta_{n})_{n\geq1}$ and $(\beta_{p,q})_{p,q\geq1}$ by the
formula

\begin{equation}\label{eq:pf-of-13}
\kappa_{p,q}(aa^*,\dots, aa^*)=\sum_{\pi \in
  \cPS_{NC}(p,q)}\beta_\pi.
\end{equation}

\begin{proof}[Proof of Theorem \ref{MT2}]
We use the formula for cumulants with products as arguments
as in Proposition \ref{mst}. Indeed,
\begin{eqnarray}\label{eq:thm-13}
\kappa_{p,q}(aa^*,\dots,aa^*)
&=&
\sum_{\substack{{(\mathcal{V},\pi)} \in \cPS_{NC}(2p,2q) \\
\gamma_{2p,2q}\pi^{-1} \textrm{ sep. } O }}
\kappa_{(\mathcal{V},\pi)}(a,\dots,a^*).
\end{eqnarray}

Let $\cX = \{ (\cV, \pi) \in \cPS_{NC}(2p, 2q) \mid
\gamma_{2p,2q}\pi^{-1}$ separates the points of $O$ and
$\kappa_{(\cV, \pi)}(a, a^*, \dots, a, a^*) \not = 0\}$.  Below, we
shall define a bijection $\psi: \cX \rightarrow \cPS_{NC}(p,
q)$ such that $\kappa_{(\cV, \pi)}(a, \dots, a^*) =
\beta_{\psi(\cV, \pi)}$. This will show that right hand
sides of Equations (\ref{eq:pf-of-13}) and (\ref{eq:thm-13})
are the same and thus prove the theorem.

Now, since $a$ is $R$-diagonal, the only non-vanishing
cumulants $\kappa_{(\mathcal{V},\pi)} $ in the sum above
have entries which alternate between $a$ and $a^*$
and in particular each block of $\pi$ must have an even
number of elements.

We have two terms on the right hand side of Equation
(\ref{eq:thm-13}). The first term is when $\pi$ is a
permutation in $S^-_{NC}(2p,2q)$, then by Lemma \ref{lemma
  check} there is a unique element $\sigma \in S_{NC}(p, q)$
such that $\hat \sigma = \pi$. We set $\psi(\pi) = \sigma$
and then by construction $\beta_\sigma =
\kappa_{\psi(\sigma)}(a, a^*, \dots, a, a^*)$. If $\pi \in
S_{NC}(2p, 2q) \setminus S_{NC}^-(2p, 2q)$ we have
\[
\kappa_\pi(a, \dots, a^*) = 0.
\]

In the second case, $(\mathcal{V},\pi)$ is in
$\cPS_{NC}(2p,2q)'$; by Proposition \ref{prop-even3} there is $(
\mathcal{\check V},\check \pi) \in \cPS_{NC}(p,q)'$ such
that $\beta_{( \mathcal{\check V},\check \pi)}=
\kappa_{(\mathcal{V},\pi)}(a,a^*,...,a,a^*)$.  So we set
$\psi(\cV, \pi) = (\check\cV, \check\pi)$ and then
\begin{eqnarray} \label{eq222}
\sum_{\substack{{(\mathcal{V},\pi)} \in \cPS_{NC}(2p,2q)
    \\ \pi^{-1} \gamma_{2p,2q} \textrm{ sep.'s } O }}
\kappa_{(\mathcal{V},\pi)}(a,\dots,a^*)
=\sum_{(\mathcal{V},\pi) \in \cPS_{NC}(p,q)}
\beta_{(\mathcal{V},\pi)}.
\end{eqnarray}
as desired.
\end{proof}

\begin{remark}
Note that Equation (\ref{eq222}) can be inverted. So that given
a second order $R$-diagonal element, $a$, the
$*$-cumulants of $a$ can be recovered from the cumulants of
$aa^*$.
\end{remark}


\section{Even Elements}\label{sec:even_elements}

\subsection{Squares of even operators}

Let $(A,\phi,\phi_2) $ be a second order non-commutative
probability space. Recall that an element $x \in \cA$ is
called \textit{even} if $\phi(x^{2n+1})=0$ for all $n\geq0$ and \ab
$\phi_2(x^p,x^q) \ab = 0$ unless both $p$ and $q$ are even.
Recall also that for an even operator $x$, we set 
\[
\beta_n =
\kappa_{2n}^{(x)} \mathrm{\ and\ } 
\beta_{p,q} = \kappa_{2p,2q}^{(x)} + \ 
\ds\sum_{\mathclap{\pi \in \sncaplus{2p}{2q}}} \ \kappa_\pi^{(x)}
\] to be the determining sequences of $x$.

Now suppose we are given $(\cU, \sigma) \in
\cPS_\NC(p,q)'$. Write $\sigma = \sigma_1 \times \sigma_2
\in NC(p) \times NC(q)$. Let $U$ be the block of $\cU$ which
is the union of a cycle of $\sigma_1$ and a cycle of
$\sigma_2$. Recall from Notation \ref{notation-grouped} that
$S_\NC^{(\cU, \sigma)}(2p, 2q)$ is the set of $\pi \in
S_\NC^+(2p, 2q)$ such that

\begin{itemize}

\item 
each non-through cycle of $\pi$ is the double of some cycle
$c$ of $\sigma$;

\item 
the union of all the through cycles of $\pi$ is $\hat U$.

\end{itemize}

\begin{lemma}

Let $x$ be an even element. Then
\[
\mathop{\sum_{\pi \in S_\NC^+(2p, 2q)}}%
_{\gamma_{2p, 2q} \pi^{-1} \sep O}
\kappa_\pi^{(x)} 
=
\sum_{(\cU, \sigma) \in \cPS(p,q)'}
\sum_{\pi \in S_\NC^{(\cU, \sigma)}(2p, 2q)} \kappa_\pi^{(x)}.
\]
\end{lemma}

\begin{proof}
By Proposition \ref{prop-even1} the sets being summed over
are the same and thus the sums are equal.
\end{proof}

\begin{lemma} \label{lem-even4}
Let $x$ be second order even and $(\cU, \sigma) \in
\cPS_\NC(p,q)'$. Then
\[
\beta^{(x)}_{(\cU, \sigma)} = \kappa_{(\hat\cU, \hat\sigma)}^{(x)}
+
\sum_{\pi \in S_\NC^{(\cU, \sigma)}(2p, 2q)} \kappa_\pi^{(x)}.
\]

\end{lemma}

\begin{proof}

Let the cycles of $\sigma$ be $c_1, c_2, \dots, c_k$ with
$c_{k-1} \subset [p]$, $c_k \subset [p+1, p+q]$. We shall
also write $\cU = \{c_1, \dots, c_{k-2}, c_{k-1} \cup
c_k\}$. For this proof we shall write $\beta_{c_i}$ for
$\beta^{(x)}_{k_i}$ if the cycle $c_i$ has $k_i$
elements. Then $\beta_{c_i} = \kappa_{\hat c_i}^{(x)}$; so
\[
\beta_{c_1} \dots \beta_{c_{k-2}} =
\kappa_{\hat c_1}^{(x)} \cdots \kappa_{\hat c_{k-2}}^{(x)}.
\]
Also $\beta_{c_{k-1}, c_k} = \kappa_{\hat c_{k-1},
\hat c_k}^{(x)} 
+ 
\ds\sum_{\pi \in \sncaplus{\hat c_{k-1}}{\hat c_k}} 
\kappa_{\pi}^{(x)}$. Hence
\begin{align*}
\beta^{(x)}_{(\cU, \sigma)}
& = 
\beta_{c_1} \beta_{c_2} \cdots \beta_{c_{k-2}}
\beta_{c_{k-1}, c_k} \\
& = 
\kappa_{\hat c_1}^{(x)} \kappa_{\hat c_2}^{(x)} 
\cdots \kappa_{\hat c_{k-2}}^{(x)}
\Big(\kappa_{\hat c_{k-1}, \hat c_k}^{(x)}
+
\sum_{\pi \in \sncaplus{\hat c_{k-1}}{\hat c_k}}
\kappa_{\pi}^{(x)}\Big) \\
& = 
\kappa_{(\hat\cU, \hat\sigma)}^{(x)} +
\sum_{\pi \in S_\NC^{(\cU, \sigma)}(2 p, 2 q)} 
\kappa_\pi^{(x)}.
\end{align*}
\end{proof}

\begin{lemma}\label{proposition:parity-reversing2}
Let $x$ be a even element and $\beta_n =
\kappa_{2n}^{(x)}$. Then
\[
\mathop{\sum_{\pi \in S_\NC^-(2p, 2q)}}
_{\gamma_{2p, 2q} \pi^{-1} \sep O}
\kappa_\pi^{(x)} 
=
\sum_{\sigma \in S_\NC(p,q)} \beta_{\sigma}.
\]
\end{lemma}

\begin{proof}
The $\pi$'s over which we must sum are the ones that are in
bijection with $S_\NC(p,q)$, by Lemma \ref{lemma
  check}. Moreover a cycle of $\pi$ which must have even
length, corresponds to a cycle of $\check\pi$ of half this
length. For such a $\pi$ we then have $\kappa_\pi^{(x)} =
\beta_{\check\pi}$. This proves the lemma.
\end{proof}

Finally we are able to prove Theorem \ref{MT3} which states
that the second order cumulants of $x^2$ are given by
\begin{equation*}
\kappa_{p,q}(x^2,\dots,x^2)=\sum_{(\mathcal{V},\pi) \in
  \cPS_{NC}(p,q)}\beta^{(x)}_{(\mathcal{V},\pi)}.
\end{equation*}

\begin{proof}[Proof of Theorem \ref{MT3}]

On one hand, Lemma
\ref{proposition:parity-reversing2} says

\[
\mathop{\sum_{\pi \in S_\NC^-(2p, 2q)}}
_{\gamma_{2p, 2q} \pi^{-1} \sep O}
\kappa_\pi^{(x)} 
=
\sum_{\pi \in S_\NC(p,q)} \beta^{(x)}_{\pi},
\]
and on the other we have by using respectively Lemma
\ref{lem-even2} and Proposition \ref{prop-even1} that
\begin{eqnarray*} \lefteqn{%
\mathop{\sum_{(\cV, \pi) \in \cPS_\NC(2p, 2q)'}}
_{\gamma_{2p, 2q} \pi^{-1} \sep O} \kappa_{(\cV, \pi)}^{(x)}
+ \mathop{ \sum_{\pi \in S_\NC^+(2p, 2q)}} _{\gamma_{2p, 2q}
  \pi^{-1} \sep O} \kappa_\pi^{(x)}} \\
  & = &
  \sum_{(\cU, \sigma)\in \cPS_\NC(p,q)'} \bigg\{
  \kappa_{(\hat\cU, \hat\sigma)}^{(x)} + \kern-1em \sum_{\pi
    \in S_\NC^{(\cU, \sigma)}(2p,2q)} \kern-1em
  \kappa_{\pi}^{(x)} \bigg\} = \sum_{(\cU, \sigma) \in
    \cPS_\NC(p,q)'} \beta^{(x)}_{(\cU, \sigma)},
\end{eqnarray*}
where the last equality follows from Lemma \ref{lem-even4}.

From which we get using the formula for cumulants with
products as arguments and separating the sum into parts over
$S_\NC^-(2p, 2q)$, $S_\NC^+(2p, 2q)$, and $\cPS_\NC(2p,
2q)'$

\begin{align*}
\kappa_{p,q}^{(x^2)}
& =
\mathop{\sum_{(\cV, \pi) \in \cPS_\NC(2p, 2q)}}_
{\gamma_{2p, 2q}\pi^{-1} \sep O} \kappa_{(\cV, \pi)}^{(x)} \\
&=
\mathop{\sum_{\pi \in S_\NC(2p, 2q)}}_
{\gamma_{2p, 2q}\pi^{-1} \sep O} \kappa_\pi^{(x)} +
\mathop{\sum_{(\cV, \pi) \in \cPS_\NC(2p, 2q)'}}_
{\gamma_{2p, 2q}\pi^{-1} \sep O} \kappa_{(\cV, \pi)}^{(x)} \\
&=\mathop{\sum_{\pi \in S_\NC^-(2p, 2q)}}_{\gamma_{2p, 2q}\pi^{-1} \sep O} 
\kappa_\pi^{(x)} +
\mathop{\sum_{\pi \in S_\NC^+(2p, 2q)}}_
{\gamma_{2p, 2q}\pi^{-1} \sep O} \kappa_\pi^{(x)} +
\mathop{\sum_{(\cV, \pi) \in \cPS_\NC(2p, 2q)'}}_
{\gamma_{2p, 2q}\pi^{-1} \sep O} \kappa_{(\cV, \pi)}^{(x)} \\
&\stackrel{(*)}{=}
\sum_{\pi \in S_\NC(p,q)} \beta^{(x)}_{\pi}+\sum_{(\cU, \sigma) \in \cPS_\NC(p,q)'} \beta^{(x)}_{(\cU, \sigma)} \\
&=\sum_{(\cU, \sigma) \in \cPS_\NC(p,q)} \beta^{(x)}_{(\cU, \sigma)} 
\end {align*}
as desired, where the equality $(*)$ follows from Lemma
\ref{proposition:parity-reversing2}.
\end{proof}

\subsection{From $R$-diagonal to even operators}

Let $(\mathcal{A},\phi,\phi_2)$ be a non-commutative
probability space, and let $d$ be a positive
integer. Consider the algebra $M_d(\mathcal{A})$ of $d\times
d$ matrices over $\mathcal{A}$ and the functionals
$\tilde\phi:= \phi \circ \tr: M_d(\mathcal{A})
\longrightarrow \bC$, and $\tilde\phi_2:=\phi_2 \circ \tr
\otimes \tr$ on $M_d(\mathcal{A}): M_d(\mathcal{A}) \otimes
M_d(\mathcal{A}) \longrightarrow \bC$ defined by the formulas
\begin{equation*}
\tilde\phi((a_{i,j})^d_{i,j=1})=\frac{1}{d}\sum^d_{i=1}\phi(a_{ii})
= \phi(\tr(A))\end{equation*}
and
\begin{equation*}
\tilde\phi_2((a_{i,j})^d_{i,j=1},
(b_{i,j})^d_{i,j=1})=\frac{1}{d^2}\sum^d_{i,j=1}\phi_2(a_{ii},
b_{jj}) = \phi_2(\tr(A), \tr(B)). 
\end{equation*}

Then $(M_d(\mathcal{A}),\tilde\phi,\tilde\phi_2)$ is itself
a second order non-commutative probability space. Note that
if either $(a_{ij})_{i,j=1}^d$ or $(b_{ij})_{i,j=1}^d$ is
zero on the diagonal then 
\[
\tilde\phi_2( (a_{ij})_{i,j=1}^d,
(b_{ij})_{i,j=1}^d ) = 0.
\]

We can realize second order even elements as $2\times2$
matrices with $R$-diagonal elements in the off-diagonal
entries. This generalizes the theorem in \cite[Prop.~15.12]{ns2}
where the first order case was considered.
\begin{proposition}\label{R diagonal to even} 
Let $a$ be second order $R$-diagonal in $(\cA,\phi, \phi_2)$ and
consider the off-diagonal matrix,
\[ A := \left( \begin{array}{cc}
   0  &     a       \\
  a^*     &    0   \end{array} \right)\]
as an element in $(M_2(\mathcal{A}), \tilde\phi, \tilde\phi_2)$. Then

\begin{enumerate}

\item
$A$ is second order even element; 

\item
$A$ has the same determining sequence as $a$: 
$\beta_{p,q}^{(A)}=\beta_{p,q}^{(a)}$ for all 
$p,q\geq1$ and $\beta_{n}^{(A)}=\beta_{n}^{(a)}$
for all $n$;

\item
any second order $R$-diagonal operator can be realized
in the form $xu$ where $u$ is a second order Haar unitary, $x$ is second order even,
and $u$ and $x$ are second order $*$-free;

\item
The second order cumulants of $a$ and the second order
cumulants of $A$, are related
by $$\kappa_{p,q}^{(a)}=\kappa_{p,q}^{(A)}+ \sum_{\pi \in
  \sncaplus{m}{n}} \kappa_\pi^{(A)}.$$

\end{enumerate}
\end{proposition}

\begin{proof}
$A$ is even because $\tr(A^p) = 0$ for $p$ odd.  On the
  first order level claims (\textit{i})$-$(\textit{iv}) are
  proved in \cite[Prop.~15.12]{ns2}, so we shall only prove
  the corresponding claims at the second order level.  For
  (\textit{ii}), the only observation needed is that from
  Theorems \ref{MT2} and \ref{MT3} the determining sequence
  of $A$ and $a$ can be calculated from the moments and
  fluctuation moments of $aa^*$ and $A^2$ in the same way,
  (see Equations (\ref{formula:main1}) and
  (\ref{formula:main2})). So it is enough to see that the
  moments and fluctuation moments of $aa^*$ and $A^2$ are
  the same, i.e. $\tilde\phi_2(A^{2m}, A^{2n}) = \phi_2(
  (aa^*)^m, (aa^*)^n)$, but this is clear from the definition
  of $A$.

(\textit{iii}) is proved as follows. Let $U$ be a second
  order Haar unitary which is second order free from $A$. By
  Theorem \ref{MT1} and Proposition
  \ref{prop:r-diagonality-haar-unitary} (below), $X=AU^*$ is
  $R$-diagonal and $XX^*=AU^*UA=A^2$. This then implies that
  the determining series of $X$ coincides with the
  determining series of $A$ and also with the determining
  series of $a$, by (\textit{ii}) . So $X$ and $a$ are two
  second order $R$-diagonal operators with the same
  distribution.
  
(\textit{iv}) By (\textit{iii}) the determining sequences
  coincide, i.e. $\beta_{p,q}^{(A)}=\beta_{p,q}^{(a)}$ for
  all $p,q\geq1$ and $\beta_{n}^{(A)}=\beta_{n}^{(a)}$ for all $n$. This
  gives us exactly the desired relation.

\end{proof}


\section{Cumulants of products of free random variables}
\label{sec:products_free_random_variables}
In this section we will prove the second order analogue of
the theorem of Arizmendi and Vargas \cite{av} for the
moments and cumulants of $a_1a_2\cdots a_n$ when $a_i$'s are
second order free. We will use the formula for products as
arguments.

\begin{definition}\label{def:divisibility} 
Let a $\pi$ be a non-crossing permutation in $S_{NC}(kp,kq)$.
\begin{enumerate}
\item
$\pi$ is called $k$-\textit{divisible} if the size of every
  cycle of $\pi$ is a multiple of $k$.

\item
$\pi$ is called $k$-\textit{equal} if the size of every
  cycle of $\pi$ is $k$.

\item
$\pi$ is called $k$-\textit{alternating} if $\pi(i)\equiv
  i+1$ $\pmod k$. $\snckalt{kp}{kq}$ is the
  set of $k$-alternating permutations. 

\item  
  $\snckea{kp}{kq}$ is the
  set of $k$-equal and $k$-alternating permutations

\item
$\pi$ is called $k$-\textit{preserving} if $\pi(i) \equiv i$
  $\pmod k$.

\item
We say that $\pi$ is $k$-\textit{completing}
if $\pi$ is $k$-preserving and $\pi^{-1}\gamma_{kp,kq}$
separates the points $\{k, 2k, \ab \dots, (p+q)k\}$. 

\end{enumerate}
\end{definition}

Note that if $\pi$ is $k$-alternating then it is also
$k$-divisible. The converse is not true but there is a
$k$-to-$1$ correspondence from $k$-divisible to
$k$-alternating partitions, giving by a relabelling in one
of the circles.

From now on we will denote by
\[
\gamma=\gamma_{kp,kq}=(1,2, \dots, kp)(kp+1, \dots, kp\ab +kq)
\]
and denote by $\Kr(\sigma)=\sigma^{-1}\gamma$.  The following is
the analogue of Proposition 3.1 in \cite{av}.
\begin{lemma}\label{lem alternating}
 Let  $\sigma$ be in $S_{NC}(kp,kq)$. 
\begin{enumerate}
\item
$\sigma$ is $k$-alternating iff $\sigma^{-1}\gamma$ is
  $k$-preserving.

\item
$\sigma$ is $k$-alternating and $k$-equal iff
  $\sigma^{-1}\gamma$ is $k$-completing.
\end{enumerate}
\end{lemma}
\begin{proof} 
(\textit{i}) Suppose that $\sigma$ is $k$-alternating. Since
  $\gamma(i) \equiv i+1$ $\pmod k$ then
  $$\sigma^{-1}\gamma(i)=\sigma^{-1}(i+1) \equiv i \pmod
  k.$$ Conversely, if $\sigma^{-1}\gamma(i) \equiv i$ $\pmod k$ then
  $\sigma(i)=\gamma (i) \equiv i+1$ $\pmod k$.

 (\textit{ii}) Note that by (\textit{i}) if $\sigma^{-1}\gamma$ is
  $k$-completing then $\sigma$ is $k$-alternating. We claim
  that $\sigma$ is $k$-equal. Suppose that this is not the
  case. That is, suppose that $\sigma$ is $k$-alternating,
  but not $k$-equal. Then $\sigma$ has at most $p+q-1$
  cycles. But $(\sigma^{-1}\gamma)^{-1}\gamma=
  \gamma^{-1}\sigma \gamma$ has the same number of cycles as
  $\sigma$. Thus $(\sigma^{-1}\gamma)^{-1}\gamma$ cannot
  separate more than $p+q-1$ points from
  $\{k,2k,..,(p+q)k\}$. Thus $\sigma^{-1}\gamma$ is not
  $k$-completing, yielding a contradiction.

Conversely if $\sigma$ is $k$-alternating and
$k$-equal. Then $\sigma^k$ is the identity permutation. Let
$K = \{k, 2k, 3k, \dots, (p+q)k\}$. Then
$\sigma(\gamma^i(K)) = \gamma^{i+1}(K)$ for $1 \leq i \leq
k$. To show that $\sigma^{-1}\gamma$ is $k$-completing we
must show that $$\gamma^{-1}\sigma^{-r}\gamma(K) \cap K =
\nO \mathrm{\ for\ } 1 \leq r < k.$$ Since $\sigma(\gamma^i(K)) =
\gamma^{i+1}(K)$ for $1 \leq i \leq k$ we have
$\gamma^{-1}\sigma^{-r}\gamma(K) \cap K = \gamma^{-r}(K)
\cap K = \nO$ for $1 \leq i \leq k$.
\end{proof}

Now we are in position to prove Theorem \ref{products}.

\begin{theorem}[Moments and Cumulants of Products of Free Variables]
\label{thm:Moments_and_Cumulants_of_Products_of_Free_Variables}
Let $a_1, \dots, a_k$ be operators which are second order
free and such that $\kappa_{p,q}(a_i)=0$ for all
$p,q\in\mathbb{N}$.  Let $a=a_1a_2\cdots a_k$. Then
\begin{equation}\label{11}
\phi_2(a^p,a^q)=\sum_{\pi \in \snckalt{kp}{kq}}
\kappa_{\Kr(\pi)}(a_1,a_2,\dots,a_k,\dots,a_1,a_2,\dots,a_k).
\end{equation}
Furthermore,
\begin{equation} \label{12}
\kappa_{p,q}(a,\dots,a)=\sum_{\pi \in \snckea{kp}{kq}}
\kappa_{\Kr(\pi)}(a_1,\dots,a_k,\dots,a_1,\dots,a_k).
\end{equation}
\end{theorem}

\begin{proof}

We shall use the moment-cumulant formula (\S
\ref{subsec:second_order_cumulants}) and the second order
freeness of the $a_i $'s. By hypothesis, the second order
cumulants of $a_i$'s are zero. In addition we have further
assumed that the $a_i$'s are second order free. Thus for all
partitioned permutations $(\cV, \pi)$ we have that
$\kappa_{(\cV, \pi)}(a_1, \dots, a_k, \dots, a_1, \dots,
a_k) = 0$.  Thus in the moment-cumulant formula we only need to
sum over $S_{NC}(kp,kq)$:
\begin{eqnarray*}
\phi(a^p,a^q)=\sum_{\pi \in
  S_{NC}(kp,kq)}\kappa_\pi(a_1,a_2,\dots,a_k,\dots,a_1,a_2,\dots,a_k).
\end{eqnarray*}
The contribution of a $\pi$ will be non-zero only if $\pi$
joins an $a_i$ with another $a_i$, which implies that $\pi$
is $k$-preserving. Let $\sigma = \gamma \pi^{-1}$; then $\Kr(\sigma)=\pi$ and by Lemma \ref{lem alternating} $(i)$
 $\sigma$ is $k$-alternating. This proves the first
claim.

For the second formula we use the formula for products as
arguments as in Proposition \ref{mst}, for
$a=a_1a_2\cdots a_k$. 
Again, since the second order
cumulants of $a_i's$ are zero the only non-vanishing
contribution are from $\pi$ in $S_{NC}(kp,kq)$. Thus we get
\[
\kappa_{p,q}(a,...,a)=\sum_{\pi\in
  S_{NC}(kp,kq)}\kern-1em \kappa_\pi
  (a_1,...,a_{k},...,a_1,...,a_{k}),
\]
where the sum runs over the permutations $\pi$ in
$S_{NC}(kp,kq)$ such that $\pi^{-1}\gamma_{kp,kq}$ separates
the points $\{k,2k,..,(p+q)k\}$.  Since the random variables
are second order free, the sum runs only over the
$k$-preserving partitions in $S_{NC}(kp,kq)$. These two
conditions on $\pi$ mean exactly that $\pi$ is
$k$-completing. Finally, by Lemma \ref{lem alternating} the
permutations involved in the sum are exactly the Kreweras
complements of $k$-equal permutations which are also
$k$-alternating, and the formula follows.
\end{proof}

Note that by Lemma \ref{lem alternating}  we can write the set 
\[
\{ \Kr(\pi) \mid \pi \in S_{NC}(kp, kq), \pi \textrm{\ is\ }
k\textrm{-equal and\ } k\textrm{-alternating} \}
\]
 as 
\[
\{ \sigma \in S_{NC}(kp,kq) \mid \sigma \textrm{\ is\ }
k\textrm{-preserving and\ } \Kr(\sigma) \textrm{\ is\ }
k\textrm{-equal\ }\}.
\]
Of particular interest will be when $\kappa_n(a_i)=1$ for
all $i$ and all $n$.  i.e. when the $a_i$'s are free Poisson
operators. This will be explained in Example
\ref{prodpoiss}.

\begin{remark}
There is a natural bijection between $k$ divisible
partitions $(k+1)$-equal ones for $NC(n)$ which is given by
the ``fattening procedure'' as described in Arizmendi
\cite{ariz}. The second order analogue is also true, that 
is,  we have  $\snckalt{kp}{kq}$ is in bijection with
$\sncjea{jp}{jq}$, for $j = k+1$; we leave this observation as an
exercise for the reader. 
\end{remark}


\section{Examples and applications}\label{sec:examples}

We now give some examples on how to use the main
theorems and constructions given above.

\subsection{A second order Haar unitary is second order $R$-diagonal}

Our first example comes from \cite[\S 5]{mst}.
\label{exa:haar}

\begin{definition}\label{def:second_order_haar_unitary}
Let $(\mathcal{A},\phi,\phi_2)$ be a second order
probability space and $u\in\mathcal{A}$ a unitary. We say
that $u$ is a \textit{second order Haar unitary} if
$\phi(u^k)=\phi(u^{*k})=0$ for $k\geq1$ (i.e. $u$ is a Haar
unitary) and $\phi_2(u^k,u^{*l})=k\, \delta_{k,l}.$
\end{definition}

Let us first see that a second order Haar unitary is second
order $R$-diagonal.

\begin{proposition}\label{prop:r-diagonality-haar-unitary}
Let $u$ a second order Haar unitary. Then $u$ is second
order $R$-diagonal. Moreover given $p$ and $q$ positive
integers and $\epsilon_1, \epsilon_2, \epsilon_3, \dots ,
\epsilon_{p+q} \in \{-1, 1\}$, we have
$\kappa_{p,q}(u^{\epsilon_1}, u^{\epsilon_2}, \dots,\ab
u^{\epsilon_{n-1}}, u^{\epsilon_n}) =0$ unless $p$ and $q$
are even and
\begin{equation}\label{p-part}
\epsilon_1 + \epsilon_2 = \cdots =
\epsilon_{p-1} + \epsilon_p =
\epsilon_{p+1} + \epsilon_{p+2} = \cdots =
\epsilon_{p+q-1} + \epsilon_{p+q} = 0
\end{equation}
i.e. the $\epsilon$'s alternate in sign except possibly
between $p$ and $p+1$.
\end{proposition}

\begin{proof}

Let us show first that if $\epsilon_1 + \cdots + \epsilon_{p
  + q} \not = 0 $ then $\kappa_{p,q}(u^{\epsilon_1},\ab
\dots, u^{\epsilon_{p+q}}) \ab = 0$. Suppose $\epsilon_1 +
\cdots + \epsilon_{p + q} \not = 0 $.  First note that for
all $\pi \in S_{NC}(p,q)$ some cycle of $\pi$, $(i_1, \dots,
i_k)$, must have $\epsilon_1 + \cdots + \epsilon_k \not = 0$
and thus $\phi_\pi(u^{\epsilon_1}, \dots, \ab
u^{\epsilon_{p+q}}) = 0$ because we assumed that $u$ is a
Haar unitary. Also for $(\cU, \pi) \in \cPS_{NC}(p,q)'$ we
have $\cU$ must have a block where the $\epsilon$'s do not
sum to $0$. If this block is a cycle, $(i_1, \dots, i_k)$,
of $\pi$ then $\phi(u^{\epsilon_{i_1}} \cdots
u^{\epsilon_{i_k}}) = 0$. If this block is the union of two
cycles $(i_1, \dots, i_k)$ and $(j_1, \dots, j_l)$ of $\pi$
then $\phi_2(u^{\epsilon_{i_1}} \cdots u^{\epsilon_{i_k}},
u^{\epsilon_{j_1}} \cdots\ab u^{\epsilon_{j_l}}) = 0$. In
either case we have $$\phi_{(\cU, \pi)}(u^{\epsilon_{1}},
\dots, u^{\epsilon_{p}}, u^{\epsilon_{p+1}}, \dots, \ab
u^{\epsilon_{p+q}}) = 0.$$ Since we have by
\cite[Def.~7.4]{cmss}
\begin{multline*}
\kappa_{p,q}(u^{\epsilon_{1}}, \dots, u^{\epsilon_{p}},
u^{\epsilon_{p+1}}, \dots, \ab u^{\epsilon_{p+q}}) \\
=
\sum_{(\cU, \pi) \in \cPS_{NC}(p,q)} 
\mu(\cV, \pi) \, \phi_{(\cU,
  \pi)}(u^{\epsilon_{1}}, \dots, u^{\epsilon_{p}},
u^{\epsilon_{p+1}}, \dots, \ab u^{\epsilon_{p+q}})
\end{multline*}
we have that $\kappa_{p,q}(u^{\epsilon_{1}}, \dots,
u^{\epsilon_{p}}, u^{\epsilon_{p+1}}, \dots, \ab
u^{\epsilon_{p+q}}) = 0$. From now on we assume that 
$\epsilon_1 + \cdots + \epsilon_{p+q} = 0$ and $p + q$
is even.

We shall prove the proposition by induction on $(p,
q)$. When $p = q = 1$, we have $$\kappa_{1,1}(u^{\epsilon_1},
u^{\epsilon_2}) = \phi_2(u^{\epsilon_1}, u^{\epsilon_2}) -
\kappa_2(u^{\epsilon_1}, u^{\epsilon_2}) = 0$$ because
$\phi_2(u^{\epsilon_1}, u^{\epsilon_2}) =
\kappa_2(u^{\epsilon_1}, u^{\epsilon_2}) = 1$.

When $p = q = 2$ we have to check that
$\kappa_{2,2}(u^{\epsilon_1}, u^{\epsilon_2},
u^{\epsilon_3}, u^{\epsilon_4}) = 0$ when either $\epsilon_1
= \epsilon_2$ or $\epsilon_3 = \epsilon_4$; in fact one
condition implies the other because $\epsilon_1 + \epsilon_2
+ \epsilon_3 + \epsilon_4 = 0$. Thus
\begin{eqnarray*}
2 &=& \phi_2(u^{\epsilon_1}
u^{\epsilon_2}, u^{\epsilon_3} u^{\epsilon_4}) = \sum_{\pi \in
S_{NC}(2,2)} \kappa_\pi(u^{\epsilon_1}, u^{\epsilon_2}, u^{\epsilon_3},
u^{\epsilon_4}) \\
&&\mbox{} + \sum_{(\cV, \pi) \in \cPS(p, q)'} \kappa_{(\cV, \pi)}
(u^{\epsilon_1}, u^{\epsilon_2}, u^{\epsilon_3}, u^{\epsilon_4})
\end{eqnarray*}
In the first term we will only have contributions from
permutations that have cycles of even length that alternate
between $u$ and $u^\ast$.  There are no permutations with
cycles of length four that alternate because we have two
$u$'s on one circle and two $u^*$'s on the other. There are
two permutations with two cycles of length two that
alternate between $u$ and $u^*$. Thus the contribution of
the first term is 2. The contribution of the second
term    is just $\kappa_{2,2}(u^{\epsilon_1},
u^{\epsilon_2}, u^{\epsilon_3}, u^{\epsilon_4})$ as all
other terms contain factors of $\kappa_{1,1}$ or
$\kappa_{2,1}$ which we have already shown to be 0. Hence
$\kappa_{2,2}(u^{\epsilon_1}, u^{\epsilon_2},
u^{\epsilon_3}, u^{\epsilon_4}) = 0$ when $\epsilon_1 =
\epsilon_2$ or $\epsilon_3 = \epsilon_4.$

Suppose that the proposition holds whenever $r < p$ and $q
\leq s$ or $r \leq p$ and $s < q$, and that either $p$ is odd,
$q$ is odd, or property (\ref{p-part}) fails. Then we can
find in either the first $p$ positions or in the last $q$
positions a cyclically adjacent pair of $\epsilon$'s of the
same sign. Recall that $\phi_2$ is tracial in each variable
so without loss of generality we may assume that $\epsilon_1
= 1$, $\epsilon_2 = -1$, and $\epsilon_3 = -1$.

Now $u^{\epsilon_1} u^{\epsilon_2} = 1$, so by
\cite[Prop.~7.8]{cmss}, $\kappa_{p-1,
  q}(u^{\epsilon_1} u^{\epsilon_2}, u^{\epsilon_3}, \dots
,\ab u^{\epsilon_{p+q}}) = 0$. Let $N = \{2, 3, 4, \dots,
p+q \}$.  By Proposition 5
\begin{eqnarray*}\label{inductivestep}\lefteqn{0=
\kappa_{p-1, q}(u^{\epsilon_1} u^{\epsilon_2},
u^{\epsilon_3}, \dots , u^{\epsilon_{p+q}})} \\
&=&
\mathop{\sum_{\pi \in S_{NC}(p,q)}}_%
{\pi^{-1}\gamma_{p,q}\ {\rm sep.\hbox{\tiny'} s}\ N}
\kappa_\pi(u^{\epsilon_1}, \dots, u^{\epsilon_{p+q}})\notag \\
&& \mbox{} +
\mathop{\sum_{(\cV, \pi) \in \cPS(p,q)'}}_%
{\pi^{-1}\gamma_{p,q}\ {\rm sep.\hbox{\tiny'} s}\ N}
\kappa_{(\cV, \pi)}(u^{\epsilon_1}, \dots, u^{\epsilon_{p+q}}) \notag
\end{eqnarray*}

Let us examine the condition on $\pi$ in each of the two
terms.  The requirement that $\pi^{-1} \gamma_{p,q}$
separates the points of $N$ means that either $\pi =
\gamma_{p,q}$ or $\pi = \gamma_{p,q}(1, r)$ for some $r \not
= 1$.

For the first term,  in order for
$\pi$ to connect the two cycles of $\gamma$, we must have $\pi =
\gamma_{p,q}(1, r)$ for some $r \in [p+1, p+q]$, . This means that
$\pi = (1, r+1, r+2, \dots , p+q, p+1, \dots , r, 2, 3, 4,
\dots, p)$ and thus the first order cumulant
$\kappa_\pi(u^{\epsilon_1}, \dots, u^{\epsilon_{p+q}})$ is 0
because $\epsilon_2$ and $\epsilon_3$ are adjacent.

In the second term we must have either $\pi = \gamma_{p,q}
(1, r)$ with $r \in [p]$ or $\pi = \gamma_{p,q}$. Let us
first suppose the former holds. Then $$\pi = (1, r+1, r+2,
\dots, p)(2, 3, 4, \dots , r)(p+1, p+2, \dots, p+q)$$ and
there are two possibilities for $\cV$. Either

\[\cV = \{ (1, r+1, r+2, \dots, p, p+1, \dots, p+q)(2, 3, \dots,
r) \}\] or
\[\cV = \{ (2, 3, \dots, r, p+1, \dots, p+q)(1, r+1, r+2, \dots, p)
\}.\]
In the first case
\begin{eqnarray*}
\kappa_{(\cV, \pi)}(u^{\epsilon_1}, \dots, u^{\epsilon_{p+q}})
&=&
\kappa_{r-1}(u^{\epsilon_2}, u^{\epsilon_3}, \dots,
u^{\epsilon_r}) \\
&& \mbox{}\times
\kappa_{p-r+1, q}(u^{\epsilon_1}, u^{\epsilon_{r+1}}, \dots,
u^{\epsilon_{p+q}}) = 0,
\end{eqnarray*}
because $\epsilon_2$ and $\epsilon_3$ are adjacent in the
first factor (unless $r=2$, but then we have a singleton so
we get 0 in this case also).

In the second case
\begin{eqnarray*}
\kappa_{(\cV, \pi)}(u^{\epsilon_1}, \dots, u^{\epsilon_{p+q}})
&=&
\kappa_{r-1,q}(u^{\epsilon_2}, u^{\epsilon_3}, \dots,
u^{\epsilon_r}, u^{\epsilon_{p+1}}, \dots, u^{\epsilon_{p+q}})
\\ && \mbox{} \times
\kappa_{p-r+1}(u^{\epsilon_1}, u^{\epsilon_{r+1}}, \dots,
u^{\epsilon_{p}}) = 0,
\end{eqnarray*}
by our induction hypothesis because $\epsilon_2$ and
$\epsilon_3$ are adjacent (unless $r = 2$ in which case
$r-1$ is odd and so we get 0 in this case also). We have now
shown that all the terms in Equation (\ref{inductivestep})
are zero except possibly the case where $\pi = \gamma_{p,q}$
and $\cV = 1_{p+q}$. Since the sum is 0 we must then also
have that this term is also 0, i.e.
\[
\kappa_{p,q}(u^{\epsilon_1}, \dots, u^{\epsilon_{p+q}}) =
\kappa_{(1_{p+q}, \gamma_{p,q})}
(u^{\epsilon_1}, \dots, u^{\epsilon_{p+q}}) =0 
\] as required. 
\end{proof}

\begin{notation}
We put
$$c_{m,n}:=\#(S_{NC}(m,n)) =  \frac{2mn}{m+n} \binom{2m-1}{m} \binom{2n-1}{n}.$$
\end{notation}

\begin{proposition}
Let $p=2m$ and $q=2n$ be even integers and
$\epsilon_1,...,\ab \epsilon_{p+q} \in \{-1,1\}$
alternating. Then
$\kappa_{p,q}(u^{\epsilon_1},...,u^{\epsilon_{p+q}})=(-1)^{m+n}c_{m,n}$.
\end{proposition}

\begin{proof}
We have just shown that $u$ is second order $R$-diagonal and
since $uu^*=1$ we can use Theorem \ref{MT2} to calculate the
cumulants of $u$. Indeed, the cumulants of $1$ are given by
$\kappa_1=1$ , $\kappa_n=0$ for $n>1$ and $\kappa_{p,q}=0$
for all $p,q\geq1$. Thus the determining sequence of $u$ is
given by the inversion formula \eqref{equation:beta_zeta} $\delta
= \beta^{(u)} * \zeta$ (see page
\pageref{equation:beta_zeta}). Hence $\beta^{(u)}$ is the M\"obius
function calculated in \cite[Thm.~5.24]{cmss}, so
$$\kappa_{p,q} (u^{\epsilon_1},...,u^{\epsilon_{p+q}})=
(-1)^{m+n}c_{m,n}.$$
\end{proof}

\begin{proposition}
Suppose that $r$ and $u$ are second order $*$-free and that
$u$ is a second order Haar unitary. Then $r$ is second order
$R$-diagonal if and only if $r$ and $ur$ have the same
second order $*$-distribution.
\end{proposition}

\begin{proof}
Suppose that $r$ is second order $R$-diagonal. We have
already seen in Proposition
\ref{prop:r-diagonality-haar-unitary} that $u$ is second
order $R$-diagonal. By Theorem \ref{MT1} we have that $ur$
is also second order $R$-diagonal. The second order
$*$-distributions of $r$ and $ur$ are given by their
$*$-cumulants and by Theorem \ref{MT2} these are given by
the cumulants of $rr^*$ and $urr^*u^*$ respectively. Since
both $\phi$ and $\phi_2$ are tracial, these cumulants are
equal. Thus $r$ and $ur$ have the same second order
$*$-distribution.

Conversely suppose that $r$ and $ur$ have the same second
order $*$-distribution. By Theorem \ref{MT1}, $ur$ is second
order $R$-diagonal. Thus $r$ is second order $R$-diagonal.
\end{proof}

\subsection{Products of Semicircular and circular operators}

Following \cite[\S 4]{mst}, a random variable $s$ in a second
order non-commutative probability space is called a
\textit{second order semicircular operator} if its first
order cumulants satisfy $\kappa_n(s,s,...,s)=0$ for all
$n\neq2$ and $\kappa_2(s,s)=1$, and for all $p$ and $q$ the
second order cumulants $\kappa_{p,q}$ are 0.  This operator
appears as the limit of GUE (Gaussian unitary ensemble) random matrices as the size
tends to infinity. See \cite[Thm.~3.5]{ms}.

\begin{example}[Square of Semicircular] \label{semicircle}
It follows from the definition that the first order
determining sequence of the semicircular operator is given
by $\beta_{1}=1$ and $\beta_{n}=0$ for $n>0$, The second
order determining sequence is given by $\beta_{k,k}=k$ and
$\beta_{p,q}=0$ if $p\neq q$, since the only terms appearing
in formula \eqref{determining 1} (see page
\pageref{determining 1}), are ``spoke'' diagrams (see
Fig. \ref{fig:spoke_diagram}) which are parity
preserving. Theorem \ref{MT3} now gives the second order
cumulants of $s^2$.
\begin{eqnarray*}\lefteqn{%
\kappa_{p,q}(s^2,\dots,s^2) }\\
& = &
\sum_{\pi \in S_{NC}(p,q)}\beta_\pi
+
\sum_{(\mathcal{V},\pi) \in \cPS_{NC}(p,q)'}\beta_{(\mathcal{V},\pi)}
= 
\sum_{\pi \in \cPS_{NC}(p,q)'}\beta_{(\mathcal{V},\pi)}
\end{eqnarray*}
because there are no partitions in $S_{NC}(p,q)$ with only
blocks of size one. Now, the only non-vanishing terms in the
last sum are of the form $\beta_{k,k}\beta_{1}^{p+q-2k}$.
Indeed, for $\beta_{(\cV, \pi)} \not = 0$ we must have $\pi
= \pi_1 \times \pi_2 \in NC(p) \times NC(q)$, such that both 
$\pi_1$ and $\pi_2$ have one cycle with $k$ elements and all
others singletons. Then $\cV$ will have one block that is the 
union of the cycle of $\pi_1$ with $k$ elements and the cycle
of $\pi_2$ which also has $k$ elements;   all other blocks
$\cV$ are singletons.
Moreover, there are $\binom{p}{k}\binom{q}{k}$ ways of choosing
the two cycles with $k$ elements. For such a $(\cV, \pi)$ we have
$\beta_{(\cV, \pi)} =  \beta_{k,k}\, \beta_{1}^{2p+2q-2k}=k$. Thus,
\begin{eqnarray*}
\kappa_{p,q}(s^2, \dots, s^2) &=& \sum_{(\mathcal{V},\pi)
  \in PS_{NC}(p,q)'}\beta_{(\mathcal{V},\pi)} =
\sum_{k>0}\binom{p}{k}\binom{q}{k}\beta_{k,k}
\\ &=&\sum_{k>0}k\binom{p}{k}\binom{q}{k}=p\binom{p+q-1}{p}.
\end{eqnarray*}
The last equality follows by the Chu-Vandermonde  formula.

\begin{figure}
\includegraphics{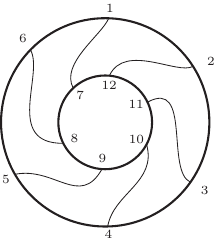}
\caption{\small\label{fig:spoke_diagram} A \textit{spoke
    diagram}, this means a non-crossing annular pairing in
  which all pairs connect the two cycles of
  $\gamma_{p,q}$. This example is \textit{parity preserving}
  because the connected points have the same parity.}
\end{figure}

Another way to derive these cumulants is using the series
from Theorem \ref{series xx}. Indeed, $B(z)=1$ implies that
$C(z)=1/(z-1)$ and then $C'(z)=-1/(z-1)^2.$ Also,
\begin{eqnarray*}
B(z,w)&=&\frac{1}{zw}\sum_{p,q\geq1}\beta_{p,q}
z^{p}w^{q}=\frac{1}{zw}\sum_{n\geq1}nz^{n}w^{n}=\frac{1}{(1-zw)^2}.
\end{eqnarray*}
and
\[
\frac{\partial^2}{\partial z\partial w}\log
\left(\frac{C(z)-C(w)}{z-w}\right)=\frac{\partial^2}{\partial
  z\partial w}\log \left(\frac{1}{(z-1)(w-1)}\right)=0.
\]
Hence, 
\begin{eqnarray*}
C(z,w)&=&C'(z)
C'(w)B(C(z),C(w))\\&=&\frac{1}{(z-1)^2(w-1)^2}
\frac{1}{(1-(z-1)^{-1}(w-1)^{-1})^2}\\&=&
\frac{1}{((z-1)(w-1)-1)^2} =\frac{1}{(zw-z-w)^2}     \\
&=&\sum_{p,q \geq 1}   p\binom{p+q-1}{p} z^{-(p+1)} w^{-(q+1)} 
\end{eqnarray*}
\end{example}

\begin{example}[Circular operator]\label{circle}
Consider $s_1$ and $s_2$ second order free semicircular
operators. We call $c=\frac{s_1+is_2}{\sqrt{2}}$ a (second
order) circular operator.  The operator $c$ is a second
order $R$-diagonal.  Indeed, since $s_1$ and $s_2$ are
second order free, their mixed free cumulants vanish, also
the second order cumulant of $s_1$ and $s_2$ vanish, thus by
linearity the same is true for $c$. That is,
$\kappa_{p,q}(c^{(\epsilon_1)},...,c^{(\epsilon_{p+q})})=0$ for
all $\epsilon_1,...,\epsilon_{p+q}\in\{\pm 1\}$.  On the other
hand, it is well known $\kappa_2(c,c^*)=\kappa_2(c^*,c)=1$ are
the only first order vanishing cumulants of $c$ (see \cite[Ex.~11.23]{ns2}). Hence $c$
is second order $R$-diagonal.
\end{example}

In \cite{dp}, G.~Dubach and Y.~Peled computed the
fluctuation moments of some words in second order $*$-free
second order circular operators. More precisely, let $c_1,
\dots, c_s$ be second order circular operators and second
order $*$-free. Dubach and Peled computed the fluctuation
moments of the form $\phi_2(w^k, (w^*)^l)$, where $w =
c_{i_1} \cdots c_{i_k}$ is a word in $c_1, \dots, c_s$. In
the next few examples, we will present the second order free
$*$-cumulants of $w$ and the second order free cumulants of
$w^*w$, when $i_1, \dots, i_k$ are distinct. See Example
\ref{exa:uau2}, for a short discussion of why our method
doesn't work more generally.

\begin{example}[$cc^*$ for a circular operator $c$]\label{examp:square_circular}
This is a continuation of Examples \ref{semicircle} and
\ref{circle}. Let ${c}$ be a circular element and $s$ a
semicircular operator as in Examples \ref{semicircle} and
\ref{circle}. Now recall that $cc^*$ and $s^2$ both have a
Poisson distribution with respect to $\phi$,
i.e. $\phi((cc^*)^n)=\phi((s^2)^n)=\frac{1}{n+1}\binom{2n}{n}$. So
one might expect that $cc^*$ and $s^2$ have the same
distribution in the second order level. We see that this is
not the case.

Indeed, the determining sequences of $c$ are given by
$\beta_{1}=1$ and $\beta_{n}=0$ for $n>0$ and
$\beta_{p,q}=0$, for all $p$ and $q$. This already shows our
claim, since $\beta_{n,n}=n$ in the case for semicircle of
Example \ref{semicircle}.  Moreover, from formula
(\ref{formula:main1}) it is readily seen that the second
order cumulants of $cc^*$ are all
zero. i.e. $\kappa_{p,q}(cc^*,cc^*,...,cc^*)=0$.

Since the first order cumulants of $cc^*$ are all $1$. Then
the $(p,q)$-fluctuation moments of $cc^*$ count the number of
elements in $S_{NC}(p,q)$. This means the \begin{eqnarray*}\phi_2((cc^*)^p, (cc^*)^q)=|S_{NC}(p,q)|
= \frac{pq}{2(p+q)}\binom{2p}{p}\binom{2q}{q}.
\end{eqnarray*}The last equality comes from the fact that
$|S_{NC}(p,q)|=\frac{1}{2}|NC_2(2p, 2q)|$ where $NC_2(2p,
2q)$ is the set of non-crossing annular pairings (see
\cite[Cor.~6.7]{mn}). 

One way to interpret the calculation above is that
$$\kappa^{(cc*)}=\zeta \qquad~ \mathrm{and}\qquad
\phi_2((cc^*)^p, (cc^*)^q) =\zeta^{*2}(1_{p+q}, \gamma_{p,q}).
$$

\end{example}

\begin{example}[Products of free circular operators]
\label{ex:Product of two circular operators}
Let $c_1$ and $c_2$ be a pair of second order circular
elements and suppose they are second order free. Let
$h=c_1c_2$. Let us recall some facts from \cite[Lecture
  15]{ns2} about first order case $R$-diagonal operators.

Let $a$ be $R$-diagonal. As we are working in a tracial
non-commutative probability space, we have from
\cite[Prop.~15.6]{ns2} that
\[
\kappa^{(aa^*)}_n = \sum_{\pi \in NC(n)} \beta^{(a)}_\pi
\]
which can be written in terms of multiplicative functions
(c.f. \cite[Lecture 10]{ns2}) as $\kappa^{(aa^*)} =
\beta^{(a)} * \zeta$. Now suppose that $a_1$ and $a_2$ are
$*$-free and $R$-diagonal. We have from
\cite[Ex.~15.24]{ns2} that
\begin{align*}
\beta^{(a_1a_2)}_n &= \sum_{\pi \in NC(n)} \beta^{(a_1)}_\pi
\sum_{\sigma \leq K(\pi)} \beta^{(a_2)}_\sigma = \sum_{\pi
  \in NC(n)} \beta^{(a_1)}_\pi ( \beta^{(a_2)} * \zeta)(0_n,
\Kr(\pi)) \\ & = \beta^{(a_1)} * \beta^{(a_2)} * \zeta( 1_n)
= (\beta^{(a_1)} * \beta^{(a_2)} * \zeta)_n.
\end{align*}
Hence $\beta^{(a_1a_2)} = \beta^{(a_1)} * \beta^{(a_2)} *
\zeta$. Thus $\kappa^{(a_1a_2a_2^*a_1^*)} = \beta^{(a_1)} *
\beta^{(a_2)} * \zeta * \zeta$.

Now let us return to the case that $c_1$ and $c_2$ are
$*$-free and circular.  Since $\kappa_{2n}(c_1, c_1^*,
\dots, c_1, c_1^*) = 0$ for $n \geq 2$ and $\kappa_2(c_1,
c_1^*) = 1$ we have $\beta^{(c_1)}_1 = 1$ and
$\beta^{(c_1)}_n = 0$ for $n \geq 2$. In the language of
multiplicative functions we have $\beta^{(c_1)} =
\delta$. Thus $\kappa^{(c_1c_2c_2^*c_1^*)} = \delta * \delta
* \zeta * \zeta = \zeta^{*2}$.

Now suppose we have $k$ circular operators, $c_1, \dots,
c_k$, which are $*$-free. Let $c = c_1 \cdots c_k$, then
$\kappa^{(cc^*)} = \zeta^{*k}$, which are in turn given by
the Fuss-Catalan numbers (see \cite[Ex.~10.24]{ns2}).

We have that $h = c_1 c_2$ is $R$-diagonal so the only
non-vanishing cumulants of $h$ are those of the form
$\kappa_{2n}(h,h^*,\dots,h,h^*)$. These are equal to $1$ for
all $n$. Let us briefly review how to see this.
\[
\kappa_{2n}(h,h^*,\dots,h,h^*)
=
\sum_{\pi} \kappa_\pi(c_1, c_2, c_2^*, c_1^*, \dots, 
c_1, c_2, c_2^*, c_1^*)
\]
where the sum runs over all $\pi \in NC(2n)$ such that $\pi
\vee \rho = 1_{4n}$ and $\rho = \{ (1, 2), \dots, (4n-1,
4n)\}$. By freeness all blocks of $\pi$ can contain only
$\{c_1, c_1^*\}$ or only $\{c_2, c_2^*\}$. Moreover since
$c_1$ and $c_2$ are circular the blocks of $\pi$ must be of
size $2$ and can only connect $c_i$ to $c_i^*$. Thus $\pi$
must be a pairing. In the diagram below we represent $c_1$,
$c_1^*$, $c_2$, and $c_2^*$ respectively by $1$,
$\overline{1}$, $2$, and $\overline{2}$. Let us start with
the leftmost $1$. If it were to be connected any
$\overline{1}$ other than the rightmost then the condition
$\pi \vee \rho = 1_{2kn}$ would be violated. Similarly if
the leftmost $2$ were connected to any $\overline{2}$ other
than the $\overline{2}$ immediately to the right of $2$ then
the condition $\pi \vee \rho = 1_{2kn}$ would be
violated. Continuing in this way we see that there is only
one such $\pi$ that satisfies the two conditions
\[
\kappa_\pi(c_1, c_2, c_2^*, c_1^*, \dots, 
c_1, c_2, c_2^*, c_1^*) \not= 0
\]
and $\pi \vee \rho = 1_{2kn}$. Moreover since $\kappa_2(c_i,
c_i^*) = \kappa_2(c_i^*, c_i) = 1$, we have that
$$\kappa_{2n}(h,h^*,\dots,h,h^*) = 1$$ as claimed.
\medskip

\hfill
\begin{tikzpicture}[anchor=base, baseline]
\node[above] at (-1.0,1.25) {$\rho$};
\node[above] at (-1.0,0.25) {$\pi$};
\node[above] at (0.00,.75) {$1$};
\node[above] at (0.50,.75) {$2$};
\node[above] at (1.25,.75) {$\overline{2}$};
\node[above] at (1.75,.75) {$\overline{1}$};
\node[above] at (2.50,.75) {$1$};
\node[above] at (3.00,.75) {$2$};
\node[above] at (3.75,.75) {$\overline{2}$};
\node[above] at (4.25,.75) {$\overline{1}$};
\node[above] at (5.00,.75) {$1$};
\node[above] at (5.50,.75) {$2$};
\node[above] at (6.25,.75) {$\overline{2}$};
\node[above] at (6.75,.75) {$\overline{1}$};
\node[above] at (7.50,.75) {$1$};
\node[above] at (8.00,.75) {$2$};
\node[above] at (8.75,.75) {$\overline{2}$};
\node[above] at (9.25,.75) {$\overline{1}$};
\draw  (0.00 , 1.5) -- (0.00, 1.75) -- (0.50, 1.75) -- (0.50, 1.5);
\draw  (1.25 , 1.5) -- (1.25, 1.75) -- (1.75, 1.75) -- (1.75, 1.5);
\draw  (2.50 , 1.5) -- (2.50, 1.75) -- (3.00, 1.75) -- (3.00, 1.5);
\draw  (3.75 , 1.5) -- (3.75, 1.75) -- (4.25, 1.75) -- (4.25, 1.5);
\draw  (5.00 , 1.5) -- (5.00, 1.75) -- (5.50, 1.75) -- (5.50, 1.5);
\draw  (6.25 , 1.5) -- (6.25, 1.75) -- (6.75, 1.75) -- (6.75, 1.5);
\draw  (7.50 , 1.5) -- (7.50, 1.75) -- (8.00, 1.75) -- (8.00, 1.5);
\draw  (8.75 , 1.5) -- (8.75, 1.75) -- (9.25, 1.75) -- (9.25, 1.5);
\draw  (0.00 , 0.75) -- (0.00, 0.25) -- (9.25, 0.25) -- (9.25, 0.75);
\draw  (0.50 , 0.75) -- (0.50, 0.50) -- (1.25, 0.50) -- (1.25, 0.75);
\draw  (1.75 , 0.75) -- (1.75, 0.50) -- (2.50, 0.50) -- (2.50, 0.75);
\draw  (3.00 , 0.75) -- (3.00, 0.50) -- (3.75, 0.50) -- (3.75, 0.75);
\draw  (4.25 , 0.75) -- (4.25, 0.50) -- (5.00, 0.50) -- (5.00, 0.75);
\draw  (5.50 , 0.75) -- (5.50, 0.50) -- (6.25, 0.50) -- (6.25, 0.75);
\draw  (6.75 , 0.75) -- (6.75, 0.50) -- (7.50, 0.50) -- (7.50, 0.75);
\draw  (8.00 , 0.75) -- (8.00, 0.50) -- (8.75, 0.50) -- (8.75, 0.75);
\end{tikzpicture}
\hfill\hbox{}

Since $c_1$ and $c_2$ are second order $R$-diagonal, so is
$h$. In Example \ref{circle} we saw that the second order
cumulants of $c_1$ and $c_2$ are all zero and the only
non-vanishing cumulants of first order are
\begin{equation}\label{eq:circular_cumulants}
\kappa_2(c_1,c_1^*) = \kappa_2(c_1^*,c_1) =
\kappa_2(c_2,c_2^*) = \kappa_2(c_2^*,c_2)=1.
\end{equation}
Let us show that all second order cumulants of $h$ are
0. Since $h$ is second order $R$-diagonal we have that the
only possible non-vanishing cumulants are
\[
\kappa_{2p,2q}(h, h^*, \dots, h, h^*)
\mbox{\ and\ }
\kappa_{2p,2q}(h^*, h, \dots, h^*, h).
\]
We shall show the first of these is 0; the proof for the
second follows, because $h^*$ is also the product of circular
operators. To compute $\kappa_{2p,2q}(h, h^*, \dots, h,
h^*)$ we use Proposition \ref{mst} to write
\[
\kappa_{2p,2q}(h, h^*, \dots, h, h^*) = \kern-1em\mathop{\sum_{\pi
    \in S_{NC}(4p, 4q)}}_{\pi^{-1}\gamma_{4p, 4q}
  \mathrm{\ sep.\ }N} \kern-1em\kappa_\pi (c_1, c_2, c_2^*, c_1^*,
\dots, c_1, c_2, c_2^*, c_1^*),
\]
with $N = \{2,4,6, \dots, 4p + 4q\}$.  Note that there are
no terms involving second order cumulants as the second
order $*$-cumulants of $c_1$ and $c_2$ vanish. For a $\pi
\in S_{NC}(4p, 4q)$ to contribute to this sum, $\pi$ must be
a pairing because the first order $*$-cumulants of $c_1$ and
$c_2$ are all 0 except those in Equation
(\ref{eq:circular_cumulants}). We claim that $\kappa_\pi
(c_1, c_2, c_2^*, c_1^*, \dots, c_1, c_2, c_2^*, c_1^*)\ab =
0$ whenever $\pi^{-1}\gamma_{4p, 4q}$ separates the points
of $N$. Such a $\pi$ must have a through string. Suppose it
connects a $c_1$ on one circle to a $c_1^*$ on the other
circle. This means that there are $j$ and $k$ such that
$\pi(4j+1) = 4k$. But then $\pi^{-1}\gamma_{4p, 4q}(4j) =
4k$, which contradicts the assumption that
$\pi^{-1}\gamma_{4p, 4q}$ sep. $N$. The case where a through
string connects a $c_2$ on one circle to a $c_2^*$ on the
other is similar. Hence $h$ is a second-order $R$-diagonal
with \textit{all} second order cumulants equal to $0$.

We now calculate the second order cumulants of $hh^*$. From
the discussion above we have the determining sequences of
$h$ are given by $\beta_{n}=1$ and $\beta_{p,q}=0$, for all
$p$ and $q$. From Equation (\ref{formula:main1}) we have
that
\begin{eqnarray*}\lefteqn{
\kappa_{p,q}(hh^*, \dots, hh^*) 
=
\sum_{(\mathcal{V},\pi) \in PS_{NC}(p,q)'}\beta_{(\mathcal{V},\pi)} 
+ \sum_{\pi \in S_{NC}(p,q)}\beta_\pi }
\\
& = &
\sum_{\pi \in S_{NC}(p,q)}1=|S_{NC}(p,q)|
= 
\frac{pq}{2(p+q)}\binom{2p}{p}\binom{2q}{q}.
\end{eqnarray*}

One way to interpret the calculation above is that
$$\kappa^{(hh^*)}=\zeta^{*2}$$ and consequently 
 $$\phi_2((hh^*)^p, (hh^*)^q) =\zeta^{*3}(1_{p+q}, \gamma_{p,q})=
\frac{2}{3+1} \frac{p q}{p + q} \binom{3p}{p} \binom{3q}{q}.
$$

The last equality will be explained in further detail for
all powers, $\zeta^{*k}$, in Remark
\ref{remark:bousquet_melou}.

Before going to more general examples. Let us notice that
the determining sequence of $h$ coincides with the cumulants
of $c_1c_1^*$. This is of course, not a coincidence as we
will see in Remark \ref{Remark conjugation}.

\end{example}

\begin{example}[Conjugation by a free circular element] \label{conjugating}
The $c$ be a second order circular operator as in Example
\ref{circle} which is second order $*$-free from $a$. We are
interested in the second order cumulants of $cac^*$.  Recall
that for the first order cumulants we have
\[
\kappa_n(cac^*, \dots, cac^*) = \phi(a^n).
\]

\setbox1=\hbox{%
\begin{tikzpicture}[anchor=base, baseline]
\node[above] at (-0.75,0.25) {$\pi = \mbox{}$};
\node[above] at (0.00,.75) {$c$};
\node[above] at (0.50,.75) {$a$};
\node[above] at (1.00,.75) {$c^*$};
\node[above] at (1.50,.75) {$c$};
\node[above] at (2.00,.75) {$a$};
\node[above] at (2.50,.75) {$c^*$};
\node[above] at (3.00,.75) {$c$};
\node[above] at (3.50,.75) {$a$};
\node[above] at (4.00,.75) {$c^*$};
\node[above] at (4.50,.75) {$c$};
\node[above] at (5.00,.75) {$a$};
\node[above] at (5.50,.75) {$c^*$};
\node[above] at (6.25,.75) {$\cdots$};
\node[above] at (6.75,.75) {$c$};
\node[above] at (7.25,.75) {$a$};
\node[above] at (7.75,.75) {$c^*$};
\node[above] at (8.20,.75) {$c$};
\node[above] at (8.75,.75) {$a$};
\node[above] at (9.25,.75) {$c^*$};
\node[above] at (9.50,.50) {.};
\draw  (0.00 , 0.75) -- (0.00, 0.25) -- (9.25, 0.25) -- (9.25, 0.75);
\draw  (1.00 , 0.75) -- (1.00, 0.50) -- (1.50, 0.50) -- (1.50, 0.75);
\draw  (2.50 , 0.75) -- (2.50, 0.50) -- (3.00, 0.50) -- (3.00, 0.75);
\draw  (4.00 , 0.75) -- (4.00, 0.50) -- (4.50, 0.50) -- (4.50, 0.75);
\draw  (5.50 , 0.75) -- (5.50, 0.50) -- (5.75, 0.50);
\draw  (6.50, 0.50) -- (6.75, 0.50) -- (6.75, 0.75);
\draw  (7.75 , 0.75) -- (7.75, 0.50) -- (8.25, 0.50) -- (8.25, 0.75);
\end{tikzpicture}}

To see this use the formula for cumulants with products for entries
\[
\kappa_n(cac^*, \dots, cac^*) =
\sum_{\pi \in NC(3n)} \kappa_\pi(c, a, c^*, \dots, c , a, c^*)
\]
here the sum runs over all $\pi$ such that $\pi^{-1}\gamma_{3n}$ separates the points on $\{3, 6, 9, \dots, 3n\}$. Recall that $\kappa_2(c, c^*) = \kappa_2(c^*, c) = 1$ and all other cumulants vanish. Thus, for $$\kappa_\pi(c, a, c^*, \dots, c , a, c^*) \not = 0,$$ $\pi$ must pair a $c$ with a $c^*$. There is only one way this can be achieved and also satisfy the requirement that  $\pi^{-1}\gamma_{3n}$ separates the points on $\{3, 6, 9, \dots, 3n\}$:

\noindent\leavevmode\hfill\box1\hfill\hbox{}\break
As this exposes all the $a$'s we get the sum: $\ds\sum_{\pi \in NC(n)} \kappa_\pi(a, \dots, a) = \phi(a^n)$, as claimed. Thus for the cumulant generating function 
\[
C_{cac^*}(z) = \ds\sum_{n=0}^\infty \kappa_n(cac^*, \dots, cac^*) z^n
\] we have
\[ C_{cac^*}(z) = z^{-1}G_a(z^{-1})
\]
where $G_a(z) = \sum_{n=0}^\infty \phi(a^n) z^{-(n+1)}$ is
the Cauchy transform of $a$.

Let us show that a similar formula holds in the second order
case. 

\begin{theorem}\label{thm:free_conjugation_circular}
Let $c_1, \dots, c_k$ be second order circular operators and
suppose that $\{c_1, c_1^*\}$, $\{c_2, c_2^* \}$, \dots,
$\{c_k, c_k^*\}$, and $\{ a\}$ are second order free in a
second order probability space $(\cA, \phi, \phi_2)$. Then
\[
\kappa^{(c_1 c_2 \cdots c_k a c_k^* \cdots\ab c_2^* c_1^*)}
= \kappa^{(a)} * \zeta^{*k}.
\]
\end{theorem}

\begin{proof}

It suffices to prove this for $k = 1$ and use induction; for
this we use the formula for products as arguments directly.
\[
\kappa_{p,q}(cac^*, ...,cac^*)=\sum_{(\cV, \pi) \in
\mathcal{PS}_{NC} (kp,kq)} \kappa_{(\cV, \pi)} (c,a,c^*, ...,,...,cac^*),
\]
where the sum runs over permutations such that
$\pi^{-1}\gamma_{kp,kq}$ separates the points
$\{3,6,..,3(p+q)\}$.

Let us analyze the possible permutations in the sum. Recall
that the only non-vanishing $*$-cumulants of $c$ are
$\kappa_2(c,c^*)$ and $\kappa_2(c^*, c)$. So in $\pi$, any
$c$ must be connected to a $c^*$. For this, in principle
there are two possibilities:

$(i)$ $c$ and $c^*$ are on the same circle, as in Figure
\ref{fig:cac} (left), or

$(ii)$ $c$ an $c^*$ are in opposite circles, as in Figure
\ref{fig:cac} (right).

\begin{figure}[t]
\hfill\includegraphics{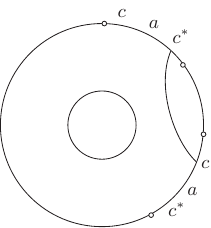}\hfill\includegraphics{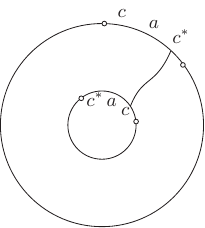}\hfill\hbox{}
\caption{\small{}The two cases (\textit{i}) and (\textit{ii}) below.\label{fig:cac}}
\end{figure}
We shall show that $(ii)$ is not possible and in case $(i)$
the $a$ and $a^*$ must be adjacent as in Example
\ref{ex:Product of two circular operators}.  In either case,
suppose $c^*$ is in position $3i$ and $c$ is in position
$3j+1$, then $3i$ and $3j$ are not separated by
$\pi^{-1}\gamma$. Indeed,
$\pi^{-1}\gamma(3j)=\pi^{-1}(3j+1)=3i$. The only possibility
is then that $j=i$, i.e. each $c^*$ in a position $3i$ is
connected to $c$ in the position $3i+1$, in the same circle;
this excludes case $(ii)$.

Thus, we may write the partitioned permutations $(\cV, \pi)$
in the sum above as $\pi=\pi_1\pi_2$ where
$\pi_1=(3,4)(6,7)\cdots(3p+3q-3,3p+q-2)$ and $\pi_2$ is any
annular non-crossing permutation on
$[2,5,8,...,3p-1;3p,...,3p+3q-1]$ and when $\pi$ has no
through blocks, $\cV$ connects a block of $\pi_2$ on the
outer circle to a block of $\pi_2$ on the inner circle. Thus
the restriction of $(\cV, \pi)$ to
$[2,5,8,...,3p-1;3p,...,3p+3q-1]$ is an arbitrary
partitioned permutation and thus the set of restrictions is
isomorphic to the set $\mathcal{PS}_{NC} (p,q)$. Since
$\kappa_2(c,c^*)=1$ we have $\kappa_\pi(c,c^*, \dots, c,c^*)=1$
and thus
\begin{equation}\label{eq:second_order_cac*}
\kappa_{p,q}(cac^*, ...,cac^*)=\sum_{(\cV, \pi) \in
 \mathcal{PS}_{NC} (p,q)}\kappa_{(\cV, \pi)}(a,...,a)=\phi_2(a^p,a^q).
\end{equation}
Hence
\[
\kappa^{(cac^*)}_{p, q} = \phi_2(a^p, a^q)
=
\sum_{(\cV,\pi) \in \cPS_{NC}(p, q)} \kappa^{(a)}_{(\cV,\pi)}
=
\kappa^{(a)} * \zeta(1_{p+q}, \gamma_{p,q}).
\]
Since this is true for all $p$ and $q$ we have $\kappa^{(cac^*)} = \kappa^{(a)} * \zeta^{*k}$, as claimed. 
\end{proof}

We let 
\[
C_{cac^*}(z, w) = \sum_{m,n \geq 1} \kappa_{m,n}(cac^*, \dots, cac^*) z^{m} w^{n}
\] 
be the generating function for the second order
cumulants. Then by equation (\ref{eq:second_order_cac*}) we
have $$C_{cac^*}(z, w) = z^{-1} w^{-1} G_a(z^{-1}, w^{-1})$$
where
\[
G_a(z, w) = \sum_{m,n \geq 1} \phi_2(a^m, a^n) z^{-(m+1)} w^{-(n + 1)}
\]
is the second order Cauchy transform of $a$.

A particularly important example is the case when the
fluctuation moments of $a$ are 0. In this case $G_a(z,w)=0$
and Equation (\ref{eq:second_order_r_transform}) is reduced to
\[
G_{cac^*}(z,w)= \frac{\partial^2}{\partial z\partial w} \log
\left(\frac{ G_{cac^*}(z)- G_{cac^*}(w)}{z-w}\right).
\]
From the random matrix perspective this corresponds to the
case when $a$ is a limit of deterministic matrices and
$cac^*$ then corresponds to $WAW^*$ where $W$ is a Ginibre
matrix and $A$ is deterministic.  When $A$ is positive semidefinite, this is also called a complex Wishart Matrix with covariance matrix $A$.
\end{example}

\begin{remark} \label{Remark conjugation}
Observe in the last example, that in the case where
$a=rr^*$, the relation, $\kappa^{(crr^*c^*)} = \beta^{(cr)}
*\zeta$ (see Equation \ref{equation:beta_zeta}), between the
determining sequence of $cr$ and the cumulants of $crr^*c^*$
is the same as the one between the cumulants of $rr^*$ and
the cumulants of $crr^*c^*$, $\kappa^{(crr^*c^*)} =
\kappa^{(rr^*)} * \zeta$.  Thus
$\beta^{(cr)}=\kappa^{(rr^*)}$, which in particular explains
the last comment in Example \ref{ex:Product of two circular
  operators}.
\end{remark}

\begin{remark}\label{remark:bousquet_melou}
If in Theorem \ref{thm:free_conjugation_circular} we let $a
= 1$ and $h = c_1 c_2 \cdots c_k a c_k^*\ab \cdots\ab c_2^*
c_1^*$ then we have $\kappa^{(a)} = \kappa^{(1)} = \delta$
and so
\[
\kappa^{(h)} = \delta *\zeta^{*k} = \zeta^{*k}. 
\]
Thus 
\[
\phi_2(h^p, h^q) = \ 
\sum_{\mathclap{(\cV,\pi) \in \cPS_{NC}(p, q)}} \ 
\kappa^{(h)}_{(\cV, \pi)} = \kappa^{(h)}* \zeta(1_{p+q}, \gamma_{p,q}) =
\zeta^{*(k+1)}(1_{p+q}, \gamma_{p,q}).
\]
By the formula of Bousquet-M\'elou and Schaeffer \cite{bms}
(see the discussion in \cite[\S5.17, p. 38]{cmss} for the
interpretation in our current notation) we have that for $l
\geq 2$
\begin{equation}\label{eq:bousquet_melou_schaeffer}
\zeta^{*l}(1_{p+q}, \gamma_{p, q})
= \frac{l-1}{l} \frac{p q}{p+ q} \binom{lp}{p} \binom{lq}{q}.
\end{equation}
Thus in our example, when $h = c_1 c_2 \cdots c_k c_k^*
\cdots\ab c_2^* c_1^*$, we have
\begin{equation}\label{eq:h_moments}
\phi_2(h^p, h^q) 
= \frac{k}{k+1} \frac{p q}{p + q} \binom{(k+1)p}{p} \binom{(k+1)q}{q}. 
\end{equation} 
Note that by Theorem \ref{MT1} we have that $c_1 \cdots c_k$
is $R$-diagonal and the non-vanishing cumulants are given by
\begin{multline*}
\kappa_{p,q}(c_1 \cdots c_k, c_k^* \cdots c_1^*, \dots, c_1
\cdots c_k, c_k^* \cdots c_1^*) = \beta^{(c_1 \cdots
  c_k)}_{p,q} \\ = \kappa^{(c_1 \cdots c_kc_k^* \cdots
  c_1^*)} * \mu(1_{p+q}, \gamma_{p,q}) =
\zeta^{*(k-1)}(1_{p+q}, \gamma_{p,q})
\end{multline*}
which we can easily evaluate by the formula
(\ref{eq:bousquet_melou_schaeffer}) of Bousquet-M\'elou and
Schaeffer.
\end{remark}

\begin{remark}\label{remark:connection_to_dartois_forrester}
Let us show that Equation (\ref{eq:h_moments}) allows us to
complete a claim in \cite[Remark 5]{df} concerning the
second order moments of the product of two Wishart random
matrices with shape parameter 1. In \cite{df} the authors
consider two independent complex Ginibre matrices, $X_1$ and
$X_2$. They let $S_2 = X_1X_1^\dagger X_2 X_2^\dagger$,
where $X^\dagger$ means the conjugate transpose, $X^*$, of
$X$. By writing $X_1$ in its real and imaginary parts we see
( \textit{c.f.} \cite[\S3.2]{ms}) that the second order
joint limit distribution of $\{ X_1, X_2\}$ is that of two
second order $*$-free circular operators.  Let us denote
them by $x_1$ and $x_2$ respectively and set $s_2 = x_1
x_1^* x_2 x_2^*$. Now let $c_1 = x_1^*$ and $c_2 =
x_2$. Then $c_1$ and $c_2$ are two second order $*$-free
circular operators. Let, as in Remark
\ref{remark:bousquet_melou}, $h = c_1 c_2 c_2^* c_1^*$.
Then $s_2^p = c_1^* h^{p-1} c_1c_2 c_2^*$, and so
$\phi_2(h^p, h^q) = \phi_2(s_2^p, s_2^q) = c^{[0]}_{p,q}$
(in the notation of \cite{df}), because $\phi_2$ is tracial
in each of its arguments. Thus
\begin{align*}
c^{[0]}_{p, q} = \phi_2(h^p, h^q)
=
\frac{2}{3+1} \frac{p q}{p + q} \binom{3p}{p} \binom{3q}{q}.
\end{align*}
Note that this is now an easy conclusion of a very general
result. Moreover, the calculations are much simpler that
those required in \cite{df}.
\end{remark}

\begin{example}[Product of $k$ free circular operators] \label{prodpoiss}
Let $h=c_1c_1^*c_2c_2^*\ab\cdots c_kc_k^*$. From Equations
(\ref{11}) and (\ref{12}) of Theorem
\ref{thm:Moments_and_Cumulants_of_Products_of_Free_Variables}
we get a combinatorial description of the fluctuation
moments and cumulants:
$$\phi_2(h^p,h^q)=|\snckalt{kp}{kq}|~\text{and}~
\kappa_{p,q}(h,...,h)=|\snckea{kp}{kq}|. $$ To our best of
our knowledge a precise formula for this quantities was not
known for $k>2$. However, the method of Example
\ref{remark:connection_to_dartois_forrester} can be extended
to compute the fluctuation moments of $c_1^*c_1 c_2^*c_2
\cdots c_n^*c_n$, thereby extending the result of \cite{df}
($n=2$) to the general case and as a byproduct, giving a
formula for the above quantities. First we must relate the
fluctuation moments of $c_1^*c_1 c_2^*c_2 \cdots c_n^*c_n$
to those of $h^*h$, where $h=c_1c_2\cdots c_k$; in fact we
first show a general result that will imply that these two
variables are identically distributed. Here we shall follow
an idea of Kargin \cite{k}, but avoid the use of the
$S$-transform, which doesn't yet exist at the second order
level.

Suppose $(\cA, \phi, \phi_2)$ is a (tracial) second order
$*$-probability space and $x_1,\ab \dots, x_n \in \cA$ are
second order $*$-free and identically $*$-distributed. Let $y =
x_2 \cdots x_n$. Then $x_1$ and $y$ are free. By traciality
we have $x_1x_1^*$ and $x_1^*x_1$ are identically
distributed and, likewise,  $yy^*$ and $y^*y$ are identically
distributed.

Suppose $k \geq 1$ then
\[
\phi\big([(x_1y)^*(x_1y)]^k\big) 
=
\phi\big([(yy^*)(x_1^*x_1)]^k\big).
\]
Now both $yy^*$ and $y^*y$ are free from $x_1^*x_1$. So we
can replace $yy^*$ by $y^*y$ to get.
\[
\phi\big([(yy^*)(x_1^*x_1)]^k\big)
=
\phi\big( [(y^*y)(x_1x_1^*)]^k\big)
=
\phi\big( [(y^*y)(x_1^*x_1)]^k\big)
\]
Thus we have 
\begin{equation}\label{eq:inductive_step}
\phi\big([(x_1y)^*(x_1y)]^k\big)  
= 
\phi\big( [(x_1^*x_1)(y^*y)]^k\big).
\end{equation}
\medskip
When $n = 2$ we have 
\[
\phi\big([(x_1x_2)^*(x_1x_2)]^k\big)  
= 
\phi\big( [(x_1^*x_1)(x_2^*x_2)]^k\big).
\]

\begin{lemma}\label{lemma:first_order_manifesto}
For all $k \geq 1$ we have
\[
\phi\big( [(x_1 \cdots x_n)^*(x_1 \cdots x_n)]^k\big)
=
\phi\big( [ (x_1^*x_1) \cdots (x_n^*x_n)]^k \big).
\]
\end{lemma}

\begin{proof}
We have already checked this when $n = 2$. Suppose it holds
for $n = l-1$ and we let $y = x_2 \cdots x_l$. Then $y^*y$
and $(x_2^*x_2) \cdots (x_l^*x_l)$ are identically
distributed. By Equation (\ref{eq:inductive_step})
\[
\phi\big( [(x_1 \cdots x_l)^*(x_1 \cdots x_l)]^k\big)
=
\phi\big( [(x_1^*x_1)(y^*y)]^k\big).
\]
Since $y^*y$ and $(x_2^*x_2) \cdots (x_l^*x_l)$ are
identically distributed we have
\[
\phi\big( [(x_1^*x_1)(y^*y)]^k\big)
=
\phi\big( [(x_1^*x_1)(x_2^*x_2) \cdots (x_l^*x_l)]^k\big).
\]
With the previous equation we now have 
\[
\phi\big( [(x_1 \cdots x_l)^*(x_1 \cdots x_l)]^k\big)
=
\phi\big( [(x_1^*x_1)(x_2^*x_2) \cdots (x_l^*x_l)]^k\big).
\]
This completes the inductive step. 
\end{proof}

\bigskip
Now we want to repeat this in the second order case.  By
traciality we have $x_1x_1^*$ and $x_1^*x_1$ are identically
distributed at the second order level and, likewise, $yy^*$ and
 $y^*y$ are identically distributed at the second
order level. Let $p, q \geq 1$.
\begin{align*}\lefteqn{
\phi_2\big( [(x_1y)^*(x_1y)]^p, [(x_1y)^*(x_1y)]^q \big) }\\
&=
\phi_2\big( [(yy^*)(x_1^*x_1)]^p, [(yy^*)(x_1^*x_1)]^q \big) \\
&=
\phi_2\big( [(y^*y)(x_1^*x_1)]^p, [(y^*y)(x_1^*x_1)]^q \big) \\
&=
\phi_2\big( [(x_1^*x_1)(y^*y)]^p, [(x_1^*x_1)(y^*y)]^q \big). 
\end{align*}

This shows that $(x_1x_2)^*(x_1x_2)$ and
$(x_1^*x_1)(x_2^*x_2)$ have identical second order
distributions. If, by induction hypothesis, we have $y^*y$
and $(x_2^*x_2) \cdots (x_n^*x_n)$ have identical second
order distributions then, as $x_1$ is second order $*$-free
from $\{ x_2, \dots, x_n\}$, we have that
\[
(x_1 \cdots x_n)^*(x_1 \cdots x_n)
\mathrm{\ and\ }
(x_1^*x_1)(x_2^*x_2) \cdots (x_n^*x_n)
\]
have identical second order distributions. We state this as a lemma.

\begin{lemma}
Suppose $x_1, \dots, x_n$ are second order $*$-free. Then for all $p, q \geq 1$
\[
\phi_2\big( [(x_1 \cdots x_n)^*(x_1 \cdots x_n)]^p, 
            [(x_1 \cdots x_n)^*(x_1 \cdots x_n)]^q)
\]
\[
=
\phi_2\big( [(x_1^*x_1)\cdots(x_n^*x_n)]^p, 
            [(x_1^*x_1)\cdots(x_n^*x_n)]^q).
\]
\end{lemma}
\qed

This now shows that $c_1^*c_1 c_2^*c_2 \cdots c_k^*c_k$ and
$h^*h$ have the same fluctuation moments. As in Example
\ref{remark:bousquet_melou} we have
\[
\phi_2\big( [c_1^*c_1 \cdots c_k^*c_k]^p, [c_1^*c_1 \cdots c_k^*c_k]^q)
=
\phi_2(h^p, h^q) 
=
\zeta^{*(k+1)}(1_{p+q}, \gamma_{p,q})
\]
and by Bousquet-M\'elou Schaeffer \cite{bms} we have 
\begin{multline*}
\phi_2\big( [c_1^*c_1 \cdots c_k^*c_k]^p, [c_1^*c_1 \cdots c_k^*c_k]^q) \\
=
\frac{k}{k+1} \frac{pq}{p+q} \binom{(k+1)p}{p} \binom{(k+1)q}{q}. 
\end{multline*}
\end{example}


We shall point out that Gorin and Sun \cite[Theorem
  4.13]{gs} give an integral formula for the limiting
covariance of the height function $\mathcal{H}_t$ of the
eigenvalues of a product of Ginibre matrices.  It is not
clear how to use such a formula to obtain the combinatorial
formulas above for the fluctuation moments.

\subsection{Counterexamples}
We end with two examples that show that not all properties
lift to the second order level.

First, contrary to the first order case the powers of a
second order $R$-diagonal operator is not $R$-diagonal as
the following example shows.

\begin{example}\label{exa:haar2}[Powers of Haar unitary]
Let $u$ be a second order Haar unitary, then $u^n$ is first
order $R$-diagonal. However, it is not  $R$-diagonal of
second order. 

We will show that
\begin{equation}\label{power:haar}
\kappa_{m,n}(u^p,\dots,u^p,u^{-p},
\dots,u^{-p})=(p-1)n\delta_{m,n}.
\end{equation}
This shows that $u^p$ is not second order $R$-diagonal, for
$p\neq1$, since there are non-alternating cumulants which do
not vanish.

To prove \eqref{power:haar} we use the moment-cumulant
formula and the fact that $u^p$ is a first order Haar
unitary. Recall that
\begin{eqnarray*}\delta_{m,n}np
& = &
\phi_2(u^{mp},u^{-np})
= 
\sum_{\pi\in S_{NC}(m,n)}
\kappa_{\pi}(u^p, \dots, u^p ,u^{-p}, \dots, u^{-p})        \\
& & \mbox{} + \ 
\sum_{\mathclap{(\cV,\pi)\in P_{NC}(m,n)'}}\ 
\kappa_{(\cV,\pi)}(u^p, \dots, u^p, u^{-p}, \dots, u^{-p}). 
\end{eqnarray*}

Notice that the first sum does not depend on $p$ since we
only consider first order cumulants. Let us call this
quantity $\Sigma_1(m,n)$, so that

\begin{eqnarray*}\delta_{m,n}np&=&\Sigma_1(m,n)
\ + \ 
\sum_{\mathclap{(\cV,\pi)\in P_{NC}(m,n)'}} \ 
\kappa_{(\cV,\pi)}(u^p,\dots,u^p,u^{-p},\dots,u^{-p}) 
 \end{eqnarray*}

For the second sum, notice that if
$(\cV,\pi)\neq(1_{m,m},\gamma_{m,m})$
then $$\kappa_{(\cV,\pi)}(u^p,\dots,u^p,u^{-p},\dots,u^{-p})=0.$$
Indeed, at least one block of $\cV$, say $B$, is
contained in either $[m]$ or $[m+1, m+n]$, and for this
block either $\kappa_{|B|}(u^p,...,u^p)=0$ or, respectively,
$\kappa_{|B|}(u^{-p},...,u^{-p})=0$, since $u^p$ is a first
order Haar unitary.

This means that 
\begin{eqnarray}\label{nice!!!}
\kappa_{m,n}(u^p,\dots,u^p,u^{-p},\dots,u^{-p})=np\delta_{m,n}-\Sigma_1(m,n)
\end{eqnarray}

It remains to calculate $\Sigma_1(m,n)$, which only depends
on $m$ and $n$, and not on $p$. In particular, we may take
$p=1$ in \eqref{nice!!!} to obtain
\[
0 = \kappa_{m,n}(u,\dots,u,u^{-1},\dots,u^{-1}) =
n\delta_{m,n} - \Sigma_1(m,n),
\]
from which $\Sigma_1(m,n)= \delta_{m,n}n$, and then
\eqref{power:haar} follows.

More generally we may consider a sequence of exponents
$(\epsilon, \theta)=(\epsilon_1,\dots,\epsilon_m,
\theta_1,\dots,\theta_n)$ with $\epsilon_i, \theta_j \in
\{-1, 1\}$.  We are interested in the second order cumulant
\[
\kappa_{m,n}^{(\epsilon,\theta,p)}:=\kappa_{m,n}(u^{p
  \epsilon_1},\dots, u^{p \epsilon_m} ,u^{p \theta_1},\dots,
u^{p \theta_n}).
\]
We will show in Equation (\ref{eq:linear_reduction}) that we
may write $\kappa_{m,n}^{(\epsilon,\theta,p)}$ in terms of
$p$, $\kappa_{m,n}^{(\epsilon,\theta,1)}$ and
$\kappa_{m,n}^{(\epsilon,\theta,2)}$.

We prove by induction on $m+n$ that there are functions
$a_{(\cV, \pi)}(\epsilon, \theta)$ and $b_{(\cV,
  \pi)}(\epsilon, \theta)$ such that for all $(\cV, \pi) \in
\cPS_{NC}(m, n)$ and all $p$ we have
\begin{multline*}
\kappa_{(\cV, \pi)}(u^{p\epsilon_1}, u^{p\epsilon_2}, \dots,
u^{p\epsilon_m}, u^{p\theta_1}, u^{p\theta_2}, \dots,
u^{p\theta_n}) \\ = a_{(\cV, \pi)}(\epsilon, \theta) p +
b_{(\cV, \pi)}(\epsilon, \theta).
\end{multline*}

When $m = n = 1$ we set $\cV = \{(1, 2)\}$ and $\pi=
(1)(2)$. For $\epsilon, \theta \in \{-1, 1\}$ we set
\[
a_{(\cV, \pi)}(\epsilon, \theta) =
\begin{cases} 1 & \epsilon + \theta =0 \\
              0 & \epsilon + \theta \not=0
\end{cases}
\mbox{\ and\ }
b_{(\cV, \pi)}(\epsilon, \theta) =
\begin{cases}-1 & \epsilon + \theta =0 \\
              0 & \epsilon + \theta \not=0
\end{cases}.
\]
By Equation (\ref{power:haar}) we have
$\kappa_{1,1}(u^{p\epsilon}, u^{p\theta}) = (p-1)
\delta_{\epsilon, -\theta} = a_{(\cV, \pi)}(\epsilon,
\theta) p + b_{(\cV, \pi)}(\epsilon, \theta)$. This starts
the induction.

To continue,  we write $\epsilon = \epsilon_1 + \epsilon_2 
+ \cdots + \epsilon_m$; then $u^{p\epsilon_1}u^{p\epsilon_2} 
\cdots u^{p\epsilon_m} = u^{p\epsilon}$. Likewise we set $\theta 
= \theta_1 + \theta_2 + \cdots + \theta_n$; then $u^{p\theta_1}u^{p\theta_2} 
\cdots u^{p\theta_m} = u^{p\theta}$. By the 
moment-cumulant formula
\begin{eqnarray*}\lefteqn{%
    p\phi_2(u^{\epsilon} ,u^{\theta}) =
    \phi_2(u^{p \epsilon} ,u^{p\theta})} \\
  & = &
\sum_{\pi\in S_{NC}(m,n)} \kappa_\pi (u^{p
  \epsilon_1},\dots, u^{p \epsilon_m} ,u^{p \theta_1},\dots,
u^{p \theta_n}) \\
& & \mbox{}+\sum_{(\cV,\pi)\in
    P_{NC}(m,n)'}\kappa_{(\cV,\pi)}(u^{p \epsilon_1},\dots,
  u^{p \epsilon_m} ,u^{p \theta_1},\dots, u^{p \theta_n}).
\end{eqnarray*}

The sum $\Sigma^{(\epsilon, \theta)}:=\sum_{\pi\in
  S_{NC}(m,n)} \kappa_{\pi}(u^{p\epsilon_1} ,\dots, u^{p\theta_n})$
only depends on $\epsilon$ and $\theta$. On the other hand,
suppose $(\cV,\pi)\neq(1_{m+n},\gamma_{m,n})$. Write $\pi =
\pi_1 \times \pi_2 \in NC(m) \times NC(n)$. Let $B$ be the
block of $\cV$ that is the union of a cycle of $\pi_1$ (with
$k$ elements) and a cycle of $\pi_2$ (with $l$ elements).

Then there are $\eta_1, \dots, \eta_{k+l} \in \{-1, 1\}$ and
$\zeta^{(1)}_1, \dots, \zeta^{(1)}_{r_1}, \dots,
\zeta^{(s)}_1, \dots,\ab \zeta^{(s)}_{r_s} \in \{-1, 1\}$
such that
\begin{multline*}
\kappa_{(\cV,\pi)}(u^{p \epsilon_1}, \dots, u^{p
  \epsilon_m}, u^{p \theta_1}, \dots, u^{p \theta_n}) \\ =
\kappa_{k,l}(u^{p \eta_1}, \dots, u^{p \eta_{k+l}} ) \times
\prod^s_{i=1}\kappa_{r_i}(u^{p\zeta^{(i)}_{1}}, \dots,
u^{p\zeta^{(i)}_{r_i}}),
\end{multline*}
Thus, by induction this cumulant is a product of a linear
function of $p$, $ \kappa_{k,l}(u^{p \eta_1}, \dots, u^{p
  \eta_{k+l}} ) $, and terms that do not depend on $p$,
namely, cumulants of first order. Let us denote this linear
function by $\tilde{a}_{(\cV,\pi)} (\epsilon,\theta)
p+\tilde{b}_{(\cV,\pi)} (\epsilon,\theta).$ Putting this all
together we arrive at the identity
\begin{eqnarray*}\lefteqn{
\kappa_{m,n}(u^{p\epsilon_1} , \dots, u^{p \theta_n}) 
=
p\phi_2(u^{\epsilon} ,u^{\theta}) + \Sigma^{(\epsilon, \theta)} }\\
&& \mbox{} + \kern-1em
\sum_{\substack{(\cV,\pi)\in \cPS_{NC}(m,n)' \\ (\cV,\pi)\neq(1_{m+n},\gamma_{m,n})}} \kern-1em
(\tilde{a}_{(\cV,\pi)}(\epsilon, \theta)p+\tilde{b}_{(\cV,\pi)}(\epsilon, \theta))                    \\
& = &
p \Big(\phi_2(u^{\epsilon} ,u^{\theta})
+ \kern-1em
\sum_{\substack{(\cV,\pi)\in \cPS_{NC}(m,n)' \\ (\cV,\pi)\neq(1_{m+n},\gamma_{m,n})}}
\tilde{a}_{(\cV,\pi)} (\epsilon, \theta)\Big)  
+ \Sigma^{(\epsilon, \theta)} \\
&& \mbox{} + 
\sum_{\substack{(\cV,\pi)\in \cPS_{NC}(m,n)' \\ (\cV,\pi)\neq(1_{m+n},\gamma_{m,n})}}
\tilde{b}_{(\cV,\pi)}(\epsilon, \theta).
\end{eqnarray*}
which proves the inductive step. By substituting
$\kappa_{m,n}^{(\epsilon,\theta,2)}$ and
$\kappa_{m,n}^{(\epsilon,\theta,1)}$ the formula above we
see that
\begin{equation}\label{eq:linear_reduction}
\kappa_{m,n}^{(\epsilon,\theta,p)} =  p \left(\kappa_{m,n}^{(\epsilon,\theta,2)}-\kappa_{m,n}^{(\epsilon,\theta,1)}\right) +\left (2\kappa_{m,n}^{(\epsilon,\theta,1)}-\kappa_{m,n}^{(\epsilon,\theta,2)}\right)
\end{equation}
as claimed.\qed\end{example}

\begin{figure}[t] 
\hfill\includegraphics[]{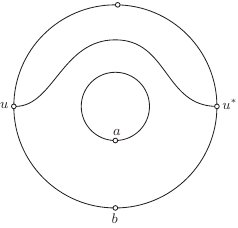}\hfill\includegraphics[]{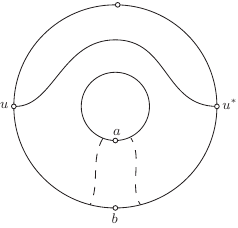}\hfill\hbox{}
\caption{\small Diagrams in Example \ref{exa:uau2} with
  non-zero contribution. \label{fig:uau2}}
\end{figure}

Finally, in first order free probability, it was shown in
\cite{ns1} that if $u$ is a Haar unitary free from $a$, then
$ubu^*$ is free from $a$, for any $b$ free from $u$.  We now
prove that, for $u$ and $a$ second order free, in general,
$ubu^*$ is not second order free from $a$,

\begin{example}\label{exa:uau2} 
Suppose that $a$ and $b$ are operators in a second order
probability space, and $u$ is a second order Haar unitary
which is second order free from $a$ and $b$.  Then, from the
formula (\ref{MST}) of Proposition \ref{mst} for cumulants
with products as arguments we have
$$\kappa_{1,1}(ubu^*,a) =\sum_{(\cV,\pi) \in
  \cPS_{NC}(3,1)}\kappa_{(\cV,\pi)}(u,b,u^*,a),$$ where the
sum is over partitions that separate the points
$\{3,4\}$. The only partitions satisfying that $3$ and $4$
are in different cycles as shown in Fig \ref{fig:uau2}. Thus
\[
\kappa_{1,1}(ubu^*,a)=\kappa_2(u, u^*) \kappa_2(b,a) + \kappa_2(u,u^*)
\kappa_{1,1}(b,a) = \kappa_2(u,u^*) \phi_2(b,a)
\]
Thus, whenever $\phi_2(b,a) \not= 0$, (in particular
whenever $a$ and $b$ are \textit{not} second order free),
$uau^*$ and $b$ are not second order free since
$\kappa_{1,1}(ubu^*,a)$ does not vanish. Note that this does
not contradict \cite[Thm. 2.9]{cmss} because with the
hypotheses there we get the second order freeness of $a$ and
$b$.
\end{example}

\section{Acknowledgements}
We are very grateful to the referee for many useful suggestions that clarified the results in the paper. 


\thebottomline
\begin{thebibliography}{XX}

\bibitem {ariz} O.~Arizmendi, Statistics of blocks in
  $k$-divisible non-crossing partitions. \emph{Electron.
  J. Combin.} \textbf{19}, 2 (2012), 47pp.

\bibitem{ahs} O.~Arizmendi, T.~Hasebe, N.~Sakuma, On
  the law of free subordinators, \textit{ALEA
    Lat. Am. J. Probab. Math. Stat.} \textbf{10} (2013),
  271-291.


\bibitem{av} O.~Arizmendi and C.~Vargas, Products of
  free random variables and $k$-divisible non-crossing
  partitions, \textit{Electron. Commun. Probab.} \textbf{17}
  (2012), no. 11, 13 pp.

\bibitem{bs} Z.~Bai and J.~Silverstein, \textit{Spectral
  analysis of large dimensional random matrices}, 2$^{nd}$
  ed., Springer, 2010.

\bibitem{BP}\label{BP99} H. Bercovici and V. Pata,
  Stable laws and domains of attraction in free probability
  theory, \textit{Ann. of Math.}, \textbf{149} (1999),
  1023-1060.

\bibitem{b} P.~Biane, \newblock Some properties of
  crossings and partitions, \newblock {\em Discrete Math.},
  175(1-3):41--53, 1997.



\bibitem{bcgls} G.~Borot, S.~Charbonnier,
  E.~Garcia-Failde, F.~Leid, and S.~Shadrin, Analytic theory
  of higher order free cumulants, \textit{arXiv}.2112.12184.

\bibitem{bms} M.~Bousquet-M\'elou and G.~Schaeffer,
  Enumeration of planar constellations, {\em Adv. in
    Appl. Math.}, 24(4):337--368, 2000.


\bibitem{df}S.~Dartois and P.~J.~Forrester,
  Schwinger-Dyson and loop equations for a product of square
  Ginibre random matrices, \textit{J.~Phy.~A} \textbf{53},
  (2020), 175201, 35pp.

\bibitem{dp} G.~Dubach and Y.~Peled, On words of
  non-Hermitian random Matrices, \textit{Ann. of Probab.}
  \textbf{49} (2021), 1886-1916.

\bibitem{dm} M.~Diaz and J.~A.~Mingo, On the Analytic
  Structure of Second-Order Non-Commutative Probability
  Spaces and Functions of Bounded Fr\'echet Variation,
  \textit{arXiv}{\sl:2201.04112}.

\bibitem{cmss} B.~Collins, J.~A.~Mingo, P.~\'Sniady,
  and R.~Speicher, Second Order Freeness and Fluctuations of
  Random Matrices: III. Higher order Freeness and Free
  Cumulants, \textit{Documenta Math.} \textbf{12} (2007),
  1-70.

\bibitem{f} P.~J.~Forrester, A review of exact results
  for fluctuation formulas in random matrix theory,
  \textit{Probab. Surv.} \textbf{20} (2023), 170–225.


\bibitem{gs} V.~Gorin and Y.~Sun, Gaussian fluctuations
  for products of random matrices, \textit{Amer.~J.~Math.}
  \textbf{144} (2022), 287-393.

\bibitem {gkz} A.~Guionnet, M.~ Krishnapur and
  O~Zeitouni, The single ring theorem. \textit{Ann. of
    Math.} (2011), 1189-1217.

\bibitem{hs} U.~Haagerup and H.~Schultz, Invariant
  subspaces for operators in a general II$_1$-factor,
  \textit{Publ. Math. de l'IHES}, \textbf{109} (2009),
  19-111.

\bibitem{hws} N.~S.~Hawley and M.~Schiffer, Half-order
  differentials on Riemann surfaces, \textit{Acta Math.}
  \textbf{115} (1966), 199-236.

\bibitem{k} V.~Kargin, The norm of products of free
  random variables, \textit{Probab. Theory Relat. Fields},
  \textbf{13} (2007), 397-413.

\bibitem{ks} B.~Krawczyk and R.~Speicher, Combinatorics
  of Free Cumulants, \textit{J. Combin. Theory A},
  \textbf{90} (2000), 267-292.

\bibitem{l} G.~Lambert, Limit theorems for biorthogonal
  ensembles and related combinatorial identities,
  \textit{Adv. in Math.} \textbf{329} (2018), 590-648.

\bibitem{mn} J.~A.~Mingo and A.~Nica, Annular
  non-crossing permutations and partitions, and second-order
  asymptotics for random matrices,
  \textit{Inter. Math. Res. Not. IMRN}, \textbf{28} (2004),
  1413 - 1460.

\bibitem{mss} J.~A.~Mingo, P.~\'Sniady, and
  R.~Speicher, Second order freeness and fluctuations of
  random matrices: {II.} {U}nitary random matrices,
  \textit{Adv. in Math.} \textbf{209} (2007), 212-240.


\bibitem{ms} J.~A.~Mingo and R.~Speicher, \newblock
  {Second Order Freeness and Fluctuations of Random
    Matrices: I.  Gaussian and Wishart matrices and Cyclic
    Fock spaces}, \newblock{\em J. Funct. Anal.}, 235, 2006,
  pp. 226-270.

\bibitem{ms2} J.~A.~Mingo and R.~Speicher, \textit{%
  Free Probability and Random Matrices}, Fields Institute 
  Monographs, vol. 35, Springer, 2017.
  
\bibitem{mst} J.~A.~Mingo, R.~Speicher, and E.~Tan,
  Second Order Cumulants of Products,
  \textit{Trans. Amer. Math. Soc.}, \textbf{361}, (2009),
  4751-4781.

\bibitem{ns1} A.~Nica and R.~Speicher, $R$-diagonal
  pairs -- a common approach to Haar unitaries and and
  circular elements, in \textit{Free Probability Theory},
  D.-V.~Voiculescu ed., Fields Institute Communications,
  \textbf{12} (1997), 149-188.

\bibitem{ns2} A.~Nica and R.~Speicher,
  \textit{Lectures on the Combinatorics of Free
    Probability}, LMS Lecture Note Series, \textbf{335},
  Cambridge University Press, 2006.

\bibitem{v} D.-V. Voiculescu: Limit laws for random
  matrices and free products, \textit{Invent. Math.}
  \textbf{104} (1991), 201 - 220.

\bibitem{vdn} D.-V. Voiculescu, K. J. Dykema, A. Nica,
  \textit{Free random variables. A non-commutative
    probability approach to free products with applications
    to random matrices, operator algebras and harmonic
    analysis on free groups}, CRM Monograph Series, vol. 1,
  Amer. Math. Soc., 1992.

\bibitem{zwsmh} Z.~Zheng, L.~Wei, R.~Speicher,
  R.~M\"uller, J.~H\"am\"al\"ainen, and J.~Corander,
  Asymptotic analysis of Rayleigh product channels: a free
  probability approach, \textit{IEEE Trans. Inform. Theory}
  \textbf{63} (2017), 1731-1745.



\end{thebibliography}
\end{document}